\documentclass[preprint]{article}

\usepackage{geometry}

\usepackage{amsthm}
\usepackage{amsfonts}
\usepackage{graphicx}
\usepackage{epstopdf}
\usepackage{dsfont}
\usepackage{bm}
\usepackage{bbm}
\usepackage{tikz}
\usepackage{pgfplots}
\usetikzlibrary{patterns}
\usepackage{ulem}
\usetikzlibrary{shapes}
\usetikzlibrary{plotmarks}

\usepackage{tikz}
\usepackage{pgfplots}
\usetikzlibrary{shapes}
\usepackage{tikz-3dplot}
\usetikzlibrary{decorations.pathreplacing,angles,quotes,patterns}
\usetikzlibrary{decorations.pathmorphing}
\tikzset{snake it/.style={decorate, decoration=snake}}
\usetikzlibrary{positioning}
\usetikzlibrary{patterns}
\usepackage{subfigure}
\usetikzlibrary{patterns,positioning}
\usetikzlibrary{intersections}
\usetikzlibrary{backgrounds,patterns}
\usetikzlibrary{tikzmark,calc,fadings,arrows,shapes,decorations.pathreplacing}
\usetikzlibrary{shapes,positioning,shadings}
\usepgfplotslibrary{fillbetween}

\usepackage{caption}
\usepackage{chngcntr}
\counterwithout{figure}{section}

\usepackage{color}
\usepackage{todonotes}

\usepackage{amsopn}
\RequirePackage{amsmath}
\RequirePackage{fix-cm}

\usepackage{graphicx}
\usepackage{subfigure}
\usepackage{epstopdf}
\usepackage{indentfirst}
\usepackage{color}
\usepackage{bm}
\usepackage{tikz}
\usetikzlibrary{calc}
\usepackage{ifthen}

\usepackage{ulem}

\numberwithin{equation}{section} 

\newcommand\ds{\displaystyle}

\def\O{\Omega}
\def\o{\omega}
\def\G1{{\bf 1}}
\def\r{{\mathfrak{r}}}

\def\S{{\GS}}

\def\sc{\longrightarrow}

\def\e{\varepsilon}

\def\fv{\frak{v}}
\def\fu{\frak{u}}

\def\C{\mathcal{C}}

\def\Te{{\cal T}_\varepsilon}

\def\Ga{{\bf a}}

\def\Ge{{\bf e}}

\def\ds{\displaystyle}

\def\N{{\mathbb{N}}}

\def\R{{\mathbb{R}}}
\def\Z{{\mathbb{Z}}}

\def\D{{\mathbb{D}}}
\def\R{{\mathbb{R}}}

\def\N{{\mathbb{N}}}
\def\Z{{\mathbb{Z}}}

\def\Ec{{\mathcal{E}}}

\def\Uc{{\mathcal{U}}}

\def\Zc{{\mathcal{Z}}}
\def\Vc{{\mathcal{V}}}

\def\Yc{{\mathcal{Y}}}

\def\Wc{{\mathcal{W}}}

\def\Mc{{\mathcal{M}}}

\def\wh{\widehat }

\def\wt{\widetilde }
\def\X{\times }

\def\Ga{{\bf a}}

\def\Ge{{\bf e}}

\def\Gu{{\bf u}}
\def\Gv{{\bf v}}
\def\Gw{{\bf w}}

\def\GB{{\bf B}}
\def\GC{{\bf C}}

\def\GF{{\bf F}}

\def\GH{{\bf H}}
\def\GI{{\bf I}}
\def\GJ{{\bf J}}

\def\GM{{\bf M}}

\def\GQ{{\bf Q}}
\def\GR{{\bf R}}
\def\GS{{\bf S}}

\def\GF{{\bf F}}
\def\GV{{\bf V}}

\definecolor{skyblue}{RGB}{135,206,235}
\definecolor{deepskyblue}{RGB}{0, 191, 255}

\usepackage{amsfonts}
\usepackage{graphicx}
\usepackage{epstopdf}
\usepackage{dsfont}
\usepackage{bm}
\usepackage{bbm}
\usepackage{algorithmic}
\usepackage{tikz}
\usepackage{pgfplots}
\usetikzlibrary{patterns}

\usepackage{nicefrac}
\usepackage{mathtools}

\usepackage{ulem}

\usepackage{url, breakurl}
\usepackage{charter} % for \FloatBarrier
\usepackage{longtable}
\usepackage{float}
\usepackage{supertabular}
\usepackage{graphicx}
\usepackage{color}
\usepackage{subfigure}
\usepackage{overpic}
\usepackage{textcomp}
\usepackage[section]{placeins}

\usepackage{enumitem}

\usepackage{siunitx}
\newtheorem{theorem}{Theorem}
\newtheorem{definition}{Definition}
\newtheorem{corollary}{Corollary}
\newtheorem{proposition}{Proposition}
\newtheorem{assumption}{Assumption}

\newtheorem{remark}{Remark}

\newtheorem{lemma}{Lemma}

\begin{document}

%\begin{frontmatter}

%\pretitle{}
\title{Dimension reduction and homogenization of composite plate with matrix pre-strain}
%\runtitle{Simultaneous homogenization and dimension reduction}
%\subtitle{}

% For one author:
%\author{\inits{N.}\fnms{Name1} \snm{Surname1}\ead[label=e1]{first@somewhere.com}}
%\address{Department first, \orgname{University or Company name},
%Abbreviate US states, \cny{Country}\printead[presep={\\}]{e1}}

% Two or more authors:
%\begin{aug}
%
%\end{aug}
\author{Amartya Chakrabortty, Georges Griso, Julia Orlik}
\date{October 5, 2023}

\maketitle

{\bf Mathematics Subject Classification (2010):} 35B27, 35J86, 35C20, 74K10, 74F10, 76M30, 76M45.
{\let\thefootnote\relax\footnotetext{
	Amartya Chakrabortty, Julia Orlik: SMS, Fraunhofer ITWM, Kaiserslautern 67663, Germany
	
	Emails: amartya.chakrabortty@itwm.fraunhofer.de, julia.orlik@itwm.fraunhofer.de\\
	
	Georges Griso: Laboratoire Jacques-Louis Lions (LJLL), Sorbonne Universit\'e, CNRS, Universit\'e de Paris, F-75005 Paris, France
	
	Email: griso@ljll.math.upmc.fr\\
	
	Correponding author Email: amartya.chakrabortty@itwm.fraunhofer.de}}

%\address[A]{SMS, \orgname{Fraunhofer ITWM},
%	Kaiserslautern 67663, \cny{Germany}\printead[presep={\\}]{e1,e3}}
%\address[B]{Laboratoire Jacques-Louis Lions (LJLL), \orgname{Sorbonne Universit\'e, CNRS, Universit\'e de
%		Paris},
%	F-75005 Paris, \cny{France}\printead[presep={\\}]{e2}}
\begin{abstract}
		This paper focuses on the simultaneous homogenization and dimension reduction of periodic composite plates within the framework of non-linear elasticity. The composite plate in its reference (undeformed) configuration consists of a periodic perforated plate made of stiff material with holes filled by soft matrix material. The structure is clamped on a cylindrical part. Two cases of asymptotic analysis are considered: one without pre-strain and the other with matrix pre-strain. In both cases, the total elastic energy is in the von-K\'arm\'an (vK) regime ($\varepsilon^5$).
		A new splitting of the displacements is introduced to analyze the asymptotic behavior. The displacements are decomposed using the Kirchhoff-Love (KL) plate displacement decomposition. The use of a re-scaling unfolding operator allows for deriving the asymptotic behavior of the Green St. Venant's strain tensor in terms of displacements. The limit homogenized energy is shown to be of vK type with linear elastic cell problems, established using the $\Gamma$-convergence.
		Additionally, it is shown that for isotropic homogenized material, our limit vK plate is orthotropic. The derived results have practical applications in the design and analysis of composite structures.
\end{abstract}

%\begin{keyword}
%\kwd{Homogenization}
%\kwd{dimension reduction}
%\kwd{unfolding operators}
%\kwd{$\Gamma$-convergence}
%\kwd{non-linear elasticity}
%\kwd{von-K\'arm\'an plate}
%\kwd{pre-strain}
%\end{keyword}

%\end{frontmatter}

%%%%%%%%%%% The article body starts:

	% Your content here

	\section{Introduction}
Nowadays, the use of 3D printing technology allows for the deployment or folding of pre-stretched membranes into desired shapes (e.g., see \cite{gries}). In this manufacturing process, the membrane is initially extended in specific directions, and then a pattern is printed on it. An interesting technological challenge is to determine the optimal pre-strain of the membrane to achieve the desired shape of the homogenized composite shell. The key aspect of this challenge is the transfer of pre-strain from the soft membrane to the stiff part of the homogenized composite shell. This motivation serves as the basis for studying this type of problem.

In practical applications, the material of the membrane is typically softer than the printed plastics. However, the findings of this paper indicate that a soft membrane alone does not transfer sufficient stress to the stiff part of the homogenized composite shell. This transfer of stress is necessary for folding the shell, particularly when the printed pattern consists of thick beams. In the second case presented in the paper, it is assumed that there is pre-strain in the soft membrane part of the periodic plate, which provide the desired effect.\\
	
	This paper presents a theoretical study of homogenization and dimension reduction of composite structures in the framework of non-linear elasticity. The structure under consideration consists of a periodic perforated plate made of thick straight beams that intersect each other in the two in-plane directions creating a periodic structural frame ($\O_\e^B$), with soft matrix ($\O_\e^M$) filling the holes; see Figure \ref{Fig01}. The resulting heterogeneous composite plate is periodically distributed in cells isometric to the domain $(0,\e)^2\X(-r,r)$, where  both $\e$ and $r$  go to zero. It should be noted that in this paper, the authors have considered the thickness $r$ and the periodicity $\e$ to be related by the relation $r=\kappa\e$, where $\kappa$ is strictly positive constant\footnote{When the thickness $r$ and periodicity $\varepsilon$ are of the same order, it is referred to as thick beams, however, when $r$ is of a smaller order than $\varepsilon$, it is called thin beams.}. Although the analysis begin with a composite plate of specific geometry, the derived results do not depend on this particular geometry and it is extended to more general class of plate composed of a connected periodic perforated structure with the holes filled by soft matrix in Section \ref{Sec09}.\\

	For the purpose of analysis, it is assumed that the composite structure in the undeformed configuration is made of an elastic material. To simplify the analysis, it is assumed that the composite plate is made of a heterogeneous elastic material, with the matrix part $\Omega_\e^M$ relatively softer by an order of $\e^2$ compared to the structural frame part $\Omega_\e^B$. This assumption allows us to consider a composite plate made of the same material but with different levels of stiffness. A classical example of this type of elastic body is the St. Venant's Kirchhoff material (see Remark \ref{Rem2} for details) with Lam\'e's constants $(\lambda, \mu)$ and $(\e^2\lambda, \e^2\mu)$ in the periodic perforated part $\Omega_\e^B$ and matrix part $\Omega_\e^M$ respectively.
	 
 The limit behavior of the composite structure depend on the order of the infimum of the elastic energy with respect to the parameter $\e$, which in our case is in von-K\'arm\'an (vK) regime. In this paper, two cases are handled :
	\begin{enumerate}
		\item  The applied forces are scaled in such a way that infimum of the total elastic energy is in vK regime, i.e.
		$$ \begin{aligned}
			\GJ_\e(v)=&\int_{\O_\e}\wh{W}_\e(x,\nabla v)dx-\int_{\O_\e} f_\e\cdot(v-\text{\bf id})dx, \quad \hbox{with $f_\e$ is the re-scaled force from $f\in L^2(\o)^3$},\\
			&\hskip 30mm \text{and}\quad\int_{\O_\e}\wh{W}_\e(x,\nabla v)dx\sim O(\e^5).
		\end{aligned}$$
		\item  It is assumed that, there is pre-strain in the soft matrix part ($\O_\e^M$) and the applied forces are scaled in such a way that infimum of the total elastic energy is in vK regime, i.e.
		$$ \begin{aligned}
			\GJ^*_\e(v)=&\int_{\O_\e}\wh{W}_\e(x,(\nabla v)A^{-1}_\e)dx-\int_{\O_\e} f_\e\cdot(v-\text{\bf id})dx,\quad\hbox{with}\;\;
			A_\e=\left\{\	\begin{aligned}
				&\GI_3-\e B_\e\;\;&&\text{in}\;\;\O_\e^M,\\
				&\GI_3\;\;&&\text{in}\;\;\O_\e^B,
			\end{aligned}\right.\\
			&\hskip 27mm \text{and}\quad \int_{\O_\e}\wh{W}_\e(x,(\nabla v)A^{-1}_\e)dx \sim O(\e^{5}).
		\end{aligned}$$
	\end{enumerate}
	
	As far as a minimizing sequence of deformations $\{v_\e\}_\e$ of the energy is concerned, this leads to the following estimates of the Green St. Venant's strain tensors for the both cases
	$$ \begin{aligned}
		\|\nabla v_\e (\nabla v_\e)^T-\GI_3\|^2_{L^2(\O^B_\e)}+\e^2\|(\nabla v_\e)^T\nabla v_\e -\GI_3\|^2_{L^2(\O^M_\e)}&\sim O(\e^{5}),\\
		\|\nabla v_\e (\nabla v_\e)^T-\GI_3\|^2_{L^2(\O^B_\e)}+\e^2\|(\nabla v_\e A_\e^{-1})^T\nabla v_\e A_\e^{-1} -\GI_3\|^2_{L^2(\O^M_\e)}&\sim O(\e^{5}).
	\end{aligned}	$$
	
	The limit homogenized model for the composite structure is described by the vK plate model. This model emerges as the $\Gamma$-convergence limit in terms of displacements (see Remark \ref{Re7} for more details) from non-linear problems stated in terms of deformations, highlighting the importance of considering initially deformations (see Remark \ref{Re1} for more details) in studying the limit behavior. Despite the non-linear nature of the initial and homogenized problems, the cell problems associated with the vK plate are linear. In the Case 1, the homogenized limit energy has energy only from the stiff material part i.e. the limit energy terms from the matrix part vanishes due to the linear elastic cell problems. For Case 2, the homogenized energy has additional term due to the presence of pre-strain in the matrix part. With these cell problems and using the symmetry of the structure it is shown that the homogenized vK plate is orthotropic for isotropic homogeneous materials, which is consistent with the vK energy and linear elasticity formulations. This orthotropic plate model provides a valuable simplification for analyzing many mechanical behavior of the vK plate under matrix pre-strain.\\
	
	The tool used for simultaneous homogenization and dimension reduction is the re-scaling unfolding operator, which is a variation of the periodic unfolding operator introduced in \cite{CDG2} and further developed in \cite{CDG1}. The unfolding operator is very suited for periodic homogenization problems set on a domain which depends on the parameter $\e$. For more information on homogenization techniques see \cite{tartar1}, \cite{asymptotic1}, \cite{Homogenization1}, \cite{twoscale1}, and \cite{twoscale2}. Additionally, for more literature on dimension reduction see \cite{multi1}, \cite{friesecke02}.
	
	To ensure the existence of at least one minimizer for the limit energy problem, $\Gamma$-convergence techniques are used, for more details on homogenization using $\Gamma$-convergence see \cite{masoG}, \cite{ABGamma}, \cite{Friesecke06} and \cite{FJMKarman}. For general references on the theory of elasticity, see \cite{elasticity}, \cite{CiarletShell1} and for non-linear plate theory, the authors recommend the seminal work presented in \cite{friesecke02}, \cite{Friesecke06}, \cite{FJMKarman} and \cite{de2021hierarchy}. Moreover, for more literature on junction between plates and rods see \cite{BGG1} and \cite{BGG2}.	The approach presented in this paper of homogenization and dimension reduction by re-scaling unfolding operator is similar to the one given in \cite[Chapeter 11]{CDG} and \cite{DGJunction1}.\\

	Previous works have explored the topic of $\Gamma$-convergence (homogenization and dimension reduction) for plates under different non-linear (von-K\'arm\'an and bending) regimes with or without pre-strain, the authors mention a few here \cite{NV}, \cite{Neu1}, and \cite{Velcic1}. In these works the pre-strain is modeled by a multiplicative decomposition $\nabla v A^{-1}$ of the deformation gradient $\nabla v$ with a factor $A$ that is close to identity, i.e. $A=\GI_3-\e B$. In context of non-Euclidean elasticity, references such as \cite{lewicka01} and \cite{lewicka05}, assume a Riemannian manifold as the reference configuration and the factor $A$ in the decomposition $F_{el}A^{-1}$ is viewed as the square root of the metric.
	
	In this context, it is worth mentioning that when the pre-strain is assumed to be in the form of $A_\varepsilon=\GI_3-\varepsilon B_\varepsilon$ for analytical purposes, it is typically assumed that $B_\varepsilon$ converges to $B$ through strong two-scale convergence, which is equivalent to strong convergence when using the re-scaling unfolding operator. However, in the analysis conducted, it has been assumed that $B_\varepsilon$ converges weakly via the unfolding operator. This assumption constitutes a weaker condition on the pre-strain, and to the best of our knowledge, this is the first result that considers a weak convergence assumption on the pre-strain (see Assumption \ref{Ass01} for more details).
	
	An extensive series of studies on homogenization and dimension reduction problems in structures composed of beams and plates has been conducted. Specifically, the behavior of structures involving curved rods with various contact conditions has been investigated, as documented in the works \cite{GOW} and \cite{GFOW}. Furthermore, structures made of perforated plates have been the focus of attention, and for a comprehensive understanding of this topic, reference to the studies presented in \cite{hauck2} and \cite{larysa1} is recommended.

Moreover, in \cite{GOW2}, consideration was given to a woven canvas with a fixed junction between the beams. This setup is a special consequence of Case 1 with the periodic structural frame as the domain in the current study. However, due to the presence of the soft matrix in the holes, additional terms are introduced in the limit energy. Nevertheless, the limit homogenized energy remains unchanged. The cell problems in the region $Y_\kappa^B$ correspond to those defined in \cite{GOW2}, and an additional linear elastic cell problem for the matrix part is obtained due to the presence of the soft matrix. This becomes relevant when the soft matrix part is subjected to thermal, electric, or chemical expansion, resulting in in-plane stress in the homogenized structural frame (see \eqref{Eq1042}$_2$ and \eqref{Eq1046}--\eqref{EQ947}).\\
	
The paper is structured as follows:
In the beginning, general notations are introduced in Section \ref{N-1}. Following that, the composite structure and boundary conditions are defined in Sections \ref{sec03}.
The subsequent sections, namely Sections \ref{Sec04}--\ref{Sec05}, focus on deriving Korn-type inequalities for deformations. These sections provide estimates in terms of $\e$ for the linearized strain tensor corresponding to the displacement in the regions $\O_\e^B$ and $\O_\e^M$.
Moving on to Section \ref{Sec06}, a new splitting result for the displacements $u$ is established. This section decomposes every displacement acting on the whole plate $\O_\e$ into a Kirchhoff-Love (KL) displacement and a residual part. Furthermore, this decomposition is utilized to derive a Korn-type inequality for the displacements.

Section \ref{Sec07} delves into Case 1, discussing the asymptotic behavior of the composite plate without pre-strain. This case examines the re-scaling of forces and its impact.
Similarly, Section \ref{Sec08} focuses on Case 2, which explores the asymptotic behavior of the composite plate under matrix pre-strain.
Finally, the paper concludes by extending the results to a more general periodic composite plate in Section \ref{Sec09}.

	\section{Notations} \label{N-1}
	
	Throughout the paper, the following notation will be used:
	\begin{itemize}
		\item $\o\doteq(-L, L)^2$ and $\O_\e=\o\X (-r,r)$ with $L>0$;
		\item  $\e, r\in \R^+$ are small parameter such that $r=\kappa\e$ with $\kappa$ being a strictly positive fixed real constant;
		\item $\Yc:=Y\times (-\kappa, \kappa )$ be the reference cell of the structure, where $Y:=(0,1)^2$;
		\item $\o_\e=(-\e,\e)^2$;
		%		\item $\phi_i\doteq \phi\cdot\Ge_i$ for $i\in \{1,2,3\}$,
		\item $x=(x',x_3)=(x_1,x_2,x_3)\in \R^3$, $\ds \partial_i\doteq\frac{\partial}{\partial x_i}$ denote the partial derivative w.r. to $x_i$, $i\in \{1,2,3\}$;
		\item {\bf id} denote the identity map from $\R^3$ to $\R^3$, {\bf id}$_{2}$ denotes the identity map from $\R^2$ to $\R^3$ and $\GI_3$ denote the identity matrix of order $3$;
		\item $\GM_3$ is the space of $3\times 3$  real matrices and $\S_3$ is the space space of real $3\X3$ symmetric matrices;
		\item the mapping $v:\O_\e\to\R^3$ denotes the deformation arising in response to forces and pre-strain, it is related to the displacement $u$ by $u=v-\text{\bf id}$. The gradient of deformation and displacement is denoted by $\nabla v$ and $\nabla u$ respectively, with $\nabla v(x), \nabla u(x)\in \GM_3$ for every $x\in \O_\e$;
		\item  for a.e. $x'\in\R^2$, 
					$$ x'=\left(2\e\left[{x'\over 2\e}\right]+\e\left\{{x'\over2\e}\right\}\right),\quad \text{where}\;\; \left[{x'\over 2\e}\right]\in\Z^2,\;\;\left\{{x'\over2\e}\right\}\in Y;$$
		\item  $|\cdot|_F$ denote the Frobenius norm on $\R^{3\X3}$;
		\item  for every displacement $u\in H^1(\O_\e)^3$, the linearized strain tensor is given by
		$$e(u)={1\over 2}\big(\nabla u+(\nabla u)^T\big)\quad \text{and}\quad e_{ij}(u)={1\over 2}\left({\partial u_i\over \partial x_j}+{\partial u_j\over \partial x_i}\right)$$
		where $i,j\in \{1,2,3\}$;
		\item for every deformation $v\in H^1(\O_\e)^3$ the Green St. Venant's strain tensor is given by
		$${1\over 2}\left((\nabla v)^T(\nabla v)-\GI_3\right)={1\over 2}\left((\nabla u)^T+\nabla u+(\nabla u)^T(\nabla u)\right),$$
		where $u= v-\text{\bf id}$ is corresponding displacement;
		\item for every $F\in \GM_3$, denote by 
		$$ Sym(F)={1\over 2}\left(F+F^T\right)\quad\text{and}\quad {\bf E}(F)={1\over 2}\left(F^TF-\GI_3\right);$$
		\item $(\alpha,\beta)\in \{1,2\}^2$ and $(i,j)\in \{1,2,3\}^3$ (if not specified).
	\end{itemize}
	In this paper, the Einstein convention of summation over repeated indices is employed, and a generic constant denoted as $C$ is used, which is independent of $\e$. These represent some of the general notations utilized in this paper. Notations that are not explicitly defined here can be found in the main content of the paper.

	\section{Description of the structure and natural assumption} \label{sec03}
In this section, the structure is defined as an elastic body that occupies a slender three-dimensional domain in its undeformed configuration. This domain is denoted by $\Omega_\varepsilon$, which is given by $\Omega_\varepsilon = \omega \times (-r,r)$, with $\omega = (-L,L)$ and $L>0$.

It is assumed that the elastic body consists of a composite material distributed as a periodic perforated plate with holes filled with a relatively softer heterogeneous material. The composite plate is fixed in a cylindrical part.

The section concludes by presenting the sets of admissible deformations and displacements.
	
	\begin{figure}[ht]
		\centering	
		\subfigure[$\kappa={1\over 8}$]{\begin{tikzpicture}[scale=0.5]
				
				\begin{scope}
					\clip (0,-0.50)--(0,10)--(10.25,10)--(10.25,-0.50)--cycle;
					
					\foreach \x in 
					{0.25, 1.25, 2.25, 3.25, 4.25, 5.25, 6.25, 7.25, 8.25, 9.25, 10.25, 11.25, 12.25, 13.25}
					{\fill[red!70] (\x,0) rectangle (\x+0.75, 0.75);}
					
					\foreach \x in
					{0, 1, 2, 3, 4, 5, 6, 7, 8, 9, 10, 11, 12, 13}
					{\fill[deepskyblue] (\x, -0.25) rectangle (\x+0.25, 12);}
					
					\foreach \y in
					{ -0.25, 0.75, 1.75, 2.75, 3.75,4.75,5.75,6.75,7.75,8.75,9.75,10.75,11.75}
					{ \fill[deepskyblue] (0,\y) rectangle (12.25,\y+0.25);}
					
					\foreach \x in {0.25,1.25,2.25,3.25,4.25,5.25,6.25,7.25,8.25,9.25,10.25,11.25,12.25,13.25,14.25,15.25,16.25,17.25,18.25,19.25,20.25}
					{
						\fill[red!70] (\x,1.00) rectangle (\x+0.75,1.75);
						\fill[red!70] (\x,2.00) rectangle (\x+0.75,2.75);
						\fill[red!70] (\x,3.00) rectangle (\x+0.75,3.75);
						\fill[red!70] (\x,4.00) rectangle (\x+0.75,4.75);
						\fill[red!70] (\x,5.00) rectangle (\x+0.75,5.75);
						\fill[red!70] (\x,6.00) rectangle (\x+0.75,6.75);
						\fill[red!70] (\x,7.00) rectangle (\x+0.75,7.75);
						\fill[red!70] (\x,8.00) rectangle (\x+0.75,8.75);
						\fill[red!70] (\x,9.00) rectangle (\x+0.75,9.75);
						\fill[red!70] (\x,10.00) rectangle (\x+0.75,10.75);
						\fill[red!70] (\x,11.00) rectangle (\x+0.75,11.75);
						\fill[red!70] (\x,12.00) rectangle (\x+0.75,12.75);
						\fill[red!70] (\x,13.00) rectangle (\x+0.75,13.75);
						\fill[red!70] (\x,14.00) rectangle (\x+0.75,14.75);
						\fill[black!50, semitransparent ] (\x,15.00) rectangle (\x+0.75,15.75);
						\fill[black!50, semitransparent ] (\x,16.00) rectangle (\x+0.75,16.75);
						\fill[black!50, semitransparent ] (\x,17.00) rectangle (\x+0.75,17.75);
						\fill[black!50, semitransparent ] (\x,18.00) rectangle (\x+0.75,18.75);
						\fill[black!50, semitransparent ] (\x,19.00) rectangle (\x+0.75,19.75);
						\fill[black!50, semitransparent] (\x,20.00) rectangle (\x+0.75,20.75);
					}
					
					%\draw[blue!50] (0,0.25)--(21, 0.25);
					\draw[black] (0,0.75)--(21, 0.75);

					\draw[black] (0,0)--(11,0);
					\draw[black] (0,-0.25)--(0,10);
					\draw[black] (0,-0.25)--(10.25,-0.25);
					\foreach \y in 
					{1,2,3,4,5,6,7,8,9,10,11,12,13,14,15,16,17,18,19,20,21}
					{\draw[black] (0, \y)--(21,\y);}
					
					\foreach \y in
					{ 1.75,2.75,3.75,4.75,5.75,6.75,7.75,8.75,9.75,10.75,11.75,12.75,13.75,14.75,15.75,16.75,17.75,18.75,19.75,20.75,21.75}
					{ \draw[black] (0, \y)--(21, \y);}
					
					\draw[black] (0.25, -0.25)--(0.25, 21);
					\draw[black] (1.25,0)--(1.25, 21);
					\draw[black] (10.50, 0)--(10.50, 21);
					
					\foreach \x in 
					{1,2,3,4,5,6,7,8,9,10,11,12,13,14,15,16,17,18,19,20,21}
					{\draw[black] (\x, -0.25)--(\x,21);}

					\foreach \x in
					{ 1.25,2.25,3.25,4.25,5.25,6.25,7.25,8.25,9.25,10.25,11.25,12.25,13.25,14.25,15.25,16.25,17.25,18.25,19.25,20.25,21.25}
					{\draw[black] (\x, -0.25)--(\x, 21);}
					
					%					\fill[blue]
					%					         (0,0) rectangle (21, 0.25);
					
				\end{scope}
				\draw[black, ultra thick] (5,5) circle (2.5);
		\end{tikzpicture} }\hspace{0.2cm}
	\subfigure[XY-axis]{\begin{tikzpicture}
			  % Arrow from (0,0) to (0,1) with label "e"
			\draw[->] (5,0) -- (5,1) node[midway, left] {$\Ge_2$};
			
			% Arrow from (0,0) to (1,0) with label "f"
			\draw[->] (5,0) -- (6,0) node[midway, below] {$\Ge_1$};
		\end{tikzpicture}
	
	} \hspace{0.2cm}
		\subfigure[$\kappa={3\over 8}$]{\begin{tikzpicture}[scale=0.445]
				\foreach \x in 
				{0,0.75,1,1.75,2,2.75,3,3.75,4,4.75,5,5.75,6,6.75,7,7.75,8,8.75,9,9.75,10,10.75,11,11.75}
				{\draw[black] (\x, -0.75)--(\x, 11);}
				\foreach \y in 
				{-0.75, 0, 0.25,1,1.25,2, 2.25, 3, 3.25, 4, 4.25, 5, 5.25,6,6.25, 7, 7.25, 8,8.25,9,9.25,10,10.25,11}
				{\draw[black] (0,\y)--(11.75,\y);}
				\foreach \x in {0.75,1.75,2.75,3.75,4.75,5.75,6.75,7.75,8.75,9.75,10.75}
				{
					\fill[red!70] (\x, 0) rectangle (\x+0.25,0.25);
					\fill[red!70] (\x, 1) rectangle (\x+0.25,1.25);
					\fill[red!70] (\x,2.00) rectangle (\x+0.25,2.25);
					\fill[red!70] (\x,3.00) rectangle (\x+0.25,3.25);
					\fill[red!70] (\x,4.00) rectangle (\x+0.25,4.25);
					\fill[red!70] (\x,5.00) rectangle (\x+0.25,5.25);
					\fill[red!70] (\x,6.00) rectangle (\x+0.25,6.25);
					\fill[red!70] (\x,7.00) rectangle (\x+0.25,7.25);
					\fill[red!70] (\x,8.00) rectangle (\x+0.25,8.25);
					\fill[red!70] (\x,9.00) rectangle (\x+0.25,9.25);
					\fill[red!70] (\x,10.00) rectangle (\x+0.25,10.25);
					
				}
				\foreach \x in {0,1,2,3,4,5,6,7,8,9,10,11}
				{
					\fill[deepskyblue, semitransparent] (\x, -0.75) rectangle (\x+0.75, 11);
				}
				\foreach \y in {-0.75, 0.25, 1.25, 2.25,3.25,4.25,5.25,6.25,7.25,8.25,9.25,10.25}
				{
					\fill[deepskyblue, semitransparent] (0,\y) rectangle (11.75,\y+0.75);
				} 
				\draw[black, ultra thick] (5.875,5.125) circle (3);
			\end{tikzpicture}
		}
		\caption{
			Domain $\O_\e$ with circular clamped condition, the blue part and red part describes $\O_\e^B$ and $\O_\e^M$ respectively.}\label{Fig01}
	\end{figure}
	
	%%%%%%%%%%%%%%%%%%%%%%%%%%%%%%%%%%%%%%%%%%%%%%%%%%%%%%%%%%%%%%%%%%%%

	\subsection{Structure of the stiff connected part}
	The structure is made of yarns orthogonal to each other. Then, the straight reference beams in the two directions are defined by 
$$	\begin{aligned}
		&P^{(1)}_\e\doteq \bigl\{x\in \R^3\;|\; x_1\in (-L,L),\enskip (x_2,x_3)\in  (-r,r)^2\big\},\\
		&P^{(2)}_\e\doteq  \bigl\{x\in \R^3\;|\; x_2\in (-L,L),\enskip (x_1,x_3)\in  (-r,r)^2\big\}.
	\end{aligned}$$
	
	Let us denote 
	$$ P^{(1,q)}_\e = q\e {\bf e_2} + P^{(1)}_\e \;\; \text{and} \;\;  P^{(2,p)}_\e = p\e {\bf e_1} + P^{(2)}_\e. $$
	Here, the middle line of the straight beams $P^{(1,q)}_\e$ and $P^{(2,p)}_\e$ passes through the points $(0,q)$ and $(p,0)$, respectively.\\
	The periodic structural frame $\O^{B}_\e$ is given by 
	$$\O^{B}_\e\doteq  \O_\e\cap\Big( \bigcup_{q=-N_\e}^{N_\e} P^{(1,q)}_\e \cup  \bigcup_{p=-N_\e}^{N_\e} P^{(2,p)}_\e\Big).$$
	%	  where   
	%	  $$\O_\e =\o \X  (-\kappa\e, \kappa\e),\qquad \O\doteq (-L,L)^2.$$
	 The part  of $\O_\e$ filled by soft matrix material  is denoted by $\O_\e^M$  (See Figure \ref{Fig01}) and is given by
	
	 	$$\O^{M}_\e = \O_\e \backslash \overline{\O^{B}_\e }. $$
	The composite plate $\O_\e$ is the periodic structural frame $\O^B_\e$ with the holes filled with soft matrix.

Here, it is important to note that $\kappa$ belongs to the interval $(0, \frac{1}{2})$. This is because there exist soft matrix between the beams, meaning that 
$$\e - 2\kappa \e > 0\quad\text{with}\quad \kappa > 0.$$
	
	\subsection{The boundary condition}
	It is assumed that the cylindrical part $\Gamma_r$ of the composite plate is fixed, its intersection with the mid-surface is a Lipschitz continuous boundary $\gamma$, i.e.
$$
		\Gamma_r=\gamma\X(-r,r)\quad\text{such that}\;\;\gamma\subset\overline{\o}.
$$
	Then, the set of admissible deformations are denoted by
$$
		\GV_{\e}\doteq\Big\{v\in H^1(\O_{\e})^3\;|\; v=\text{\bf id}\quad \hbox{a.e. in } \Gamma_r\Big\}.
$$
	\begin{remark}\label{Re1}
	The non-linear elasticity problems stated in Subsections \ref{Ssec71} and \ref{Ssec81} are formulated in terms of deformations. Therefore, it is crucial to begin with a decomposition of deformation and derive Korn-type inequalities for $v$.
	\end{remark}
	
	\section{Estimates in the periodic perforated plate $\O_\e^B$}\label{Sec04}
   In this section, an extension result for deformation is first introduced. Subsequently, the results on plate deformations are reviewed, and our assumptions on the boundary conditions are simplified. Following this, the plate deformation is applied to the extended periodic structural frame, and a Korn-type inequality for it is derived. Finally, an estimate for the strain tensor $e(u)$ corresponding to the displacement $u=v-\text{\bf id}$ in the region $\O_\e^B$ is provided.
	
	\subsection{Extension result}
	An extension result from \cite{GOW2} is recalled. This extension depends on the fact that beams intersect each other.
		\begin{proposition}\label{ExDef1}
			For every deformation $v$ in $H^1(\O^B_\e)^3$, there exists a deformation $\Gv\in H^1(\O_\e)^3$ satisfying
			\begin{equation}\label{EQ51-0}
				\begin{aligned}
					&\Gv=v,\quad \text{a.e. on}\quad \O_\e^B,\\
					&\|dist(\nabla\Gv,SO(3))\|_{L^2(\O_\e)}\leq C\|dist(\nabla v,SO(3))\|_{L^2(\O_\e^B)}.
				\end{aligned}
			\end{equation}
			The constant does not depend on $\e$.\\
			Moreover, if $v\in H^1(\O_\e^B)^3$ such that $v=0$ a.e on $\Gamma_r$ then
			$$\Gv=\text{\bf id}\qquad \hbox{a.e. in } \; \gamma\X (-r,r),$$ where $\gamma\subset \overline{\o}$.
		\end{proposition}	
		
		Henceforth, we use the extended deformation $\Gv\in H^1(\O_\e)^3$ (resp $\GV_\e$), which is a deformation of the whole  plate $\O_\e$.
		\subsection{Decomposition of the plate deformations} \label{Elementary}
		Let $v \in \GV^B_\e$ be a deformation, which is extended to the whole plate $\O_\e$ using Proposition \ref{ExDef1}. The extension of $v$ is denoted $\Gv$. From \cite[Theorem 3.1]{Shell1}, we know that if  $\|dist(\nabla\Gv,SO(3))\|_{L^2(\O_\e)}\leq C\e^{3/2}$ (where $C$ is a constant which only depends on $L$ and $\kappa$) then
		\begin{equation}\label{EQ51}
			\Gv = V_{e} + \overline{v} \;\; \text{ a.e. in } \;\; \O_\e.
		\end{equation}
		The quantity   $V_{e} \in H^1{(\O_\e)}^3$ is called elementary plate deformation and is defined by 
		\begin{equation}\label{Eqelem}
			V_e(x)= \mathcal{V}(x_1,x_2) +x_3{\GR}(x_1,x_2) \Ge _{3}  \quad \hbox{for a.e. } x\in \O_\e
		\end{equation}  with $\Vc\in H^1(\o)^3$ and $\GR\in H^1(\o, SO(3))^3$. $\overline{v} \in H^1{(\O_\e)}^3$ is called warping (or residual term).
		The estimates of the terms of this decomposition are given by  \cite[Theorem 3.1]{Shell1} \footnotemark
		\footnotetext{By convention, in estimates, it is commonly denoted as $L^2(\Omega_\varepsilon)$ instead of $L^2(\Omega_\varepsilon)^3$ or $L^2(\Omega_\varepsilon)^{3\times 3}$. The complete space notation is used when referring to weak or strong convergences.}
		\begin{equation}\label{E0}
			\begin{aligned}
				&\|\overline{v}\|_{L^2(\O_\e)}+\e\|\nabla \overline{v}\|_{L^2(\O_\e)}+ \e\|\nabla \Gv-{\GR}\|_{L^2(\O_\e)}\le C\e\|dist(\nabla \Gv,SO(3))\|_{L^2(\O_\e)},\\
				&\big\| \partial_\alpha\GR\big\|_{L^2(\o)}\le {C\over \e^{3/2}}\|dist(\nabla \Gv,SO(3))\|_{L^2(\O_\e)},\\
				&\big\|\partial_\alpha\mathcal{V}-{\GR}\Ge_\alpha\big\|_{L^2(\o)}\leq {C \over \e^{1/2}} \|dist(\nabla \Gv,SO(3))\|_{L^2(\O_\e)}.
			\end{aligned}
		\end{equation}  The constant does not depend on $\e$.\\
	Now, from the above decomposition we have for $\Gv\in\GV_\e$ 
	$$\Vc=\text{\bf id}_2,\quad \GR=\GI_3\quad \text{a.e on}\;\;\gamma\quad\text{and}\quad\overline{v}=0,\quad\text{a.e. on}\quad\Gamma_r.$$
		\begin{lemma}\label{FinEst1}
			The terms $\GR$ and $\Vc$ of the decomposition \eqref{Eqelem} satisfy
			\begin{equation}\label{FinEst} 
				\begin{aligned}
					&\|\GR-\GI_3\|_{H^1(\o)}+\|\Vc-\text{\bf id}_2\|_{H^1(\o)} \leq {C\over \e^{3/2}} \|dist(\nabla \Gv, SO(3))\|_{L^2(\O_\e)}.
				\end{aligned}
			\end{equation} The constants do not depend on $\e$.
		\end{lemma}	
		\begin{proof}
			Estimate  \eqref{E0}$_2$, the boundary conditions and Poincar\'e inequality give
			$$\begin{aligned}
			\|\GR-\GI_3\|_{H^1(\o)} \leq {C\over \e^{3/2}} \|dist(\nabla \Gv, SO(3))\|_{L^2(\O_\e)},\;\;
               \|\partial_\alpha\Vc-\Ge_\alpha\|_{L^2(\o)}\leq {C\over \e^{3/2}}\|dist(\nabla \Gv,SO(3))\|_{L^2(\O_\e)}
			\end{aligned}.$$
			From the above estimate and  the boundary conditions on $\Vc$ we obtain
			$$ \|\Vc-\text{\bf id}_2\|_{L^2(\o)}\leq  {C\over \e^{3/2}}\|dist(\nabla \Gv,SO(3))\|_{L^2(\O_\e)}.$$
			Hence from the above estimates we get \eqref{FinEst}$_2$.
		\end{proof}	
		From the above inequalities, we attain the non-linear Korn-type inequality for the plate $\O_\e$.
		\begin{lemma}\label{KornB1}
			One has
			\begin{equation}\label{EQ448} 
				\|\Gv-\text{\bf id}\|_{H^1(\O_\e)}\leq{C\over\e}\|dist(\nabla \Gv,SO(3))\|_{L^2(\O_\e)}.
			\end{equation} The constant does not depend on $\e$.
		\end{lemma}
		We also have
		\begin{lemma}
			The linearized strain tensor of an admissible deformation  satisfies
			\begin{equation} \label{Eq II.3.10}
				\|e(\Gu)\|_{L^2(\O_\e)} \leq  C  \|dist(\nabla \Gv, SO(3))\|_{L^2{(\O_\e)}} + {C \over \e^{5/2}} \|dist(\nabla \Gv, SO(3))\|^2_{L^2{(\O_\e)}}.
			\end{equation} 
			where $\Gu=\Gv-\text{\bf id}$. The constant does not depend on $\e$.
		\end{lemma}	
		\begin{proof}
			First, observe that $$2e(\Gu)=\nabla \Gv +(\nabla \Gv)^T- 2\GI_3=(\nabla \Gv-\GR)+(\nabla \Gv-\GR)^T+(\GR+\GR^T-2\GI_3).$$
			Hence, using estimate \eqref{E0}$_3$, we get
			\begin{equation}\label{EQ414}
				\begin{aligned}
					\|\nabla \Gv +(\nabla \Gv)^T- 2\GI_3\|_{L^2(\O_\e)}  & \leq 2\|\nabla \Gv-\GR\|_{L^2(\O_\e)}+ C\e^{1/2}\|\GR+\GR^T-2\GI_3\|_{L^2(\o)}\\
					&\leq C\|dist(\nabla \Gv, SO(3))\|_{L^2(\O_\e)} +C\e^{1/2}\|(\GR-\GI_3)^2\|_{L^2(\o)},\\
					&\leq C\|dist(\nabla \Gv, SO(3))\|_{L^2(\O_\e)} +C\e^{1/2}\|\GR-\GI_3\|^2_{L^4(\o)},
				\end{aligned}
			\end{equation}
			since we have the equality $\GR+\GR^T-2\GI_3 = \GR^T(\GR-\GI_3)^2$ and $\|\GR\|_{L^\infty(\o)}=\sqrt{3}$.
			Now,  $H^1(\o)$ is continuously included in $L^4(\o)$, so 
			$$
			\|\GR-\GI_3\|_{L^4(\o)}\leq C\|\GR-\GI_3\|_{H^1(\o)} \leq{C\over \e^{3/2}} \|dist(\nabla \Gv, SO(3))\|_{L^2(\O_\e)}.
			$$
			Hence, from the above estimates we get the required estimate for the linearized strain tensor given by \eqref{Eq II.3.10}.
		\end{proof}			
		As a consequence of the estimates \eqref{EQ448} and \eqref{Eq II.3.10}, together with the extension result \ref{ExDef1}, we obtain the following result:
		\begin{lemma}\label{EFSTB}
			For the corresponding displacement $u=v-\text{\bf id}$, we have
$$
				\begin{aligned}
					&\|u\|_{H^1(\O^B_\e)}\leq{C\over\e}\|dist(\nabla v,SO(3))\|_{L^2(\O^B_\e)},\\
					%&\|\nabla(v-I_d)\|_{L^2(\O^B_\e)}\leq{C\over\e}\|dist(\nabla v,SO(3))\|_{L^2(\O^B_\e)},\\
					&\|e(u)\|_{L^2(\O^B_\e)} \leq  C  \|dist(\nabla v, SO(3))\|_{L^2{(\O^B_\e)}} + {C \over \e^{5/2}} \|dist(\nabla v, SO(3))\|^2_{L^2{(\O^B_\e)}}.
				\end{aligned}
$$
			The constants do not depend on $\e$.	
		\end{lemma}	
		
		\section{Estimates in the region with soft matrix $\O_\e^M$}\label{Sec05}
     In this section, a Korn-type inequality and an estimate for the strain tensor $e(u)$ corresponding to the displacement $u$ for the part filled with soft material $\O_\e^M$ are derived. This is achieved by utilizing the fact that $\O_\e^M$ is a union of small plates $\O_\e^{pq}$.
	We start with the following set up
		\begin{equation}\label{Def2}
			\begin{aligned}
				&Y_\kappa=\big(\kappa,1-\kappa\big)^2,\qquad &&\wt{Y}_\kappa=Y\setminus Y_\kappa,\\
				%&{\GY}_\kappa\doteq \big(-\kappa,1+\kappa\big)^2, \quad 
				& \o_\e^{pq}=\e\big(\xi_{pq}+ Y_\kappa\big),\quad &&\O_\e^{pq}=\o_\e^{pq}\X(-r,r),\\
				& \wt{\o}_\e^{pq}=\e\big(\xi_{pq}+ \wt{Y}_\kappa\big) ,\quad &&\wt{\O}_\e^{pq}=\wt{\o}_\e^{pq}\X(-r,r ),\\
				&\xi_{pq}=p\Ge_1+q\Ge_2,\quad  &&(p,q)\in\{-N_\e,\ldots,N_\e-1\}^2.
			\end{aligned}
		\end{equation}
		Observe that
		\begin{equation}\label{Def3}
			\O_\e^M= \bigcup_{p=-N_\e}^{N_\e-1}\bigcup_{q=-N_\e}^{N_\e-1} \O_\e^{pq} \quad\text{and}\quad \O_\e^B=\bigcup_{p=-N_\e}^{N_\e-1}\bigcup_{q=-N_\e}^{N_\e-1} \wt{\O}_\e^{pq}.
		\end{equation}
		Now, observe that $\O_\e^{pq}$ is a star-shaped domain with respect to a ball of radius $\ds \e\inf\Big\{\kappa,{1}-\kappa \Big\}$ and that its diameter is less than $\e(2+2\kappa)$.
		So, we know from Theorem II.1.1 in \cite{BGRod} that there exist a vector $\Ga_{pq}\in \R^3$ and  a matrix $\GR_{pq}\in SO(3)$ such that 
		\begin{equation}\label{BGShell3.3}
			\begin{aligned}
				&  \|\nabla v-\GR_{pq}\|_{L^2(\O_\e^{pq})}\leq C\|dist(\nabla v, SO(3))\|_{L^2(\O_\e^{pq})},\\
				& \big\|v-\Ga_{pq}-\GR_{pq}(x-\e\xi_{pq})\big\|_{L^2(\O_\e^{pq})}\leq C\e\|dist(\nabla v, SO(3))\|_{L^2(\O_\e^{pq})}.
			\end{aligned}
		\end{equation}
		The constant does not depend on $\e$, $p$ and $q$, it only depends on the ratio $\ds {2+2\kappa\over \inf\big\{\kappa,{1}-\kappa\big\}}$ .
		\begin{figure}[ht]\label{Fig03}
			\centering
			\subfigure[$\kappa={1\over 8}$]{\begin{tikzpicture}[scale=1.1]
					
					\draw (1.5,1.5)--(4,1.5);
					\draw (1.5,1.5)--(1.5,4);
					\draw (2,1.5)--(2,4);
					\draw (1.5,2)--(4,2);
					\draw (1.5,3.5)--(4,3.5);
					\draw (1.5,4)--(4,4);
					\draw (3.5,1.5)--(3.5,4);
					\draw (4,1.5)--(4,4);

					%\draw (4.5,2)--(6.5,2);
					%\draw (6.5,2)--(6.5,4);
					%\draw (6.5,4)--(4.5,4);
					
					%\draw (6.5,2)--(7,2);
					%\draw (7,2)--(7,4);
					%\draw (7,4)--(6.5,4);

					%\fill[red, semitransparent] (2,2) rectangle (4,4);	
					%\fill[red, semitransparent] (4.5,2) rectangle (6.5,4);
					
					%\fill[red, semitransparent] (4,1.5) rectangle (4.5,2);	
					
					\fill[deepskyblue, semitransparent] (1.5,1.5) rectangle (2,4);	
					%\fill[blue, semitransparent] (4,1.5) rectangle (7,2);	
					\fill[deepskyblue, semitransparent] (1.5,4) rectangle (4,3.5);	
					%\fill[blue, semitransparent] (4,4) rectangle (7,4.5);
					\fill[deepskyblue, semitransparent] (1.5,2) rectangle (4,1.5);	
					\fill[deepskyblue, semitransparent] (3.5,1.5) rectangle (4,4);	
					%\fill[blue, semitransparent] (6.5,2) rectangle (7,4);
					\fill[deepskyblue, semitransparent] (1.5, 1.5) rectangle (2,2);
					\fill[deepskyblue, semitransparent] (3.5, 1.5) rectangle (4,2);
					\fill[deepskyblue, semitransparent] (1.5, 3.5) rectangle (2,4);
					\fill[deepskyblue, semitransparent] (3.5, 3.5) rectangle (4,4);

					\fill[red!70] (2,2) rectangle (3.5,3.5);

					%\fill[red, semitransparent] (4,1.5) rectangle (4.5,2);	

					%\draw[ thick, ->] (5.7,1.8,-0.25) -- (5.7,1.2,-0.25) node[shift={(-0.4,-0.4)}] {$ \o_\e^{pq}$};
					
			\end{tikzpicture}} \hspace{2cm}
			\subfigure[$\kappa={3\over 8}$]{\begin{tikzpicture}[scale=0.81]
					
					\draw[black] (1.5,1.5)--(5,1.5);
					\draw[black] (1.5,1.5)--(1.5,5);
					\draw[black] (3,1.5)--(3,5);
					\draw[black] (1.5,3)--(5,3);
					\draw[black] (1.5,3.5)--(5,3.5);
					\draw[black] (1.5,5)--(5,5);
					\draw[black] (3.5,1.5)--(3.5,5);
					\draw[black] (5,1.5)--(5,5);

					%\draw (4.5,2)--(6.5,2);
					%\draw (6.5,2)--(6.5,4);
					%\draw (6.5,4)--(4.5,4);
					
					%\draw (6.5,2)--(7,2);
					%\draw (7,2)--(7,4);
					%\draw (7,4)--(6.5,4);

					%\fill[red, semitransparent] (2,2) rectangle (4,4);	
					%\fill[red, semitransparent] (4.5,2) rectangle (6.5,4);
					
					%\fill[red, semitransparent] (4,1.5) rectangle (4.5,2);	
					
					\fill[deepskyblue, semitransparent] (1.5,1.5) rectangle (3,5);	
					%\fill[blue, semitransparent] (4,1.5) rectangle (7,2);	
					\fill[deepskyblue, semitransparent] (1.5,5) rectangle (5,3.5);	
					%\fill[blue, semitransparent] (4,4) rectangle (7,4.5);
					\fill[deepskyblue, semitransparent] (1.5,3) rectangle (5,1.5);	
					\fill[deepskyblue, semitransparent] (3.5,1.5) rectangle (5,5);	
					%\fill[blue, semitransparent] (6.5,2) rectangle (7,4);
					\fill[deepskyblue, semitransparent] (1.5, 1.5) rectangle (3,3);
					\fill[deepskyblue, semitransparent] (3.5, 1.5) rectangle (5,3);
					\fill[deepskyblue, semitransparent] (1.5, 3.5) rectangle (3,5);
					\fill[deepskyblue, semitransparent] (3.5, 3.5) rectangle (5,5);

					\fill[red!70] (3,3) rectangle (3.5,3.5);

					%\fill[red, semitransparent] (4,1.5) rectangle (4.5,2);	

					%       		\draw[thick, ->] (1.5,1.5,-2) -- (0,1.5,-2) node[shift={(-0.3,0.1)}] {$\o_\e^{pq}$};%
					%       		\draw[ thick, ->] (4.1,1.8,-0.25) -- (4.1,1.2,-0.25) node[shift={(-0.2,-0.4)}] {$\wt{\o}_\e^{pq}$};
					%       		\draw[ thick, ->] (5.7,1.8,-0.25) -- (5.7,1.2,-0.25) node[shift={(-0.4,-0.4)}] {$ \o_\e^{pq}$};

					%\draw[ thick, ->] (5.7,1.8,-0.25) -- (5.7,1.2,-0.25) node[shift={(-0.4,-0.4)}] {$ \o_\e^{pq}$};
					
			\end{tikzpicture}}
			\caption{The red and blue part represent $\o_\e^{pq}$ and $\wt{\o}_\e^{pq}$ respectively for some $p,q$.}
		\end{figure}
		
		Recalling the following classical estimates of the traces 
		\begin{lemma}\label{IEQ3} For every $ \phi\in H^1(Y_\kappa\X (-\kappa,\kappa))$ and  $ \psi\in H^1(\wt{Y}_\kappa\X (-\kappa,\kappa))$, one has
%			\begin{equation}\label{IEQ2-1}
%				\|\phi\|_{L^2(\partial Y_\kappa\X (-\kappa,\kappa))}\leq C\big(\|\phi\|_{L^2(Y_\kappa\X (-\kappa,\kappa))}+\|\nabla\phi\|_{L^2(Y_\kappa\X (-\kappa,\kappa))}\big).
%			\end{equation}
			\begin{equation}\label{IEQ2-2}
				\begin{aligned}
						\|\phi\|_{L^2(\partial Y_\kappa\X (-\kappa,\kappa))}\leq C\big(\|\phi\|_{L^2(Y_\kappa\X (-\kappa,\kappa))}+\|\nabla\phi\|_{L^2(Y_\kappa\X (-\kappa,\kappa))}\big),\\
								\|\psi\|_{L^2(\partial Y_\kappa\X (-\kappa,\kappa))}\leq C\big(\|\psi\|_{L^2(\wt{Y}_\kappa\X (-\kappa,\kappa))}+\|\nabla\psi\|_{L^2(\wt{Y}_\kappa\X (-\kappa,\kappa))}\big).
				\end{aligned}
			\end{equation}
		\end{lemma}
			\begin{lemma}\label{lem52}
			Let $\Ga$ be in $\R^3$ and $\GB\in \R^{3\X 3}$, we set 
			$$\Phi(x)=\Ga+\GB\,x,\qquad x\in \R^3.$$ Then we have
			\begin{equation}\label{EQ55}
				c\big(|\Ga|^2+|\GB|_F^2\big)\leq \|\Phi\|^2_{L^2(\partial Y_\kappa\X (-\kappa,\kappa))}\leq C\big(|\Ga|^2+|\GB|_F^2\big),
			\end{equation}  where $|\cdot|$  is the euclidean norm on $\R^3$. 
			The constants $c$ and $C$  depend  only on $\kappa$.
		\end{lemma}
		\begin{proof} The maps 
			$$(\Ga,\GB)\in \R^3\X\R^{3\X 3}\longmapsto \sqrt{|\Ga|^2+|\GB|_F^2}\quad \hbox{and}\quad (\Ga,\GB)\in \R^3\X\R^{3\X 3}\longmapsto \|\Phi\|_{L^2(\partial Y_\kappa\X (-\kappa,\kappa))} $$ are norms on the space $ \R^3\X\R^{3\X 3}$. Since this space is of finite dimension, all norms are equivalent.
		\end{proof}		
		As a consequence of the above lemmas, we can derive the following result 
		\begin{lemma}\label{lem53} (Korn-type inequalities)
			Let $v$ be a deformation in $H^1(\O_\e^M)^3$ and the plate is clamped in the part $\Gamma_r$. Then, we have
			\begin{equation}\label{KornP2}
				\begin{aligned}
					\|\nabla (v-\text{\bf id})\|_{L^2(\O_\e^M)} & \leq C_M\Big(\|dist(\nabla v, SO(3))\|_{L^2(\O_\e^M)}+{1\over\e}\|dist(\nabla v,SO(3))\|_{L^2(\O_\e^B)}\Big),\\
					\|v-\text{\bf id}\|_{L^2(\O_\e^M)}&\leq C\Big(\|dist(\nabla v, SO(3))\|_{L^2(\O_\e^M)}+{1\over\e}\|dist(\nabla v,SO(3))\|_{L^2(\O_\e^B)}\Big).
				\end{aligned}
			\end{equation} 
			The constants $C$ and $C_M$ does not depend on $\e$.
		\end{lemma}
		\begin{proof}
			{\it Step 1.} We show that
			\begin{equation}\label{eqfvm1}
				\sum_{p=-N_\e}^{N_\e-1}\sum_{q=-N_\e}^{N_\e-1}\big\|v-\Ga_{pq}-\GR_{pq}(x-\e\xi_{pq})\big\|^2_{L^2(\partial \o_\e^{pq}
					\X(-\kappa\e,\kappa\e))}\leq  C \e\|dist(\nabla v, SO(3))\|^2_{L^2(\O_\e^M)}.
			\end{equation}	
			From the estimates \eqref{BGShell3.3} and \eqref{IEQ2-2}$_1$, after a change of variables, we obtain
			$$
			\begin{aligned}
				\big\|v-\Ga_{pq}-&\GR_{pq}(x-\e\xi_{pq})\big\|^2_{L^2(\partial\o_\e^{pq}\X(-\kappa\e,\kappa\e))}\leq C\Big({1\over \e}\big\|v-\Ga_{pq}-\GR_{pq}(x-\e\xi_{pq})\big\|^2_{L^2(\O_\e^{pq})}\\
				&+\e \big\|\nabla\big(v-\Ga_{pq}-\GR_{pq}(x-\e\xi_{pq})\big)\big\|^2_{L^2(\o_\e^{pq}\X(-\kappa\e,\kappa\e))}\Big)\leq C\e\|dist(\nabla v, SO(3))\|^2_{L^2(\O_\e^{pq})}.
			\end{aligned}
			$$
			Adding all the above estimates lead to the required result.\\[1mm]
			\noindent{\it Step 2.} We prove that
			\begin{equation}\label{eqfvm2}
				\sum_{p=-N_\e}^{N_\e-1}\sum_{q=-N_\e}^{N_\e-1}\big\|v- V_e\big\|^2_{L^2(\partial \o_\e^{pq}\X(-\kappa\e,\kappa\e))}\leq C\e\|dist(\nabla v, SO(3))\|^2_{L^2(\O^B_\e)}.
			\end{equation} 
			Using the trace inequality \eqref{IEQ2-2}$_2$, after a change of variables, we get
			$$
			\begin{aligned}
				\big\|v-V_e\big\|^2_{L^2(\partial\o_\e^{pq}\X(-\kappa\e,\kappa\e))}&\leq C\Big({1\over \e}\big\|v-V_e\big\|^2_{L^2(\wt{\O}_\e^{pq})}+\e \big\|\nabla\big(v-V_e\big)\big\|^2_{L^2(\wt{\O}_\e^{pq})}\Big)\\
				&= C\Big({1\over \e}\big\|\overline{v}\big\|^2_{L^2(\wt{\O}_\e^{pq})}+\e \big\|\nabla\overline{v}\big\|^2_{L^2(\wt{\O}_\e^{pq})}\Big).
			\end{aligned}
			$$
			Now, from the estimates \eqref{E0}$_{1,2}$, we get \eqref{eqfvm2}.\\[1mm]
			\noindent{\it Step 3.} We prove
			\begin{equation}\label{EQ610}
				\begin{aligned}
					&\sum_{p=-N_\e}^{N_\e-1}\sum_{q=-N_\e}^{N_\e-1}\big|\Mc_{pq}(\Vc)-\Ga_{pq}\big|^2 \leq  {C\over \e}\big(\|dist(\nabla v, SO(3))\|^2_{L^2(\O_\e^M)}+\|dist(\nabla v, SO(3))\|^2_{L^2(\O_\e^B)}\big),\\
					&\sum_{p=-N_\e}^{N_\e-1}\sum_{q=-N_\e}^{N_\e-1}\big| \Mc_{pq}(\GR)-\GR_{pq}\big|^2\leq {C\over \e^3}\big(\|dist(\nabla v, SO(3))\|^2_{L^2(\O_\e^M)}+\|dist(\nabla v, SO(3))\|^2_{L^2(\O_\e^B)}\big).
				\end{aligned}
			\end{equation}
			For every $\phi\in H^1(\o)^k$, $k\in \N^*$, we set 
			$$\Mc_{pq}(\phi)={1\over |\wt{\o}^{pq}_\e|}\int_{\wt{\o}^{pq}_\e}\phi(x')\, dx',\qquad (p,q)\in \{-N_\e,\ldots,N_\e-1\}^2.$$
			Now using  Poincar\'e-Wirtinger inequality, we obtain the following three estimates 
			$$\|\GR-\Mc_{pq}(\GR)\|_{L^2(\wt{\o}^{pq}_\e)}\leq C\e\|\nabla\GR\|_{L^2(\wt{\o}^{pq}_\e)}$$	and
			$$\begin{aligned}
				&\|\nabla\left(\Vc-\Mc_{pq}(\Vc)-\Mc_{pq}(\GR)(x'-\e\xi_{pq})\right)\|_{L^2(\wt{\o}^{pq}_\e)} \leq C\sum_{\alpha=1}^2\Big\|{\partial\Vc\over \partial x_\alpha}-\Mc_{pq}(\GR)\Ge_\alpha\Big\|_{L^2(\wt{\o}^{pq}_\e)},\\
				&\hskip6cm\leq C\sum_{\alpha=1}^2\Big(\Big\|{\partial\Vc\over \partial x_\alpha}-\GR\Ge_\alpha\Big\|_{L^2(\wt{\o}^{pq}_\e)}+\|(\GR-\Mc_{pq}(\GR))\Ge_\alpha\|_{L^2(\wt{\o}^{pq}_\e)}\Big)\\
				&\|\Vc-\Mc_{pq}(\Vc)-\Mc_{pq}(\GR)(x'-\e\xi_{pq})\|_{L^2(\wt{\o}^{pq}_\e)} \\
				&\hskip 6cm\leq C\e \sum_{\alpha=1}^2\Big(\Big\|{\partial\Vc\over \partial x_\alpha}-\GR\Ge_\alpha\Big\|_{L^2(\wt{\o}^{pq}_\e)}+\|(\GR-\Mc_{pq}(\GR))\Ge_\alpha\|_{L^2(\wt{\o}^{pq}_\e)}\Big).
			\end{aligned}
			$$
			Then,  the trace inequality \eqref{IEQ2-2}$_2$ (after a simple change of variables) and estimates  \eqref{E0}$_{4,5}$--\eqref{FinEst} yield	
			\begin{multline*}
				\sum_{p=-N_\e}^{N_\e-1}\sum_{q=-N_\e}^{N_\e-1}\|\Vc+x_3\GR\Ge_3-\Mc_{pq}(\Vc)-\Mc_{pq}(\GR)(x-\e\xi_{pq})\|^2_{L^2(\partial\o^{pq}_\e\X(-\kappa\e,\kappa\e))}\\
				\leq C\e\|dist(\nabla v, SO(3))\|^2_{L^2(\O^B_\e)}.
			\end{multline*}
			As a consequence of  the above estimates and \eqref{eqfvm1}--\eqref{eqfvm2}, we  obtain
			\begin{multline*}
				\sum_{p=-N_\e}^{N_\e-1}\sum_{q=-N_\e}^{N_\e-1}\big\|\Mc_{pq}(\Vc)-\Ga_{pq}+\big(\Mc_{pq}(\GR)-\GR_{pq}\big)(x-\e\xi_{pq})\big\|^2_{L^2(\partial \o_\e^{pq}\X(-\kappa\e,\kappa\e))}\\
				\leq  C \e\big(\|dist(\nabla v, SO(3))\|^2_{L^2(\O_\e^M)}+\|dist(\nabla v, SO(3))\|^2_{L^2(\O_\e^B)}\big).
			\end{multline*} From the above inequality and \eqref{EQ55} after a simple change of variables, we get \eqref{EQ610}.\\[1mm]
			\noindent{\it Step 4.} We prove \eqref{KornP2}$_1$.\\[1mm]
			From \eqref{FinEst}$_1$ and \eqref{EQ51-0} we obtain
			\begin{equation}\label{EQ611}
				\sum_{p=-N_\e}^{N_\e-1}\sum_{q=-N_\e}^{N_\e-1}\big|\Mc_{pq}(\GR)-\GI_3|^2\e^2\leq \|\GR-\GI_3\|^2_{L^2(\o)}  \leq {C\over \e^{3}} \|dist(\nabla v, SO(3))\|^2_{L^2(\O^B_\e)}.
			\end{equation}
			Then, the above together with \eqref{EQ610}$_2$-\eqref{BGShell3.3}$_2$ lead to \eqref{KornP2}$_1$.\\[1mm]
			\noindent{\it Step 5.} We prove \eqref{KornP2}$_2$.\\[1mm]
			Using Poincar\'e inequality, we obtain
			$$\|v-\text{\bf id}\|^2_{L^2(\O_\e^{pq})}\leq C\|\nabla(v-\text{\bf id})\|^2_{L^2(\O_\e^{pq})}.$$
			Then, summing for all $p,q$ gives \eqref{KornP2}$_2$.
		\end{proof} 	
		
		We end this section with a estimate for the strain tensor $e(u)$ in $\O_\e^M$.
		\begin{lemma}\label{lem55}
			The strain tensor of the displacement $u=v-\text{\bf id}$ satisfies
			\begin{equation} \label{BGShell 4.9}
				\begin{aligned}
					\|e(u)\|_{L^2(\O^M_\e)} &\leq  C  \|dist(\nabla v, SO(3))\|_{L^2{(\O^M_\e)}}+{C\over \e}\|dist(\nabla v, SO(3))\|_{L^2(\O_\e^B)}\\
					& + {C \over \e^{3/2}} \|dist(\nabla v, SO(3))\|^2_{L^2{(\O^M_\e)}}+ {C\over \e^{7/2}}\|dist(\nabla v, SO(3))\|^2_{L^2(\O_\e^B)}.
				\end{aligned}
			\end{equation}
			The constant does not depend on $\e$.
		\end{lemma}	
		\begin{proof} 
			First, observe that
			$$2e(u)=\nabla v +(\nabla v)^T- 2\GI_3=(\nabla v-\GR_{pq})+(\nabla v-\GR_{pq})^T+(\GR_{pq}+\GR_{pq}^T-2\GI_3)\quad \text{a.e. in}\;\; \O^{pq}_\e.$$
			Hence, using estimate \eqref{BGShell3.3}$_1$, we get
			\begin{equation} \label{EQ515}
				\begin{aligned}
					\|\nabla v +(\nabla v)^T- 2\GI_3\|^2_{L^2(\O^{pq}_\e)}  & \leq C\|\nabla v-\GR_{pq}\|^2_{L^2(\O^{pq}_\e)}+ C\e^3|\GR_{pq}+\GR_{pq}^T-2\GI_3|^2_F\\
					&\leq C\|dist(\nabla v, SO(3))\|^2_{L^2(\O^{pq}_\e)} + C\e^3 | \GR_{pq}-\GI_3|_F^4,
				\end{aligned}
			\end{equation}
			since we have the equality $\GR_{pq}+\GR_{pq}^T-2\GI_3 = \GR_{pq}^T(\GR_{pq}-\GI_3)^2$ and $|\GR_{pq}|_F=\sqrt{3}$.\\
			Besides, from the estimates \eqref{EQ610}$_{2}$ and \eqref{EQ611}, we get
			$$
			\begin{aligned} 
				\sum_{p,\, q=-N_\e}^{N_\e-1} \e^3\big|\GR_{pq}-\GI_3\big|^2_F &\leq \sum_{p,\, q=-N_\e}^{N_\e-1} \e^3\left(|\Mc_{pq}(\GR)-\GI_3|^2+|\Mc_{pq}(\GR)-\GR_{pq}|^2\right)\\
				&\leq
				C\|dist(\nabla v, SO(3))\|^2_{L^2{(\O^M_\e)}}+ {C \over \e^{2}} \|dist(\nabla v, SO(3))\|^2_{L^2{(\O^B_\e)}}.
			\end{aligned}
			$$
			Observe that
			$$\sum_{p,\, q=-N_\e}^{N_\e-1}\big|\GR_{pq}-\GI_3\big|^4_F \leq \Big(\sum_{p,\, q=-N_\e}^{N_\e-1}\big|\GR_{pq}-\GI_3\big|^2_F\Big)^2.$$
			Hence,  the above estimates give
			$$
			\begin{aligned}
				\sum_{p,\, q=-N_\e}^{N_\e-1} \e^3\big|\GR_{pq}-\GI_3\big|^4_F & \leq {1\over \e^3}\Big(\sum_{p,\, q=-N_\e}^{N_\e-1} \e^3\big|\GR_{pq}-\GI_3\big|^2_F\Big)^2\\
				&\leq {C\over \e^3}\|dist(\nabla v, SO(3))\|^4_{L^2{(\O^M_\e)}}+ {C \over \e^{7}} \|dist(\nabla v, SO(3))\|^4_{L^2{(\O^B_\e)}}.
			\end{aligned}
			$$
			The sum of the inequalities \eqref{EQ515} and taking into account the above estimate lead to \eqref{BGShell 4.9}.
		\end{proof}	
		\section{Decomposition of displacements}\label{Sec06}
In this section, a new splitting of the displacements corresponding to the deformations defined on $\O_\e$ is derived. The primary objective of this new splitting is to reduce the contrast in the estimates in the region $\O_\e^M$. This section concludes with the decomposition of plate displacements through Kirchhoff-Love displacement and the presentation of Korn-type inequalities for the respective terms in the decomposition.
		
\subsection{New splitting of the displacement}
	Let the displacement $u=v-\text{\bf id}$ be defined in $\O_\e$. In Lemmas \ref{EFSTB} and \ref{lem55}, upper bounds for the norms of its strain tensors have been obtained. Given the high contrast in these estimates, a new splitting of $u$ is introduced below.
	
	The displacement is denoted by $\Gu$, defined as $\Gv-\text{\bf id}$, and we define $\Gu^M$ as $u-\Gu$.
		\begin{lemma} The displacements $\Gu$ and $\Gu^M$ satisfy
			\begin{equation} \label{Eq518}
				\begin{aligned}
					& \Gu=0 \quad \text{a.e. in}\;\;\;  D(O,R)\X (-r,r),\qquad \Gu^M=0\quad \hbox{a.e. in } \O^M_\e\cap \big(D(O,R)\X (-r,r)\big),\\
					&\|e(\Gu)\|_{L^2(\O_\e)} \leq  C  \|dist(\nabla v, SO(3))\|_{L^2{(\O^B_\e)}} + {C \over \e^{5/2}} \|dist(\nabla v, SO(3))\|^2_{L^2{(\O^B_\e)}},\\
					&\|e(\Gu^M)\|_{L^2(\O_\e^M)}\leq  C\|dist(\nabla v, SO(3))\|_{L^2{(\O^M_\e)}} + {C}  \|dist(\nabla v, SO(3))\|_{L^2{(\O^B_\e)}}  \\
					&\hskip 30mm + {C \over \e^{3/2}} \|dist(\nabla v, SO(3))\|^2_{L^2{(\O^M_\e)}}+{C\over \e^{7/2}}\|dist(\nabla v, SO(3))\|^2_{L^2(\O_\e^B)},\\
					&\|\Gu^M\|_{L^2(\O^M_\e)}+\e \|\nabla \Gu^M\|_{L^2(\O^M_\e)} \leq  C \e \|e(\Gu^M)\|_{L^2(\O_\e^M)}.
				\end{aligned}
			\end{equation} 
			The constants do not depend on $\e$. 
		\end{lemma}	
		
		\begin{proof} 
			The equalities \eqref{Eq518}$_{1,2}$ are direct consequence of the boundary condition on $\O_\e$ and the definition of $\Gu^M$. From the estimates \eqref{EQ51-0} and \eqref{Eq II.3.10}, we obtain \eqref{Eq518}$_{3}$.\\ 
			Since $\Gu^M=u-\Gu$, we first obtain
			$$\|e(\Gu^M)\|_{L^2(\O_\e)}=\|e(\Gu^M)\|_{L^2(\O_\e^M)}\leq \|e(u)\|_{L^2(\O^M_\e)}+\|e(\Gu)\|_{L^2(\O_\e)}.$$ Then, the estimates \eqref{Eq II.3.10} and \eqref{BGShell 4.9} give \eqref{Eq518}$_{4}$.\\
			We prove \eqref{Eq518}$_{5}$ to do that, we first apply the $3$D-Korn inequality in the fixed domain $\wt{Y}_\kappa$ with the displacement $\Gw_{pq}(y)=\Gu^M(\e\xi_{pq}+\e y)$, to attain the estimate
			$$\begin{aligned}
				\|\Gw_{pq}\|_{L^2(\wt{Y}_\kappa)}+ \|\nabla_y \Gw_{pq}\|_{L^2(\wt{Y}_\kappa)}\leq C\|e(\Gw_{pq})\|_{L^2(\wt{Y}_\kappa)}.
			\end{aligned}$$ The constant does not depend on $(p,q)\in \{-N_\e,\ldots,0,\ldots,N_\e-1\}^2$, it only depends on $\kappa$.\\
			Applying the change of variable, since
			$$ \nabla_y \Gw_{pq}(y)=\e  \nabla\Gu^M(\e\xi_{pq}+\e y) \qquad \hbox{for a.e. } y\in \wt{Y}_\kappa$$ 
			we obtain 
			$$ \|\Gu^M\|^2_{L^2(\e(\xi_{pq}+\wt{Y}_\kappa))}+\e^2\|\nabla \Gu^M\|^2_{\e(\xi_{pq}+\wt{Y}_\kappa)} \leq C\e^2\|e(\Gu^M)\|^2_{L^2(\e(\xi_{pq}+\wt{Y}_\kappa))}.$$
			So, adding for all $p,q$ yields
			\begin{equation}\label{Eq519}
				\|\Gu^M\|^2_{L^2(\O_\e)}+\e^2\|\nabla \Gu^M\|^2_{L^2(\O_\e)} \leq C\e^2\|e(\Gu^M)\|^2_{L^2(\O_\e)}.
			\end{equation} From the above estimate, we obtain the estimate \eqref{Eq518}$_{5}$. 	
		\end{proof}
		We are in a position to use the results of \cite{GKL} and to decompose the displacement $\Gu$ of the plate $\O_\e$.

		\subsection{Decomposition via a Kirchhoff-Love displacement}\label{Sec8}
In this subsection, the decomposition of every displacement of the plate $\O_\e$ is decomposed as the sum of a Kirchhoff-Love displacement and a residual part.
		Below we recall a result proven in \cite{GKL}.
		\begin{lemma}[Theorem 6.1 in \cite{GKL}]\label{DeComKL1} Every displacement belonging to $H^1(\O_\e)^3$ can be decomposed into the sum of a Kirchhoff-Love displacement and a residual part
			\begin{equation}\label{PD1}
				\Gu= U_{KL}(x) + \fu(x),\qquad \text{  a.e. in }\;\; x\in\O_\e,
			\end{equation} 
			where $U_{KL}$ is a Kirchhoff-Love displacement given by 
			$$
			U_{KL} (x) = \left(\begin{aligned}
				&\Uc_1 (x') - x_3 {\partial \Uc_3 \over \partial x_1}(x')\\
				&\Uc_2 (x') - x_3 {\partial \Uc_3 \over \partial x_2}(x')\\
				&\hspace{0.9cm}\Uc_3 (x')
			\end{aligned}\right)
			\;\; \text{for a.e.}\;\; x=(x', x_3) \in \O_\e.
			$$
			We have
$$
				\Uc_m=\Uc_1 \Ge_1 +\Uc_2 \Ge_2 \in H^1(\o)^2,\quad \Uc_3\in H^2(\o),\qquad \fu\in H^1(\O_\e)^3
$$  and the following estimates
			\begin{equation}\label{PD2}
				\begin{aligned}
					\|e_{\alpha \beta}(\Uc_{m})\|_{L^2(\o)} &\leq {C \over \e^{1/2}} \|e(\Gu)\|_{L^2(\O_\e)},\\
					\|D^2(\Uc_3)\|_{L^2(\o)} &\leq {C \over \e^{3/2} }\|e(\Gu)\|_{L^2(\O_\e)},\\
					\|\fu\|_{L^2(\O_\e)} +\e\|\nabla\fu\|_{L^2(\O_\e)} &\leq C\e \|e(\Gu)\|_{L^2(\O_\e)}.
				\end{aligned}
			\end{equation}
			The constants do not depend on $\e$.
		\end{lemma}
		\begin{remark}
			Since our plate $\O_\e$ is clamped on the part $\Gamma_\e=\gamma\X(-r,r)$ i.e. $\Gu=0$ a.e. on $\Gamma_{r}$, then we have
$$
				\Uc=0,\quad \nabla\Uc_3=0,\quad\text{a.e. on}\quad \gamma,\quad \text{and}\quad \fu=0\quad\text{a.e. on}\quad \Gamma_r.
$$
		\end{remark}
		Using the above Lemma and Remark, 
		\begin{lemma}[Korn-type inequalities] One has
			\begin{equation}\label{KornPUD1}
				\begin{aligned}
					\|\Gu_1\|_{L^2(\O_\e)}+\|\Gu_2\|_{L^2(\O_\e)}+\e\|\Gu_3\|_{L^2(\O_\e)}&\leq C\|e(\Gu)\|_{L^2(\O_\e)},\\
					\sum_{\alpha,\beta=1}^2\left\|{\partial\Gu_\beta\over \partial x_\alpha}\right\|_{L^2(\O_\e)}+\left\|{\partial\Gu_3\over \partial x_3}\right\|_{L^2(\O_\e)}&\leq C\|e(\Gu)\|_{L^2(\O_\e)},\\
					\sum_{\alpha=1}^2\left(\left\|{\partial\Gu_3\over \partial x_\alpha}\right\|_{L^2(\O_\e)}+\left\|{\partial\Gu_\alpha\over \partial x_3}\right\|_{L^2(\O_\e)}\right)&\leq {C\over\e}\|e(\Gu)\|_{L^2(\O_\e)}.
				\end{aligned}
			\end{equation}
			The constants do not depend on $\e$.
		\end{lemma}

\section{Case 1: Asymptotic behavior of the composite plate due to re-scaling of forces}\label{Sec07}
		In this section, the main results of Case 1 are presented. The section begins with the definition of the non-linear elasticity problem and the imposition of necessary assumptions on the right-hand side forces. The focus then shifts to the analysis of the asymptotic behavior of the Green St. Venant's strain tensor. To achieve this, two periodic unfolding method operators are introduced.
The key result of Case 1 is proved in Theorem \ref{MainConR}, which details the asymptotic behavior of the minimization sequence $\ds\left\{\frac{m_\e}{\e^5}\right\}_\e$. The proof of this theorem relies on techniques involving $\Gamma$-convergence.
In addition to describing the asymptotic behavior, this section also addresses the linear elastic cell problems and establishes the existence of a minimizer for the homogenized limit energy.
		
		\subsection{The non-linear elasticity problem in Von-K\'arm\'an regime}\label{Ssec71}
  In this section, the elasticity problem is introduced, and the initial setup is established.
Let us denote $\wh W_\e$ the local elastic energy density, then the total elastic energy $\GJ_\e(v)\footnote{For later convenience, the term $\int_{\O_\e} {f_\e}(x)\cdot I_d(x)dx$ is added to the usual standard energy, indeed this does not affect the minimizing problem for $\GJ_\e$.}$ over $\GV_\e$ is given by
		\begin{equation}\label{stpr}
			\GJ_{\e}(v)= 	\int_{\O_{\e}} \wh W_\e\big(x,\nabla v\big)\, dx- \int_{\O_{\e}} f_{\e}(x)\cdot(v(x)-\text{\bf id}(x))\, dx,\qquad
			\forall v\in \GV_\e.
		\end{equation}
		The local density energy $\wh W_\e\, :\ \,  \O_\e\times \GM_3 \longrightarrow [0,+\infty]$ is given by
		\begin{equation}\label{Eq62}
			\wh{W}_\e(F)=\widehat{W}_\e(x , F)=\left\{
			\begin{aligned}
				&\begin{aligned}
					& \e^{2} Q\Big(\frac{x}{\e} ,{\bf E}(F)\Big) &&\hbox{if}\;\;  \hbox{det}(F)> 0,\\
					& +\infty  &&\hbox{if}\;\;  \hbox{det}(F)\leq 0,
				\end{aligned}&&\quad\hbox{for a.e. }  x \in \O^{M}_\e,\\ 
				&\begin{aligned}
					&  Q\Big(\frac{x}{\e} , {\bf E}(F)\Big) &&\hbox{if}\;\;  \hbox{det}(F)> 0,\\
					& +\infty  &&\hbox{if}\;\;  \hbox{det}(F)\leq 0,
				\end{aligned}&&\quad\hbox{for a.e. }  x \in \O^{B}_\e.\\ 
			\end{aligned}\right.
		\end{equation}
	 The quadratic form $Q$ is defined by
		$$Q\left({x\over \e},S\right)= a^\e_{ijkl}\left({x\over \e}\right)S_{ij}S_{kl} \quad \hbox{for a.e. $x\in \O_\e$ and for all $S\in \GS_3$},$$ 
		where
		$$a^\e_{ijkl}\left({x\over \e}\right)=a_{ijkl}\left(\left\{{x'\over \e}\right\}, {x_3\over \e}\right),\quad \text{a.e.}\quad x\in\O_\e.$$
		Here $a_{ijkl} \in L^{\infty}(\Yc)$ and are  periodic with respect to $\Ge_1$ and $\Ge_2$ in the undeformed configuration.
		Moreover, the Hooke's tensor $a$ is symmetric, i.e., $a_{ijkl}= a_{jikl} = a_{klji}$. Also it is positive definite and satisfies that there exists a $c_0>0$ such that
		\begin{equation}\label{QPD}
			c_0\,S_{ij}S_{ij}\leq a_{ijkl}(y)S_{ij}S_{kl}\quad \hbox{for a.e. $y\in \Yc$ and for all $S\in \GS_3$}.
		\end{equation}
%		We also define the stress tensor as
%		$$\sigma^{\e}(u)=a_{ijkl}^\e e_{kl}(u),\qquad \forall \;\; u\in \GU_\e,$$
%		where the coefficients $a_{ijkl}^\e$ of the Hooke's tensor are given by
%		$$ a^\e_{ijkl}(x)=a_{ijkl}\left(\left\{{x'\over\e}\right\},{x_3\over \e}\right),\quad \text{a.e.}\quad x\in \O_\e.$$
			So, we have
		\begin{equation}\label{eq64}
			c_0|F^{T}F-\GI_3|^2_F\leq  Q\Big(y , {1\over 2}\big (F^{T}F -\GI_3\big)\Big)=Q\left(y,{\bf E}(F)\right)\qquad \hbox{for a.e. } x\in \Yc.
		\end{equation} 
		 We recall that (see e.g. \cite{BGRod} or \cite{Friesecke06})
		\begin{equation}\label{SymCon}
		dist(F, SO(3))\leq |F^TF-\GI_3|_F \quad \text{for all}\;\;F \in \GM_3\quad\hbox{such that}\;\; \det(F)>0.
		\end{equation}
		\begin{remark}
          From the above assumptions on the local energy, it can be observed that:
			\begin{itemize}
				\item 			The Hooke's tensor $a$ satisfy the following
  					$$ a_{ijkl}^\e(y)=a_{ijkl}(\{y'\},y_3)=a_{ijkl}(y),\quad\forall\; y\in\Yc.$$
				\item The local energy $\wh{W}_\e$, also known as stored energy, is a measurable integrand that is $[0,1)^2$ periodic in the first two components of the first variable such that
				$$\begin{aligned}
					\wh{W}_\e(\GI_3)= 0,\quad \wh{W}_\e(F)\geq c_0 dist^2(F,SO(3)) \quad \forall F\in \GM_3.
%					{\partial^2 Q\over \partial F_{ij}\partial F_{kl}}\left({x\over \e},\GI_3\right)=
%					{\clm\left.
%					 \begin{aligned}
%						a^\e_{ijkl}\left({x\over \e}\right)\quad \forall x\in\O_\e.
%%						\e^2 a^\e_{ijkl}\left({x\over \e}\right)\quad \forall x\in\O^M_\e.
%				    \end{aligned}\right.}
				\end{aligned}$$
			\end{itemize}
		\end{remark}  
		As a consequence of \eqref{Eq62}--\eqref{SymCon} we have 
		$$
		\begin{aligned}
			&c_0 \big(\e^2\|dist(\nabla v, SO(3))\|^{2}_{L^{2}(\O^M_\e)}+\|dist(\nabla v, SO(3))\|^{2}_{L^{2}(\O^B_\e)}\big)\\
			& \hskip 10mm \leq c_0\big(\e^2\|\nabla v(\nabla v)^T-\GI_3\|^2_{L^2(\O^M_\e)}+\|\nabla v(\nabla v)^T-\GI_3\|^2_{L^2(\O^B_\e)}\big)\leq   \int_{\O_{\e}} \wh W_\e\big(\cdot,\nabla v\big)\, dx.
		\end{aligned}
		$$
		Note that $\GJ_\e(\text{\bf id})=0$, so we get
		\begin{equation}\label{GOWEQ57}
			c_0 \big(\e^2\|dist(\nabla v, SO(3))\|^{2}_{L^{2}(\O^M_\e)}+\|dist(\nabla v, SO(3))\|^{2}_{L^{2}(\O^B_\e)}\big)-\int_{\O_\e}f_\e\cdot(v-\text{\bf id})dx \leq {\GJ}_{\e}(v).
		\end{equation}
		Therefore, from now on the only deformations are chosen such that $\GJ_\e(v)\leq \GJ_\e(\text{\bf id})= 0$.
		
		Then the non-linear elasticity problem reads 
		\begin{equation}\label{MinPro}
			m_{\e}=	\inf_{v\in \GV_{\e}} \GJ_{\e}(v)\leq 0.
		\end{equation}
		\begin{remark}\label{Rem2}
	As a classical example of local elastic energy satisfying the above assumptions, we mention the  St. Venant-Kirchhoff's  material  for which
				$$
			\widehat{W}_\e(x , F)=\left\{
				\begin{aligned}
				&\begin{aligned}
					&{\e^2\lambda\over 8}\big(tr(F^TF-\GI_3) \big)^2+{\e^2\mu\over 4}\left|F^T F-\GI_3\right|^2_F&&\hbox{if}\quad \det(F)>0,\\
					&+\infty &&\hbox{if}\quad\det(F)\le 0,
				\end{aligned}&& \quad\hbox{for a.e. }  x \in \O^{M}_\e,\\ 
				&\begin{aligned}
					&{\lambda\over 8}\big(tr(F^TF-\GI_3) \big)^2+{\mu\over 4}\left|F^T F-\GI_3\right|^2_F&&\hbox{if}\quad \det(F)>0,\\
					&+\infty &&\hbox{if}\quad\det(F)\le 0,
				\end{aligned}&& \quad\hbox{for a.e. }  x \in \O^{B}_\e.
				\end{aligned}
				\right.
				$$ 
		\end{remark}
\subsection{Assumption on the right-hand side forces}\label{Ssec72}
		The forces have to admit a certain scaling with respect to $\e$. For the composite plate the required forces are of the type
		\begin{equation}\label{AFB1}
			f_{\e}=\e^2f_1\Ge_1+\e^2f_2\Ge_2+\e^3f_3\Ge_3\quad \text{a.e. in}\;\; \O_\e,
		\end{equation}
		with $f\in L^2(\o)^3$. In order to obtain at the limit a von-K\'arm\'an model, the applied forces must satisfy a condition 
		\begin{equation}\label{AFB2}
			\|f\|_{L^2(\o)}\leq C^*.
		\end{equation}
		The constant depends on the reference cell $\Yc$, the mid-surface $\o$ of the structure and the local elastic energy $\wh{W}_\e$ (see Lemma \ref{GOWL54}).\\
		The scaling of the force gives rise to the order of the energy in the elasticity problem. We prove this in the lemma below.
		\begin{lemma}\label{GOWL54}
			Let $v\in \GV_{\e}$ be a deformation such that $\GJ_{\e}(v)\leq 0$. Assume the forces fulfill \eqref{AFB1}, there exists a constant $C^*$ independent of $\e$ and the applied forces such that, if $\|f\|_{L^2(\o)}\le C^*$   (see \eqref{EQ74}) one has 
			\begin{equation}\label{EFBM}
				\begin{aligned}
					\|dist(\nabla v, SO(3))\|_{L^{2}(\O^B_\e)}+\e\|dist(\nabla v, SO(3))\|_{L^{2}(\O^M_\e)}&&\leq C\e^{5/2}\|f\|_{L^2(\o)},\\
					\|\nabla v(\nabla v)^T-\GI_3\|_{L^2(\O^B_\e)}+\e \|\nabla v(\nabla v)^T-\GI_3\|_{L^2(\O^M_\e)}&&\leq C\e^{5/2}\|f\|_{L^2(\o)}.
				\end{aligned}
			\end{equation}	
			The constants do not depend on $\e$.
		\end{lemma}
		\begin{proof}
			Using \eqref{GOWEQ57} gives rise to the estimate
			$$
			c_0\Big(\|dist(\nabla v, SO(3))\|^{2}_{L^{2}(\O^B_\e)}+\e^2\|dist(\nabla v, SO(3))\|^{2}_{L^{2}(\O^M_\e)}\Big)\leq\left|\int_{\O_\e} f_\e\cdot(v-\text{\bf id})dx\right|.
			$$
			Then, with \eqref{AFB1}, we obtain 
			$$ \left|\int_{\O_\e} f_\e\cdot(v-\text{\bf id})dx\right| \leq \left|\int_{\O^B_\e} f_\e\cdot(v-\text{\bf id})dx\right| +\left|\int_{\O^M_\e} f_\e\cdot(v-\text{\bf id})dx\right|,$$
			since $u=\Gu$ a.e. on $\O_\e^B$ and the estimates \eqref{KornPUD1}$_1$ imply 
			$$\begin{aligned}
				\left|\int_{\O^B_\e} f_\e\cdot(v-\text{\bf id})dx\right| &\leq \e^{5/2} \|f\|_{L^2(\o)}\left(\|u_1\|_{L^2(\O_\e^B)}+\|u_2\|_{L^2(\O_\e^B)}+\e\|u_3\|_{L^2(\O^B_\e)}\right)\\
				&= \e^{5/2} \|f\|_{L^2(\o)}\left(\|\Gu_1\|_{L^2(\O_\e^B)}+\|\Gu_2\|_{L^2(\O_\e^B)}+\e\|\Gu_3\|_{L^2(\O^B_\e)}\right)\\
				&\leq \e^{5/2} \|f\|_{L^2(\o)}\left(\|\Gu_1\|_{L^2(\O_\e)}+\|\Gu_2\|_{L^2(\O_\e)}+\e\|\Gu_3\|_{L^2(\O_\e)}\right)\\
				&\leq C\e^{5/2}\|f\|_{L^2(\o)}\|e(\Gu)\|_{L^2(\O_\e)}.
			\end{aligned}$$
			Now, using the fact that $u=\Gu+\Gu^M$ and  the estimates \eqref{Eq518}--\eqref{KornPUD1}$_1$, we get
			$$\begin{aligned}
				\left|\int_{\O^M_\e} f_\e\cdot(v-\text{\bf id})dx\right|
				&\leq \e^{5/2} \|f\|_{L^2(\o)}\left(\|u_1\|_{L^2(\O_\e^M)}+\|u_2\|_{L^2(\O_\e^M)}+\e\|u_3\|_{L^2(\O^M_\e)}\right)\\
				&\leq C\e^{5/2}\|f\|_{L^2(\o)}\left(\|\Gu_1\|_{L^2(\O_\e^M)}+\|\Gu_2\|_{L^2(\O_\e^M)}+\e\|\Gu_3\|_{L^2(\O^M_\e)}+\|\Gu^M\|_{L^2(\O^M_\e)}\right)\\
				&\leq C\e^{5/2}\|f\|_{L^2(\o)}\left(\|e(\Gu)\|_{L^2(\O_\e)}+\e\|e(\Gu^M)\|_{L^2(\O^M_\e)}\right).
			\end{aligned}$$
			Then, from the above inequalities and the estimates \eqref{Eq518} we obtain
			\begin{multline*}
				\left|\int_{\O_\e} f_\e\cdot(v-\text{\bf id})dx\right| \leq C\e^{5/2}\|f\|_{L^2(\o)}\Big(\|dist(\nabla v, SO(3))\|_{L^2(\O_\e^B)}+\e\|dist(\nabla v, SO(3))\|_{L^2(\O_\e^M)}\Big)\\
				+C\|f\|_{L^2(\o)}\Big(\e^2 \|dist(\nabla v, SO(3))\|^2_{L^2{(\O^M_\e)}}+\|dist(\nabla v, SO(3))\|^{2}_{L^{2}(\O^B_\e)}\Big).
			\end{multline*}
			Then, we have
			$$
			\begin{aligned}
				&\|dist(\nabla v, SO(3))\|^{2}_{L^{2}(\O^B_\e)}+\e^2\|dist(\nabla v, SO(3))\|^{2}_{L^{2}(\O^M_\e)}\\
				&\qquad \leq C_0\e^{5/2}\|f\|_{L^2(\o)}\Big(\|dist(\nabla v, SO(3))\|_{L^2(\O_\e^B)}+\e\|dist(\nabla v, SO(3))\|_{L^2(\O_\e^M)}\Big)\\
				&\qquad +C_0\|f\|_{L^2(\o)}\Big(\e^2\|dist(\nabla v, SO(3))\|^2_{L^2{(\O^M_\e)}}+\|dist(\nabla v, SO(3))\|^{2}_{L^{2}(\O^B_\e)}\Big).
			\end{aligned}
			$$
			So, we assume that
			\begin{equation}\label{EQ74}
				C_0\|f\|_{L^2(\o)}\leq {1\over 2}.
			\end{equation}
			From the above assumption, we get
			$$
			\begin{aligned}	
				\|dist(\nabla v, SO(3))\|^{2}_{L^{2}(\O^B_\e)}&+\e^2\|dist(\nabla v, SO(3))\|^{2}_{L^{2}(\O^M_\e)}\\
				&\leq 2C_0\e^{5/2}\|f\|_{L^2(\o)}\Big(\|dist(\nabla v, SO(3))\|_{L^2(\O_\e^B)}+\e\|dist(\nabla v, SO(3))\|_{L^2(\O_\e^M)}\Big)\\
				&\leq 3C_0\e^{5/2}\|f\|_{L^2(\o)}\sqrt{\|dist(\nabla v, SO(3))\|^2_{L^2(\O_\e^B)}+\e^2\|dist(\nabla v, SO(3))\|^2_{L^2(\O_\e^M)}}.
			\end{aligned}
			$$
			This inequality gives \eqref{EFBM}$_1$. 
			
			Since, we have assumed that the deformations $v\in \GV_\e$ satisfies $\GJ_\e(v)\leq 0$, then the estimate \eqref{EFBM} can be employed to arrive at 
			\begin{equation}\label{Eq75}
				\begin{aligned}
					c_0\big(\big\|(\nabla v)^T \nabla v -\GI_3\|^2_{L^2(\O^B_\e)}+\e^2\|(\nabla v)^T\nabla v -\GI_3 \big\|^2_{L^2(\O^M_\e)}\big) 
					\leq \int_{\O_\e} \wh{W}(x,\nabla v(x))dx\\
					\leq \int_{\O_\e} f_\e(x)\cdot(v-\text{\bf id})(x)dx \leq C\e^5\|f\|^2_{L^2(\o)},
				\end{aligned}
			\end{equation}
			which in turn leads to
			\begin{equation}\label{Eq76}
				\|\nabla v (\nabla v)^T-\GI_3\|_{L^2(\O^B_\e)}+\e\|(\nabla v)^T\nabla v -\GI_3\|_{L^2(\O^M_\e)}\leq C\e^{5/2}.
			\end{equation}
			This inequality gives \eqref{EFBM}$_2$.	
		\end{proof}	
		
		As a consequence, there exists a constant $c$ strictly negative, and independent of $\e$ such that 
$$
			c\e^5 \leq \GJ_\e(v) \leq 0.
$$
		Recalling that $m_\e=\inf_{v\in \GV_{\e}} \GJ_\e(v)$\footnote{It is still an open problem whether there is a minimizer of $\GJ_\e$ over $\GV_\e$. } yields 
$$
			c\leq {m_\e\over \e^5} \leq 0.
$$
	Our next goal is to give the asymptotic behavior of the re-scaled sequence $\ds\left\{{m_\e \over \e^5}\right\}_\e$ and to characterize its limit as a minimum of the functional.

		\subsection{Asymptotic behavior of the Green St. Venant's strain tensors}\label{ABGSVST}
		In this subsection, we consider a sequence $\{v_\e\}_\e$ of deformation in $\GV_\e$ with $\GJ_{\e}(v_\e)\leq 0$ satisfying
		\begin{equation}\label{AFLP1}
			\|dist(\nabla v_\e, SO(3))\|_{L^{2}(\O^B_\e)}+\e\|dist(\nabla v_\e, SO(3))\|_{L^{2}(\O^M_\e)}\leq C\e^{5/2}.
		\end{equation}
		We give asymptotic behavior of the sequence $\{\nabla v_\e(\nabla v_\e)^T-\GI_3\}_\e$.
		From the above assumption and using the estimates \eqref{E0}, \eqref{FinEst}, \eqref{EQ448}, \eqref{KornP2} and \eqref{Eq76} we deduce the following estimates 
		\begin{equation}\label{EWA1}
			\begin{aligned}
				\e\|\nabla v_\e(\nabla v_\e)^T-\GI_3\|_{L^2(\O^M_\e)}+	\|\nabla v_\e(\nabla v_\e)^T-\GI_3\|_{L^2(\O^B_\e)}&\leq C\e^{5/2}.
			\end{aligned}
		\end{equation} 
		Moreover, from the Lemmas \ref{EFSTB} and \ref{lem55}, we have
		\begin{equation}\label{EWA3}
			\|e(u_\e)\|_{L^2(\O^B_\e)}\leq C\e^{5/2}\quad\text{and}\quad\|e(u_\e)\|_{L^2(\O^M_\e)}\leq C\e^{3/2}.
		\end{equation}
		From the estimate \eqref{Eq518}$_{3,4}$, one has 
		\begin{equation}\label{EWA2}
			\|e(\Gu_\e)\|_{L^2(\O_\e)}\leq C\e^{5/2}\quad \text{and}\quad \|e(\Gu^M_\e)\|_{L^2(\O_\e)}\leq C\e^{3/2}.
		\end{equation}
		
		Now, we give the estimates of $\Gu^M_\e$, $\Gu_\e$ and the terms of its decomposition from \eqref{AFLP1}. And using the estimates \eqref{Eq518}, \eqref{PD2} and \eqref{KornPUD1} one has
		\begin{equation}\label{EWAgu}
			\begin{aligned}
				\|\Gu_\e^M\|_{L^2(\O^M_\e)}+\e \|\nabla \Gu_\e^M\|_{L^2(\O^M_\e)} &\leq C\e^{5/2},\\
				\|e_{\alpha \beta}(\Uc_{m})\|_{L^2(\o)}+\e\|D^2(\Uc_{3,\e})\|_{L^2(\o)}&\leq C\e^{2},\\
				\|\fu_\e\|_{L^2(\O_\e)} +\e\|\nabla\fu_\e\|_{L^2(\O_\e)} &\leq C\e^{7/2},\\
				\|\Gu_{1,\e}\|_{L^2(\O_\e)}+\|\Gu_{2,\e}\|_{L^2(\O_\e)}+\e\|\Gu_{3,\e}\|_{L^2(\O_\e)}&\leq C\e^{5/2},\\
				\|\nabla\Gu_\e\|_{L^2(\O_\e)} &\leq C\e^{3/2}.
			\end{aligned}
		\end{equation} 
		We also have from the estimates \eqref{EQ448} and \eqref{KornP2}
		\begin{equation}\label{Eq106}
			\begin{aligned}
				\|u_\e\|_{L^2(\O_\e)}+	\|\nabla u_\e\|_{L^2(\O_\e)}\leq C\e^{3/2}.
			\end{aligned}
		\end{equation}
		The constants do not depend on $\e$.

		\subsection{Preliminaries on the unfolding operators}
   The convergences are achieved through the utilization of the re-scaling unfolding operator $\Pi_\e$ for homogenization and dimension reduction in $\O_\e$ and the unfolding operator $\Te$ for homogenization in $\o$.
   
   Below, the definition of the periodic re-scaling unfolding operator and the unfolding operator for functions defined in $\O_\e$ and $\o$ are recalled. Further details about the properties of these operators can be found in \cite{CDG}.
		\begin{definition}
			For every measurable function $\psi$ on $\O_\e$ the re-scaling unfolding  operator $\Pi_\e$ is defined by 
			$$ \Pi_\e(\psi) (x',y) \doteq \psi \left(2\e\left[{x'\over 2\e}\right] + \e y\right) \qquad \hbox{for a.e. }\; (x',y)\in  \o\times \Yc,\qquad \Yc\doteq (0,1)^2\times(-\kappa,\kappa).$$
			For every measurable function $\phi$ on $\o$ the unfolding operator $\Te$ is defined by 
			$$ \Te(\phi) (x',y') \doteq \phi \left(2\e\left[{x'\over 2\e}\right] + \e y'\right) \qquad \hbox{for a.e. }\; (x',y')\in  \o\times Y,\qquad Y\doteq (0,1)^2.$$
		\end{definition} 
		In what follows, we recall some of the properties and inequalities related to the re-scaling unfolding  operator and the unfolding operator, for proofs and more details see \cite{CDG}. \\ The re-scaling unfolding operator $\Pi_\e$ is a  continuous linear operator  from $L^2(\O_\e)$ into $L^2(\o\X \Yc)$ which satisfies
		\begin{equation}\label{EQ722}
			\|\Pi_\e(\psi)\|_{L^2(\o\X\Yc)}\leq{C\over \sqrt{\e}}\|\psi\|_{L^2(\O_\e)}\quad \text{ for every}\quad \psi\in L^2(\O_\e),
		\end{equation}
		and the unfolding operator is also a continuous linear operator from $L^2(\o)$ into $L^2(\o\X Y)$ which satisfies
		$$\|\Te(\phi)\|_{L^2(\o\X Y)}\leq C\|\phi\|_{L^2(\o)}\quad\text{for every}\quad \phi\in L^2(\o).$$
		The constants $C$ do not depend on $\e$.
		\begin{remark}
			Here, we note that for functions defined in $\o$, we have
			$$\Pi_\e(\phi)(x',y)=\Te(\phi)(x',y')=\phi \left(2\e\left[{x'\over 2\e}\right] + \e y'\right)\quad \text{for every}\quad \phi\in L^2(\o).$$
		\end{remark}
		Moreover, for every $\psi\in H^1(\O_\e)$ we have (see \cite[Proposition 1.35]{CDG})
		$$ \nabla_y\Pi_\e(\psi)(x',y)=\e\Pi_\e(\nabla \psi)(x',y) \quad \hbox{a.e. in}\quad \o\X\Yc.$$
		Then, from the above inequality along with the estimates \eqref{EWAgu}--\eqref{Eq106}, we have
		\begin{equation}\label{EWRUO1}
			\begin{aligned}
				\|\Pi_\e(\Gu_\e^M)\|_{L^2(\o\X\Yc)}+ \|\nabla_y\Pi_\e( \Gu_\e^M)\|_{L^2(\o\X\Yc)} &\leq C\e^{2},\\
				\|\Pi_\e(\fu_\e)\|_{L^2(\o\X\Yc)} +\|\nabla_y\Pi_\e(\fu_\e)\|_{L^2(\o\X\Yc)}&\leq C\e^{3},\\
				\|\Pi_\e(\Gu_{1,\e})\|_{L^2(\o\X\Yc)}+\|\Pi_\e(\Gu_{2,\e})\|_{L^2(\o\X\Yc)}+\e\|\Pi_\e(\Gu_{3,\e})\|_{L^2(\o\X\Yc)}&\leq C\e^{2},\\
				\|\Pi_\e(\nabla\Gu_\e)\|_{L^2(\o\X\Yc)}&\leq C\e,\\
				\|\Pi_\e(u_\e)\|_{L^2(\o\X \Yc)}+\|\Pi_\e(\nabla u_\e)\|_{L^2(\o\X \Yc)}&\leq C\e,
%				\|\Pi_\e(u_\e)\|_{L^2(\o\X Y_\kappa^B)}+\|\Pi_\e(\nabla u_\e)\|_{L^2(\o\X Y_\kappa^B)}&\leq C\e^2,
			\end{aligned}
		\end{equation}
		where we denote
		$$
		\begin{aligned}
			&Y^M_\kappa\doteq Y_\kappa\X(-\kappa,\kappa),\qquad Y^B_\kappa\doteq\Yc\setminus \overline{Y^M_\kappa},\\
			&\GH^1(Y^M_\kappa)\doteq\Big\{\phi\in H^1(Y^M_\kappa)\;|\; \phi=0\;\;\hbox{a.e. on }\; \partial Y_\kappa\X(-\kappa,\kappa)\Big\}.
		\end{aligned}
		$$
		
		\subsection{Limit behavior for the Green St. Venant's strain tensor using the unfolding re-scaling operator} 
		In this subsection, the asymptotic behavior of the unfolded sequences of deformations and displacements defined on the region $\O_\e$ are discussed. 
		
		\begin{lemma}\label{CFuu1}
			There exist a subsequence of $\{\e\}$, still denoted by $\{\e\}$, and $\Uc_m=\Uc_1\Ge_1+\Uc_2\Ge_2 \in H^1(\o)^2$ and $\Uc_3\in H^2(\o)$ such that
			\begin{equation}\label{CFU1}
				\begin{aligned}
					{1\over \e^2}\Uc_{m,\e}&\rightharpoonup \Uc_m\quad \text{weakly in}\quad H^1(\o)^2\quad\hbox{and strongly in}\quad L^2(\o)^2,\\
					{1\over \e}\Uc_{3,\e}&\rightharpoonup \Uc_3\quad \text{weakly in}\quad H^2(\o)\quad\hbox{and strongly in}\quad H^1(\o). 
				\end{aligned}	
			\end{equation}
			The fields	$\Uc_m$, $\Uc_3$ satisfy the following boundary conditions:
			$$ \Uc_m=0\quad \text{a.e. on}\quad \gamma,\quad \Uc_3=\nabla \Uc_3=0\quad\text{a.e. on}\;\; \gamma.$$
			Moreover, there exist $\fu \in L^2(\o;H^1_{per}(\Yc))^3$ and $\Gu^M \in L^2(\o;\GH^1(Y^M_\kappa))^3$ such that
			\begin{equation}\label{CFfu1}
				\begin{aligned}
					{1\over \e^3}\Pi_\e(\fu_\e) &\rightharpoonup \fu \quad \text{weakly in}\quad L^2(\o;H^1(\Yc))^3,\\
					{1\over \e^2}\Pi_\e\left(\Gu^M_\e\right) &\rightharpoonup \Gu^M\quad\text{weakly in}\quad L^2(\o;\GH^1(Y^M_\kappa))^3,
				\end{aligned}
			\end{equation}
			and
			\begin{equation}\label{CFgu1}
				\begin{aligned}
					{1\over \e^2}\Pi_\e\left(u_{\alpha,\e}\right)&\rightharpoonup \Uc_\alpha-y_3{\partial \Uc_3\over \partial x_\alpha}+ \Gu^M_\alpha\quad \text{weakly in}\quad L^2(\o;H^1(\Yc)),\\
					{1\over \e}\Pi_\e(u_{3,\e})&\to \Uc_3\quad \text{strongly in}\quad L^2(\o;H^1(\Yc)).
				\end{aligned}
			\end{equation}
			Furthermore, one has
			\begin{equation}\label{CFgu2}
				\begin{aligned}	
					{1\over\e^2}\Pi_\e\left(\nabla u_\e(\nabla u_\e)^T\right) &\rightharpoonup
					F(\Uc_3) \quad && \text{weakly in}\quad L^2(\o\X Y_\kappa^B)^{3\X3},\\
					{1\over\e}\Pi_\e\left(\nabla u_\e(\nabla u_\e)^T\right) &\rightharpoonup 0 \quad && \text{weakly in}\quad L^2(\o\X Y_\kappa^M)^{3\X3}
				\end{aligned}
			\end{equation}
			where the symmetric matrix $F(\Uc_3)$ is given by
$$
				F(\Uc_3)=\begin{pmatrix*}
					(\partial_1\Uc_3)^2&\partial_1\Uc_3\partial_2\Uc_3& 0\\
					\partial_1\Uc_3\partial_2\Uc_3& (\partial_2\Uc_3)^2& 0\\
					0&0&|\nabla\Uc_3|^2
				\end{pmatrix*}.
$$
		\end{lemma}	
		\begin{proof}
			The convergences \eqref{CFU1}, \eqref{CFfu1}$_1$ are the immediate consequences of the estimates \eqref{EWAgu}, \eqref{EWRUO1} and using the properties of the unfolding operator.
			We remind that
			\begin{equation}\label{UGUFU1}
				v_\e-\text{\bf id}=u_\e=\Gu_\e+\Gu^M_\e .
			\end{equation}
			Here  by the Lemma \ref{ExDef1} we have the extension $\Gu_\e$ of the restriction of $u_\e=v_\e-I_d$ in $\O^B_\e$, i.e.
			$$ \Gu_\e^M=u_\e-\Gu_\e\quad\text{in}\quad \O_\e^M\quad \text{and}\quad  \Gu_\e^M=0\quad \text{a.e. in}\quad \O_\e^B.$$  
			Then, from the estimates \eqref{EWAgu} and using the properties of re-scaling unfolding operator, we obtain  \eqref{CFfu1}$_2$.\\
			We prove the convergences \eqref{CFgu1}. From Lemma \ref{DeComKL1} we have
			$$\begin{aligned}
				u_{\alpha,\e}(\cdot,x_3)&=\Uc_{\alpha,\e}-x_3{\partial\Uc_{3,\e}\over\partial x_\alpha}+\fu_{\alpha,\e}(\cdot,x_3)+\Gu^M_{\alpha,\e}(\cdot,x_3),\\
				u_{3,\e}(\cdot,x_3)&=\Uc_{3,\e}+\fu_{3,\e}(\cdot,x_3)+\Gu^M_{3,\e}(\cdot,x_3),
			\end{aligned}\qquad \text{a.e. in}\quad (x',x_3)\in \O_\e.$$ 
			From the convergences \eqref{CFU1}, \eqref{CFfu1} with the estimates \eqref{EWAgu}, and using the re-scaling unfolding operator, we get 
			$$\begin{aligned}
				{1\over \e^2}\Pi_\e(u_{\alpha,\e}) & \rightharpoonup \Uc_{\alpha}-y_3{\partial\Uc_{3}\over\partial x_\alpha}+\Gu^M_{\alpha}\quad \text{weakly in}\quad L^2(\o;H^1(\Yc)),\\
				{1\over \e}\Pi_\e(u_{3,\e})&\to \Uc_3\quad\text{strongly in}\quad L^2(\o;H^1(\Yc)).
			\end{aligned}$$
			Here, we have also used the fact that the terms with $\fu_\e$ vanish, hence, we get the convergences \eqref{CFgu1}.\\ 
			We give the proof of convergence \eqref{CFgu2}. For that we see
			$$\nabla u_\e=\nabla \Gu_\e+\nabla \Gu_\e^M.$$
			Again using the decomposition \eqref{PD1}, we have $\nabla \Gu_\e=\nabla \Uc_{KL,\e}+\nabla \fu_\e$, which gives
			$$\nabla u_\e=
			\begin{pmatrix*}
				\ds \partial_1\Uc_{1,\e}-x_3{\partial^2\Uc_{3,\e}\over\partial x_1^2} &			\ds \partial_2\Uc_{1,\e}-x_3{\partial^2\Uc_{3,\e}\over\partial x_1 \partial x_2} &- \partial_1\Uc_{3,\e}\\[3mm]
				\ds \partial_1\Uc_{2,\e}-x_3{\partial^2\Uc_{3,\e}\over\partial x_1\partial x_2} & 	\ds \partial_2\Uc_{2,\e}-x_3{\partial^2\Uc_{3,\e}\over\partial x_2^2} & -\partial_2\Uc_{3,\e}\\[3mm]
				\partial_1\Uc_{3,\e}&	 \partial_2\Uc_{3,\e}&0
			\end{pmatrix*}
			+ \nabla \fu_\e+\nabla \Gu^M_\e.
			$$
			Since the second convergence of \eqref{CFU1} is strong and all other fields vanish due to the convergences \eqref{CFgu1} and \eqref{CFfu1}$_1$ together with the estimates \eqref{EWAgu} and \eqref{EWRUO1}, implies 
			$$
			{1 \over \e}\Pi_\e\left(\nabla u_\e\right)={1\over \e}\Pi_\e(\nabla\Uc_{KL,\e})+{1\over\e^2}\nabla_y\Pi_\e(\fu_\e)\quad \text{a.e. in}\quad L^2(\o\X Y_\kappa^B)^{3\X3}.
			$$
			So, we have
			\begin{equation}\label{CFrgu1}
				{1 \over \e}\Pi_\e\left(\nabla u_\e\right) \to
				\begin{pmatrix*}
					0&0& -\partial_1\Uc_3\\
					0&0& -\partial_2\Uc_3\\
					\partial_1\Uc_3&\partial_2\Uc_3&0
				\end{pmatrix*} 
				\quad \text{strongly in}\quad L^2(\o\X Y^B_\kappa)^{3\X3}.
			\end{equation}
			Therefore,
			\begin{equation}\label{C814}
				\begin{aligned}
					{1\over \e^2}\Pi_\e\left(\nabla u_\e (\nabla u_\e)^T\right) &\to
					\begin{pmatrix*}
						0&0& -\partial_1\Uc_3\\
						0&0& -\partial_2\Uc_3\\
						\partial_1\Uc_3&\partial_2\Uc_3&0
					\end{pmatrix*}
					\begin{pmatrix*}
						0&0&\partial_1\Uc_3\\
						0&0&\partial_2\Uc_3\\
						-\partial_1\Uc_3& -\partial_2\Uc_3&0
					\end{pmatrix*}\\
					& = F(\Uc_3) \qquad \text{strongly in}\quad L^1(\o\X Y_\kappa^B)^{3\X3}.
				\end{aligned}
			\end{equation} 	
			Using the estimate \eqref{KornP2} and the assumption on the sequence \eqref{AFLP1}, we deduce that the sequence $\ds\Big\{{1 \over \e}\Pi_\e\left(\nabla u_\e\right)\Big\}_\e$ is uniformly bounded in $L^2(\o\X Y_\kappa^M)^{3\X3}$, like see \eqref{EWRUO1}
			$$ {1\over \e}\left\|\Pi_\e(\nabla u_\e)\right\|_{L^2(\o\X Y_\kappa^M)}\leq C,$$  which implies
			\begin{equation}\label{C815}
				{1\over\e}\Pi_\e\left(\nabla u_\e(\nabla u_\e)^T\right) \to 0\quad\text{strongly in}\quad L^1(\o\X Y_\kappa^M)^{3\X3}.
			\end{equation}			
			To prove that the convergences \eqref{C814}--\eqref{C815} are also weak in $L^2(\o\X Y_\kappa^B)^{3\X3}$ and $L^2(\o\X Y_\kappa^M)^{3\X3}$ respectively, we show that the sequences $\ds\left\{{1\over \e^2}\Pi_\e\left(\nabla u_\e(\nabla u_\e)^T\right)\right\}_\e$ and $\ds\left\{{1\over\e}\Pi_\e\left(\nabla u_\e(\nabla u_\e)^T\right)\right\}_\e$ are bounded in $L^2(\o\X Y_\kappa^B)^{3\X 3}$ and $L^2(\o\X Y_\kappa^M)^{3\X3}$. Then, using the estimates \eqref{EWA1} and \eqref{EWA3}, we obtain
			$$\begin{aligned}
				&\|\nabla u_\e(\nabla u_\e)^T\|_{L^2(\O^B_\e)}
				\leq \|\nabla u_\e(\nabla u_\e)^T+2e(u_\e)\|_{L^2(\O^B_\e)}+2\|e(u_\e)\|_{L^2(\O^B_\e)}\\
				&\hskip 33mm =\|\nabla v_\e(\nabla v_\e)^T-\GI_3\|_{L^2(\O^B_\e)}+2\|e(u_\e)\|_{L^2(\O^B_\e)}\leq C\e^{5/2},\\
				&\|\nabla u_\e(\nabla u_\e)^T\|_{L^2(\O^M_\e)}
				\leq \|\nabla u_\e(\nabla u_\e)^T+2e(u_\e)\|_{L^2(\O^M_\e)}+2\|e(u_\e)\|_{L^2(\O^M_\e)}\\
				&\hskip 33mm =\|\nabla v_\e(\nabla v_\e)^T-\GI_3\|_{L^2(\O^M_\e)}+2\|e(u_\e)\|_{L^2(\O^M_\e)}\leq C\e^{3/2}.
			\end{aligned}$$
			From the above we get that the convergences \eqref{C814}--\eqref{C815} are also weak in $L^2(\o\X Y_\kappa^B)^{3\X3}$ and 
			$L^2(\o\X Y_\kappa^M)^{3\X3}$, respectively. 
		\end{proof}  
		\begin{definition}
			For any $\psi\in H^1(Y_\kappa^B)^3$, (resp. $ H^1(Y_\kappa^M)^3$, $L^2(\o\X Y_\kappa^B)^3$ and $L^2(\o\X Y_\kappa^M)^3$) we denote $e_y(\psi)$ the $3\X 3$ symmetric matrix whose entries are
			$$e_{y,ij}(\psi)={1\over 2}\left({\partial \psi_i\over \partial y_j}+{\partial \psi_j\over \partial y_i}\right).$$
		\end{definition}
		
		\begin{lemma}\label{CFeu2}
			For the same subsequence as in the previous Lemma, one has 
			\begin{equation}\label{CFeu1}
				\begin{aligned}
					{1\over \e^2}\Pi_\e\left(e(u_\e)\right)&\rightharpoonup E^{Lin}(\Uc)+e_y(\fu)\quad &&\text{weakly in}\quad L^2(\o\X Y_\kappa^B)^{3\X3},\\
					{1\over \e}\Pi_\e\left(e(u_\e)\right)&\rightharpoonup e_y(\Gu^M) \quad &&\text{weakly in}\quad L^2(\o\X Y_\kappa^M)^{3\X3}.
				\end{aligned}
			\end{equation}
			where $E^{Lin}(\Uc)$ denotes the symmetric matrix
$$
				E^{Lin}(\Uc)=
				\begin{pmatrix}
					e_{11}(\Uc_m)-y_3D_{11}(\Uc_3)&*&*\\
					e_{12}(\Uc_m)-y_3D_{12}(\Uc_3) &e_{22}(\Uc_m)-y_3D_{22}(\Uc_3) &*\\
					0&0&0
				\end{pmatrix}.
$$
		\end{lemma}	
		\begin{proof}
			We have from \eqref{UGUFU1} that
			$$e(u_\e)=
			\left\{
			\begin{aligned}
				&e(\Gu_\e),\quad &&\text{a.e. in}\quad L^2(\O^B_\e)^{3\X3},\\
				&e(\Gu_\e)+e(\Gu^M_\e),\quad &&\text{a.e. in}\quad L^2(\O_\e^M)^{3\X3}.
			\end{aligned}\right.
			$$
			The strain tensor of $\Gu_\e$ is given by the following $3\X3$ symmetric matrix defined a.e. in $\O_\e$ by 
$$
				\begin{aligned}
					e(\Gu_\e)&=e(U_{KL,\e})+e(\fu_\e)\\
					&=\begin{pmatrix}
						\ds e_{11}(\Uc_{m,\e})-x_3{\partial^2\Uc_{3,\e}\over\partial x_1^2}&*&*\\
						\ds e_{12}(\Uc_{m,\e})-x_3{\partial^2\Uc_{3,\e}\over\partial x_1\partial x_2}& \ds e_{22}(\Uc_{m,\e})-x_3{\partial^2\Uc_{3,\e}\over\partial x_2^2}&*\\
						0&0&0
					\end{pmatrix}+e(\fu_\e).
				\end{aligned}
$$
			From the convergences \eqref{CFU1}, \eqref{CFfu1}$_{1,2}$ with definition of $e(\Gu_\e)$ we get
			$$\begin{aligned}
				{1\over\e^2}\Pi_\e(e(u_\e))={1\over \e^2}\Pi_\e(e(\Gu_\e))&={1\over\e^2}\Pi_\e\left(e(\Uc_{KL,\e})\right)+{1\over\e^3}e_y\left(\Pi_\e(\fu_{\e})\right)\\
				&\rightharpoonup E^{Lin}(\Uc)+e_y(\fu)\quad \text{weakly in}\quad L^2(\o\X Y_\kappa^B)^{3\X3},
			\end{aligned}$$
			and 
			$$
			\begin{aligned}
				{1\over\e}\Pi_\e(e(u_\e))&={1\over\e}\Pi_\e\left(e(\Uc_{KL,\e})\right)+{1\over\e^2}e_y\left(\Pi_\e(\fu_{\e})\right)+{1\over \e^2}e_y\left(\Pi_\e(\Gu^M_\e)\right)\\ &\rightharpoonup e_y(\Gu^M)\quad \text{weakly in}\quad L^2(\o\X Y_\kappa^M)^{3\X3}.
			\end{aligned}
			$$
			Hence, we have the convergence \eqref{CFeu1}.  
		\end{proof}	
		
		We end this subsection with the limit of the Green St. Venant's strain tensor $ \left(\nabla v_\e (\nabla v_\e)^T-\GI_3\right)$, which is used in the total elastic energy.
		\begin{lemma}\label{CFGLST2}
			For the same subsequence as the previous Lemmas, one has
			\begin{equation}\label{CFGLST1}
				\begin{aligned}
					{1\over 2\e^2} \Pi_\e\left(\nabla v_\e (\nabla v_\e)^T-\GI_3\right) &\rightharpoonup E(\Uc)+e_y(\wh{\fu})\quad &&\text{weakly in}\quad L^2(\o\X Y_\kappa^B)^{3\X3},\\
					{1\over 2\e} \Pi_\e\left(\nabla v_\e (\nabla v_\e)^T-\GI_3\right) &\rightharpoonup   e_y(\Gu^M) \quad &&\text{weakly in}\quad L^2(\o\X Y_\kappa^M)^{3\X3},
				\end{aligned}
			\end{equation}
			where the symmetric matrix $E(\Uc)$ is given by
			$$ E(\Uc)=
			\begin{pmatrix*}
				\ds -y_3{\partial^2\Uc_3\over\partial x_1^2}+\Zc_{11}& \ds -y_3{\partial^2\Uc_3\over\partial x_1 \partial x_2}+\Zc_{12}&0\\
				*& \ds -y_3{\partial^2\Uc_3\over\partial x_2^2}+\Zc_{22}&0\\
				0&0&0
			\end{pmatrix*}
			$$
			
			here, we have set\footnote{Note, that the anti-symmetric part is responsible for the the non-linearity of the problem.}
			$$\begin{aligned}
				\Zc_{\alpha\beta}&=e_{\alpha\beta}(\Uc_m)+{1\over 2} {\partial \Uc_3 \over \partial x_\alpha} {\partial \Uc_3 \over \partial x_\beta},\quad\text{and}\\
				\wh{\fu}(x',y)&=\fu(x',y)+y_3|\nabla\Uc_3(x')|^2\Ge_3\quad \text{for a.e.}\quad (x',y)\in\o\X\Yc.
			\end{aligned}$$
		\end{lemma}	
		\begin{proof}
			Replacing the $v_\e$ in the strain tensor $\nabla v_\e(\nabla v_\e)^T-\GI_3$ by $u_\e$, yields
			$$ \nabla v_\e(\nabla v_\e)^T-\GI_3= \nabla u_\e(\nabla u_\e)^T+\nabla u_\e+(\nabla u_\e)^T= \nabla u_\e(\nabla u_\e)^T+2e(u_\e).$$
			Hence, from convergences \eqref{CFgu1}, \eqref{CFgu2} and re-arranging the terms we obtain the required result.
		\end{proof}

		\begin{remark}$ $
			We remark the following from the above Lemmas
			\begin{itemize}
				\item The convergence \eqref{CFGLST1} gives the limit of the Green St. Venant's strain tensor used in the energy.
				\item  From the convergences \eqref{CFgu1}$_2$, we get that the limit displacement is of Kirchhoff-Love type.
				\item The convergences \eqref{CFrgu1} and the estimate \eqref{EWAgu}$_1$ imply
				\begin{equation}\label{C920}
					\Pi_\e(\nabla v_\e)\to \GI_3\quad \text{strongly in}\quad L^2(\o\X\Yc)^{3\X3}.
				\end{equation}
			\end{itemize}
		\end{remark}
		
		\subsection{Asymptotic behavior of the sequence $\ds \left\{{m_\e \over \e^5}\right\}_\e$}
		The limit of the previous subsection allows us to investigate the limit of the elastic problem. Therefore, we introduce below the limit re-scaled elastic energy for the elasticity problem
$$
			\GJ(\Uc, \wh{\fu},\Gu^M)=\GJ_B(\Uc,\wh{\fu})+\GJ_M(\Gu^M)-|\Yc|\int_{\o} f\cdot\Uc dx'- \sum_{\alpha=1}^2\int_{\o}f_\alpha\Big(\int_\Yc\Gu_\alpha^M dy\Big)dx',					
$$
		where we have set 	the part of the energy without the external force as
$$
			\begin{aligned}
				\GJ_B(\Uc,\wh{\fu})&=\int_{\o}\int_{Y_\kappa^B}Q(y,E(\Uc)+e_y(\wh{\fu})) dydx',\\
				\GJ_M(\Gu^M)&=\int_{\o}\int_{Y_\kappa^M} Q(y,e_y(\Gu^M))dydx'.
			\end{aligned}
$$
		Let us define the limit displacement space as
		$$\O_R=\o\setminus\overline{\gamma},\qquad 
		\D_0=\Big\{\Uc=\left(\Uc_1,\Uc_2,\Uc_3\right)\in H^1(\O_R)^2\X H^2(\O_R) \;\;|\;\; \Uc=\partial_\alpha\Uc_3=0\quad\text{a.e. on}\;\; \gamma\Big \}.$$ 
		\begin{remark}\label{Re7}
		   Even though the total energy, as stated in equation \eqref{stpr}, is expressed in terms of deformation, the limit is derived in terms of displacement. This is a common practice when the total elastic energy is in the Von-K\'arm\'an regime, which is of order $O(\e^5)$.
		\end{remark}
		Before showing the convergence of the problem with $\Gamma$-convergence, we first prove that the limit-functional $\GJ$ attains a minimum on  $\D=\D_0\X L^2(\O_R;H^1_{per,0}(Y_\kappa^B))^3\X L^2(\O_R;\GH^1(Y_\kappa^M))^3$.

		\begin{lemma}\label{L107}
			Let us equip the space $\S=\R^3\X\R^3\X H^1_{per,0}(\Yc)^3$ with the semi-norm 
			$$	\|(\eta,\zeta,\wh{w})\|_\S = \sqrt{\sum_{i,j=1}^3  \|\wt{\cal E}_{ij}(\eta,\zeta,\wh{w})\|^2_{L^2( \Yc)}},
			$$ 
			where for every $(\eta,\zeta,\wh{w}) \in \S=\R^3\X\R^3\X H^1_{per,0}(\Yc)^3$, we denote $\wt{\cal E}$ the symmetric matrix by
			$$
			\wt{\cal E}(\eta,\zeta,\wh{w})=
			\begin{pmatrix}
				\eta_1-y_3\zeta_1+ e_{11,y}(\wh w)  & \eta_3 -y_3 \zeta_3 + e_{12,y}(\wh w) 
				&  e_{13,y}(\wh w)  \\
				* & \eta_2-y_3\zeta_2 + e_{22,y}(\wh w)  &  e_{23,y}(\wh w) \\
				* & *&    e_{33,y}(\wh w) 
			\end{pmatrix}.
			$$ 
			Then, this expression defines a norm on $\S$. Moreover, there exists constants $c,C>0$ such that for all $(\eta,\zeta,\wh{w}) \in \S$ 
			\begin{equation}\label{ENE1}
				c\left(\|\eta\|_2 + \|\zeta\|_2 + \|\wh{w}\|_{H^1_{per,0}(\Yc)}\right) \leq  \|(\eta,\zeta,\wh{w})\|_\S\leq C \left(\|\eta\|_2 + \|\zeta\|_2 + \|\wh{w}\|_{H^1_{per,0}(\Yc)}\right).
			\end{equation} 
		\end{lemma}
		
		\begin{proof} {\bf Step 1.} We show that  $\|(\eta,\zeta,\wh{w})\|_\S = 0$ implies $(\eta,\zeta,\wh{w})=0$.\\[1mm]	
			Let $(\eta,\zeta,\wh{w}) \in \S$ satisfy $\|(\eta,\zeta,\wh{w})\|_\S = 0$ and define the map
$$
				\tau(y) = \begin{pmatrix}
					y_1(\eta_1 - y_3\zeta_1) + y_2(\eta_3 - y_3\zeta_3)\\
					y_1(\eta_3 - y_3\zeta_3) + y_2(\eta_2 - y_3\zeta_2)\\
					\ds \frac{y_1^2}{2}\zeta_1+\frac{y_2^2}{2}\zeta_2 + y_1y_2\zeta_3
				\end{pmatrix}.
$$
			We can rewrite
			\begin{equation}
				\begin{aligned}\label{Eq1023}
					\wt{\cal E}(\eta,\zeta,\wh{w}) &= e_{y}(\tau + \wh w),\\
					\wt{\Ec}(0,0,\wh{w})&=e_y(\wh{w}).
				\end{aligned}
			\end{equation}
			So, $\|(\eta,\zeta,\wh{w})\|_\S = 0$ implies $e_y(\tau+\wh{w})=0$.
			Hence, $\tau(y) + \wh w(y) = a + b\land y$ is a rigid motion, $a=a_1\Ge_1+a_2\Ge_2+a_3\Ge_3$ and $b=b_1\Ge_1+b_2\Ge_2+b_3\Ge_3$. Since  $\wh{w}\in H^1_{per,0}(\Yc)^3$, so we have periodicity in the directions $\Ge_\alpha$ and vanishing mean, i.e,
			\begin{equation}\label{Eq113-}
				\begin{aligned}
					&\wh{w}(y+n\Ge_\alpha)=\wh{w}(y),\qquad \text{a.e. on}\quad \Yc,\quad \forall n\in \Z,\\
					&\Mc_{\Yc}(\wh{w})={1\over |\Yc|}\int_{\Yc}\wh{w}(y)dy=0.
				\end{aligned}	
			\end{equation}
			So, we have periodicity in the directions $\Ge_1$, $\Ge_2$ of the displacement
			$$
			\wh{w}(y)=		\begin{pmatrix}
				y_1(\eta_1 - y_3\zeta_1) + y_2(\eta_3 - y_3\zeta_3)-a_1-b_2y_3+b_3y_2\\
				y_1(\eta_3 - y_3\zeta_3) + y_2(\eta_2 - y_3\zeta_2)- a_2-b_3y_1+b_1y_3\\
				\ds \frac{y_1^2}{2}\zeta_1+\frac{y_2^2}{2}\zeta_2 + y_1y_2\zeta_3-a_3-b_1y_2+b_2y_1
			\end{pmatrix}.	$$	
			This first implies
			$$\eta_1=\zeta_1=0,\qquad \eta_2=\zeta_2=0,\quad \text{and}\quad \eta_3-y_3\zeta_3\pm b_3=0\quad \forall\;y_3\in(-\kappa,\kappa).$$	
			From which we get $\eta_3=\zeta_3=b_3=0$. Now, with $\eta=0=\zeta$ along with equality in the third component gives $b_1=b_2=0$.
			Applying \eqref{Eq113-}$_2$ leads to $a=0$.
			Hence, we get $a=b=\eta=\zeta=0$, which thus implies $\wh{w}=0$.\\		
			\noindent{\bf Step 2.}	Using contradiction method we show the result \eqref{ENE1}$_1$, which is enough since the other side is trivial.\\[1mm]
			Indeed, if  there does not exist any $C>0$ such that the inequality \eqref{ENE1} holds, then for all $n\in\N^*$ there exist a $(\eta_n,\zeta_n,\wh{w}_n)\in \S$ such that
			\begin{equation}\label{Eq1024}
				\|(\eta_n,\zeta_n,\wh{w}_n)\|_{\S}< {1\over n} \qquad \hbox{and}\qquad \|\eta_n\|_2 + \|\zeta_n\|_2 + \|\wh{w}_n\|_{H^1_{per,0}(\Yc)}=1 .
			\end{equation}
			From the above inequality \eqref{Eq1024}$_2$,  there exist a subsequence (still denoted by $n$), and $(\eta,\zeta,\wh{w}) \in \S$ such that
			$$
			\begin{aligned}
				&\eta_n\to \eta,\qquad \zeta_n\to \zeta,\\
				&\wh{w}_n\rightharpoonup \wh{w}\quad\hbox{weakly in } H^1_{per,0}(\Yc)^3.
			\end{aligned}
			$$ So, from the above convergences we get
			$$\wt{\cal E}(\eta_n,\zeta_n,\wh{w}_n) \rightharpoonup \wt{\cal E}(\eta,\zeta,\wh{w})\quad\hbox{weakly in } L^2(\Yc)^{3\X 3}.$$ Besides, from 
			\eqref{Eq1024}$_1$	we have
			$$\wt{\cal E}(\eta_n,\zeta_n,\wh{w}_n) \to 0 \quad\hbox{strongly in } L^2(\Yc)^{3\X 3}.$$ From the result obtained in Step 1, we get $(\eta,\zeta,\wh{w})=(0,0,0)$. Then, we have from \eqref{Eq1023}$_2$ that
			$$e_y(\wh{w}_n)\to  0\quad\hbox{strongly in } L^2(\Yc)^{3\X 3}.$$ Since $\|e_y(\cdot)\|_{L^2(\Yc)}$ is a norm in $H^1_{per,0}(\Yc)^3$, this finally leads to
			$$\lim_{n\to 0}\big(\|\eta_n\|_2 + \|\zeta_n\|_2 + \|\wh{w}_n\|_{H^1_{per,0}(\Yc)}\big)=0$$ which is a contradiction.
		\end{proof}

		\begin{lemma}\label{ExMin1}
			The functional $\GJ$ admits a minimum on $\D$.
		\end{lemma}
		\begin{proof}
			First, we observe from \eqref{QPD} and \eqref{ENE1}$_1$ that there exist a constant $C_1>0$ such that
			\begin{equation}\label{Eq823}
				\begin{aligned}
					&\sum_{\alpha,\beta=1}^2\left[\left\|e_{\alpha\beta}(\Wc)+{1\over 2}{\partial \Wc_3\over \partial x_\alpha}{\partial \Wc_3\over \partial x_\beta}\right\|^2_{L^2(\o)}+\left\|{\partial^2\Wc_3\over\partial x_\alpha\partial x_\beta}\right\|^2_{L^2(\o)}\right]+\|\wh{w}\|^2_{L^2(\o;H^1(Y_\kappa^B))} \leq C_1\GJ_B(\Wc,\wh{w}),\\
					&\hskip 83mm \text{for all}\quad(\Wc,\wh{w})\in\D_0\X L^2(\O_R;H^1_{per,0}(Y_\kappa^B))^3.
				\end{aligned}
			\end{equation} Since $\GJ(0,0,0)=0$, from now on, we only consider triplets $(\Uc,\wh{\fu},\Gu^M)\in \D$ such that $\GJ(\Uc,\wh{\fu},\Gu^M)\leq 0$.\\[1mm]
			Let us set
			$$ m=\inf_{(\Uc,\wh{\fu},\Gu^M)\in\D} \GJ(\Uc,\wh{\fu},\Gu^M).$$  We have $m\in[-\infty,0]$.\\
			
			\noindent{\it {\bf Step 1.}} We prove that $m\in (-\infty,0]$.\\
			
			Due to the boundary conditions on $\Uc_3$ in $\D_0$, the fact that $\wh{\fu}$ belongs to $L^2(\O_R;H^1_{per,0}(Y_\kappa^B))^3$ and with using the estimate \eqref{Eq823}, we immediately have
			\begin{equation}\label{Eq824}
				\|\Uc_3\|^2_{H^2(\o)} \leq C\sum_{\alpha,\beta=1}^2\left\|{\partial^2\Uc_3\over\partial x_\alpha\partial x_\beta}\right\|^2_{L^2(\o)}\leq C \GJ_B(\Uc,\wh{\fu}),\qquad \|\wh{\fu}\|^2_{L^2(\o;H^1(Y_\kappa^B))}\leq C_3\GJ_B(\Uc,\wh{\fu}).
			\end{equation}
			Using the estimate \eqref{Eq823}, with the fact that $\ds \Zc_{\alpha\beta}=e_{\alpha\beta}(\Uc)+{1\over 2}\partial_\alpha\Uc_3\partial_\beta\Uc_3$ and due to  the embedding $H^2(\O_R) \hookrightarrow W^{1,4}(\O_R)$, we get
			$$
			\begin{aligned}
				\sum_{\alpha,\beta=1}^2\|e_{\alpha\beta}(\Uc)\|^2_{L^2(\O_R)} & \leq 2\sum_{\alpha,\beta=1}^2\Big(\Big\|e_{\alpha,\beta}(\Uc)+{1\over 2}{\partial \Uc_3\over \partial x_\alpha}{\partial \Uc_3\over \partial x_\beta}\Big\|^2_{L^2(\o)}+{1\over 4}\Big\|{\partial \Uc_3\over \partial x_\alpha}{\partial \Uc_3\over \partial x_\beta}\Big\|^2_{L^2(\O_R)}\Big)\\
				&\leq C\GJ_B(\Uc,\wh{\fu})+C\|\nabla\Uc_3\|^4_{L^4(\O)}\leq C\GJ_B(\Uc,\wh{\fu})+C\GJ_B(\Uc,\wh{\fu})^2.
			\end{aligned}$$
			Then, using the $2$D Korn's inequality, we obtain
			\begin{equation}\label{Eq825}
				\|\Uc_1\|^2_{H^1(\o)}+\|\Uc_2\|^2_{H^1(\o)}\leq C\sum_{\alpha,\beta=1}^2\|e_{\alpha\beta}(\Uc)\|^2_{L^2(\o)}  \leq C\GJ_B(\Uc,\wh{\fu})+C_0^2\GJ_B(\Uc,\wh{\fu})^2.
			\end{equation}
			From the estimates \eqref{ENE1}$_1$, \eqref{QPD} and the fact that $\Gu^M\in L^2(\O_R;\GH^1(Y_\kappa^M))^3$, we have
			\begin{equation}\label{Eq826}
				\|\Gu^M\|^2_{L^2(\o;H^1(Y_\kappa^M))}\leq C \GJ_M(\Gu^M).
			\end{equation}
			With the estimates \eqref{Eq824}--\eqref{Eq826}, we obtain
$$
				\begin{aligned}
					&\GJ_B(\Uc,\wh{\fu})+\GJ_M(\Gu^M)\\
					&\leq \sum_{i=1}^3 \|f_i\|_{L^2(\o)}\|\Uc_i\|_{L^2(\o)}+\|f\|_{L^2(\o)}\|\Gu^M\|_{L^2(\o,H^1(Y_\kappa^M))}\\
					&\leq C\|f_3\|_{L^2(\o)}\sqrt{\GJ_B(\Uc,\wh{\fu})}+C\sqrt{\|f_1\|^2_{L^2(\o)}+\|f_2\|^2_{L^2(\o)}}\left[\sqrt{\GJ_B(\Uc,\wh{\fu})}+\GJ_B(\Uc,\wh{\fu})\right]\\
					&\hskip 100mm +C\|f\|_{L^2(\o)}\sqrt{\GJ_M(\Gu^M)}.
				\end{aligned}  
$$
			Thus, we get from the estimate 
			$$ \GJ_B(\Uc,\wh{\fu})+\GJ_M(\Gu^M)\leq C\|f\|_{L^2(\o)}\sqrt{\GJ_B(\Uc,\wh{\fu})+\GJ_M(\Gu^M)}+\GC\|f\|_{L^2(\o)}\GJ_B(\Uc,\wh{\fu}).$$
			From the above inequality with the assumption on the force $\GC\|f\|_{L^2(\o)}\leq 1/2$ (see \eqref{EQ74}), we have
			$$ \GJ_B(\Uc,\wh{\fu})+\GJ_M(\Gu^M)\leq C.$$
			Finally, for every $(\Uc,\wh{\fu},\Gu^M)\in \D$ such that $\GJ(\Uc,\wh{\fu},\Gu^M)\leq 0$ we have 
			$$ \|\Uc_1\|_{L^2(\o)}+\|\Uc_2\|_{L^2(\o)}+\|\Uc_3\|_{L^2(\o)}+\|\wh{\fu}\|_{L^2(\o;H^1(Y_\kappa^B)}+\|\Gu^M\|_{L^2(\o;H^1(Y_\kappa^M))}\leq C.$$
			As a consequence $m\in (-\infty,0]$.\\
			
			\noindent{\it {\bf Step 2.}} We prove that $m$ is a minimum.\\
			
			Consider a minimizing sequence $\{(\Uc_n,\wh{\fu}_n,\Gu^M_n)\}_n \subset \D$ satisfying ${\GJ}(\Uc^n,\wh{\fu}_n,\Gu^M_n)\leq {\GJ}(0,0,0) = 0$ and
$$
				m = \inf_{(\Uc,\wh{\fu},\Gu^M)\in \D} {\GJ}(\Uc,\wh{\fu},\Gu^M) = \lim_{n\rightarrow+\infty} {\GJ}(\Uc^n,\wh{\fu}_n,\Gu^M_n).
$$
			From Step 1, one has
			$$ 
			\|\Uc^n_1\|_{H^1(\o)}+\|\Uc^n_2\|_{H^1(\o)}+\|\Uc^n_3\|_{H^2(\o)}+\|\wh{\fu}_n\|_{L^2(\o; H^1(Y_\kappa^B))}+\|\Gu^M_n\|_{L^2(\o;H^1(Y_\kappa^M))}\leq C,
			$$ where the constant does not depend on $n$.\\
			Hence, there exists a  subsequence of $\{(\Uc^n,\wh{\fu}_n,\Gu^M_n)\}_n$, still denoted $\{(\Uc^n,\wh{\fu}_n,\Gu^M_n)\}_n$,  such that $$(\Uc^n,\wh{\fu}_n,\Gu^M_n) \rightharpoonup  (\Uc',\wh{\fu}',\Gu^M)\quad\text{weakly 
				in }\D.$$
			Furthermore, by the lower semi-continuity of ${\GJ}$ we have 
$$
				{\GJ} (\Uc',\wh{\fu}',(\Gu^M)') \leq \liminf_{n\rightarrow+\infty} {\GJ}(\Uc^n,\wh{\fu}_n,\Gu^M_n) = \lim_{n\rightarrow+\infty} {\GJ}(\Uc^n,\wh{\fu}_n,\Gu^M_n) = m.
$$
			However, since $m = \inf_{(\Uc,\wh{\fu},\Gu^M)\in \D} {\GJ}(\Uc,\wh{\fu},\Gu^M)$ we conclude that for every $(\Uc,\wh{\fu},\Gu^M)\in \D$ it holds
$$
				{\GJ} (\Uc',\wh{\fu}',(\Gu^M)')\leq m \leq {\GJ}(\Uc,\wh{\fu},\Gu^M).
$$
			Hence, choosing $(\Uc,\wh{\fu},\Gu^M)=(\Uc',\wh{\fu}',(\Gu^M)')$, we get 
			$$  m=\inf_{(\Uc,\wh{\fu},\Gu^M)\in\D} \GJ(\Uc,\wh{\fu},\Gu^M)=\GJ(\Uc',\wh{\fu}',(\Gu^M)')$$
			which proves that $m$ in-fact is a minimum.
		\end{proof}

		The following theorem is the main result of this section. It characterizes the limit of the re-scaled infimum of the total energy  $\ds {m_\e\over \e^5}={1\over \e^5}\inf_{v_\e\in\GV_\e}\GJ_\e(v_\e)$  as the minimum of the limit energy $\GJ$ over the space $\D$. 
		\begin{theorem}\label{MainConR}
			Under the assumption on the forces \eqref{AFB1}--\eqref{EQ74}, we have
$$
				m=\lim_{\e\to0}{m_\e \over \e^5} = \min_{(\Uc,\wh{\fu},\Gu^M)\in \D} \GJ(\Uc,\wh{\fu},\Gu^M).
$$
		\end{theorem}
		\begin{proof}
			The following proof uses a form of $\Gamma$-convergence.\\
			\noindent{\it {\bf Step 1.}} In this step we show that
$$
				\liminf_{\e\to0} {m_\e\over \e^5}\geq\min_{(\Uc,\wh{\fu},\Gu^M)\in \D} \GJ(\Uc,\wh{\fu},\Gu^M).
$$
			
			Let $\{v_\e\}_\e \subset \GV_\e$, be a sequence of deformations, i.e, it satisfies 
			$$ \lim_{\e\to0} {\GJ_\e(v_\e)\over \e^5} = \liminf_{\e\to 0} {m_\e \over \e^5}.$$
			Without loss of generality, we can assume that the sequence satisfies $\GJ_\e(v_\e)\leq \GJ_\e(\text{\bf id})= 0$ and as a consequence of the estimates from the previous section, in particular \eqref{AFLP1}, the sequence $\{v_\e\}_\e$ satisfies
			$$ 	\|dist(\nabla v_\e, SO(3))\|_{L^{2}(\O^B_\e)}\leq C\e^{5/2} \quad \text{and}\quad \|dist(\nabla v_\e, SO(3))\|_{L^{2}(\O^M_\e)}\leq C\e^{3/2}.$$
			The estimates in \eqref{EWA1} give
			$$ 	\|\nabla v_\e(\nabla v_\e)^T-\GI_3\|_{L^2(\O^M_\e)}\leq C\e^{3/2}\quad \text{and}\quad
			\|\nabla v_\e(\nabla v_\e)^T-\GI_3\|_{L^2(\O^B_\e)}\leq C\e^{5/2}.$$
			Therefore, for any fixed $\e$, we are allowed to use the decomposition given in the Lemma \ref{DeComKL1} for the displacement $u_\e=v_\e-\text{\bf id}$. So, using the decomposition we obtain the estimates \eqref{EWAgu} and the convergences as in the Lemmas \ref{CFuu1}, \ref{CFeu2} and \ref{CFGLST2}. Then the assumption on forces \eqref{AFB1}--\eqref{EQ74} lead to
			\begin{equation}\label{CFf1}
				\begin{aligned}
					\lim_{\e\to0}{1\over \e^5}\int_{\o\X\Yc}\Pi_\e&\big(f_\e\cdot(v_\e-\text{\bf id} )\big)dx'dy=\lim_{\e\to0}{1\over 2\e^5}\int_{\o\X\Yc}\Pi_\e(f_\e\cdot u_\e)dx'dy\\
					&=\lim_{\e\to0}{1\over \e^5}\left(\int_{\o\X\Yc}\Pi_\e(f_\e\cdot\Gu_\e)dx'dy+ \int_{\o\X\Yc}\Pi_\e(f_\e\cdot\Gu^M_\e)dx'dy\right)\\
					&=|\Yc|\int_{\o}f\cdot \Uc dx'  +  \sum_{\alpha=1}^2\int_{\o}f_\alpha\left(\int_{\Yc}\Gu_\alpha^M dy\right) dx'.
				\end{aligned}
			\end{equation} 
			Consequently, we have
			$$\begin{aligned}
				\int_{\O_\e}f_\e\cdot \Gu_\e dx& =\sum_{\alpha=1}^2 \e^3\int_{\o}f_\alpha\Uc_{\alpha,\e} dx'+\e^4\int_{\o}f_3\Uc_{3,\e} dx',\\
				\int_{\O_\e} f_\e\cdot \Gu^M_\e dx'& =\sum_{\alpha=1}^2\e^2\int_{\o}f_\alpha\left(\int_{-\kappa\e}^{\kappa\e}\Gu^M_{\alpha,\e}dx_3\right)dx'+\e^3\int_{\o}f_3\left(\int_{-\kappa\e}^{\kappa\e}\Gu^M_{3,\e}dx_3\right)dx'. 
			\end{aligned}$$
			Further \eqref{CFf1} is converging as a product of a weak and a strong convergent sequence using the convergences \eqref{CFU1}, \eqref{CFfu1}. We also have 
			$$\begin{aligned}
				&{1\over \e^5}\int_{\O_\e}\wh{W}_\e(x,\nabla v_\e)dx\\
				&\geq {1\over \e^5}\int_{\O^B_\e}Q\left({x\over\e},{1\over2}(\nabla v_\e(\nabla v_\e)^T-\GI_3)\right)dx+ {1\over \e^3}\int_{\O^M_\e}Q\left({x\over\e},{1\over2}(\nabla v_\e(\nabla v_\e)^T-\GI_3)\right)dx.
			\end{aligned}$$
			Hence, as consequence of the above inequality with the convergences \eqref{CFGLST1}, \eqref{CFf1} and using the weak semi-continuity of $\GJ$, we get
			\begin{equation*}
				\liminf_{\e\to0}{\GJ_\e(v_\e)\over \e^5}\geq \GJ_B(\Uc,\wh{\fu})+\GJ_M(\Gu^M)-|\Yc|\int_{\o}f\cdot \Uc dx' -  \sum_{\alpha=1}^2\int_{\o}f_\alpha\left(\int_{Y_\kappa^M}\Gu_\alpha^M dy\right) dx'. 
			\end{equation*}
			\noindent{\it {\bf Step 2.}}
			We recall from the Step 2 of the proof of Lemma \ref{ExMin1} that there exist $(\Uc',\wh{\fu}',(\Gu^M)') \in \D$ such that 
			$$ m=\inf_{(\Uc,\wh{\fu},\Gu^M)\in \D} {\GJ}(\Uc,\wh{\fu},\Gu^M) = \GJ(\Uc',\wh{\fu}',(\Gu^M)').$$
			In this step we show that 
$$
				\limsup_{\e\to0}{m_\e\over \e^5} \leq \GJ(\Uc,\wh{\fu},\Gu^M)\qquad \forall (\Uc,\wh{\fu},\Gu^M)\in \D.
$$
			To do that, we will build a sequence $\{V_{n,\e}\}_{n,\e}$ of admissible deformations such that 
			$$ \limsup_{\e\to0}{m_\e\over \e^5}  \leq \lim_{n\to \infty}\lim_{\e\to0} {\GJ_\e(V_{n,\e})\over\e^5}=\GJ(\Uc,\wh{\fu},\Gu^M).$$
			We consider a sequence $\{\Uc_n,\wh{\fu}_{n},\Gu^M_{n}\}_{n}$ such that
			\begin{itemize}
				\setlength{\itemindent}{5mm}
				\item $\Uc_n\in \D_0\cap \big(\C^1(\overline{\O_R})^2\X\C^2(\overline{\O_R})\big)$, satisfying
				\begin{equation}\label{Eq1113}
					\begin{aligned}
						\Uc_{n,\alpha} &\to \Uc_\alpha\qquad \text{strongly in}\quad H^1(\O_R),\\
						\Uc_{n,3} &\to \Uc_3\qquad \text{strongly in}\quad H^2(\O_R).
					\end{aligned}
				\end{equation}
				Existence of such a sequence is because of the regularity of the boundary of $\O_R$.
				\item  $\wh{\fu}_{n}\in L^2(\O_R;H^1_{per}(Y_\kappa^B))^3\cap \C^1(\overline{\O_R}\X\overline{Y_\kappa^B})^3$ and $\Gu^M_{n} \in L^2(\O_R; \GH^1(Y_\kappa^M))^3\cap \C^1(\overline{\O_R}\X\overline{Y_\kappa^M})^3$   satisfying
				\begin{equation}\label{Eq1114}
					\begin{aligned}
						\wh{\fu}_{n} &\to \wh{\fu}\qquad \text{strongly in}\quad L^2(\O_R;H^1_{per}(Y_\kappa^B))^3\\
						\Gu^M_{n} &\to \Gu^M\qquad \text{strongly in}\quad L^2(\O_R; \GH^1(Y_\kappa^M))^3.
					\end{aligned}
				\end{equation}
			\end{itemize}  
			We show that for fixed $n$, there exists a sequence $\{V_{n,\e}\}_\e$ such that
			$$\limsup_{\e\to0}{m_\e\over \e^5} \leq \lim_{\e\to 0} {\GJ_\e(V_{n,\e})\over \e^5}= \GJ(\Uc_n,\wh{\fu}_n,\Gu^M_{n}).$$
			So, we define the following sequence  $\{V_{n,\e}\}_\e$ of deformations of the whole structure $\O_\e$ as 
			\begin{itemize}
				\setlength{\itemindent}{5mm}
				\item In $\O^B_\e$ we set
				\begin{equation}\label{Eq1115}
					\begin{aligned}
						V_{n,\e,1}^{B}(x)&=x_1+\e^2\left(\Uc_{n,1}(x_1,x_2)-{x_3\over \e}\partial_1\Uc_{n,3}(x_1,x_2)+\e\wh{\fu}_{n,1}\left(x_1,x_2,{x_3\over \e}\right)\right),\\
						V_{n,\e,2}^{B}(x)&=x_2+\e^2\left(\Uc_{n,2}(x_1,x_2)-{x_3\over \e}\partial_2\Uc_{n,3}(x_1,x_2)+\e\wh{\fu}_{n,2}\left(x_1,x_2,{x_3\over \e}\right)\right),\\
						V_{n,\e,3}^{B}(x)&=x_3+\e\left(\Uc_{n,3}(x_1,x_2)+\e^2\wh{\fu}_{n,3}\left(x_1,x_2,{x_3\over\e}\right)\right).
					\end{aligned}
				\end{equation}
				\item  In $\O_\e$ we set
				\begin{equation}\label{Eq1116}
					\begin{aligned}
						V_{n,\e,1}(x)&=V^{B}_{n,\e,1}(x)+\e^2\Gu^M_{n,1}\left(x_1,x_2,{x_3\over \e}\right),\\
						V_{n,\e,2}(x)&=V^{B}_{n,\e,2}(x)+\e^2\Gu^M_{n,2}\left(x_1,x_2,{x_3\over \e}\right),\\
						V_{n,\e,3}(x)&=V^{B}_{n,\e,3}(x)+\e^2\Gu^M_{n,3}\left(x_1,x_2,{x_3\over \e}\right).
					\end{aligned}
				\end{equation}
			\end{itemize}
			By construction, the deformations $V_{n,\e}$ belong to $\GV_\e$ and satisfy
			$$\begin{aligned}
				\|\nabla V_{n,\e}-\GI_3\|_{L^\infty(\O_\e^B)}&\leq C(n)\e,\\
				\|\nabla V_{n,\e}-\GI_3\|_{L^\infty(\O_\e^M)}&\leq C(n)\e.
			\end{aligned} $$
			So, we get $$\|\nabla V_{n,\e}-\GI_3\|_{L^\infty(\O_\e)}\leq C(n)\e,$$
			which gives the estimate of the displacement gradient. 
			Here the constant $C(n)$ does not depend on $\e$, but depends on $n$ such that $C(n)\to +\infty$ for $n\to +\infty$ because  the strong convergence  given in \eqref{Eq1113}--\eqref{Eq1114} are given in $H^1(\O_R)$, $H^2(\O_R)$ and $L^2(\O_R;H^1_{per}(\Yc))^3$. This implies for $\e$ small enough, that for a.e. $x\in \O_\e$ we have $|\nabla V_{n,\e}-\GI_3|_F<1$ a.e. in $\O_R$. As a consequence we get det$\big(\nabla V_{n,\e}(x)\big)>0$ for all $n\in \N$. This leads to
			\begin{equation}\label{Eq1118}
				m_\e\leq \GJ_\e(V_{n,\e}).
			\end{equation}
			In the expressions \eqref{Eq1115}--\eqref{Eq1116} of the displacement $V_{n,\e}-\text{\bf id}$, the explicit dependence with respect to $\e$ permits to derive directly the limit of the Green St. Venant's strain tensor as $\e$ tends to $0$ ($n$ being fixed) as in Lemma \ref{CFGLST2}, so we obtain
			$$\begin{aligned}
				{1\over 2\e^2} \Pi_\e\left(\nabla V_{n,\e} (\nabla V_{n,\e})^T-\GI_3\right) &\to E(\Uc_n)+e_y(\wh{\fu}_n)\quad &&\text{strongly on}\quad L^\infty(\o\X Y_\kappa^B)^{3\X3},\\
				{1\over 2\e} \Pi_\e\left(\nabla V_{n,\e} (\nabla V_{n,\e})^T-\GI_3\right) &\to e_y(\Gu^M_n)\quad &&\text{strongly on}\quad L^\infty(\o\X Y_\kappa^M)^{3\X3}.
			\end{aligned}$$
			The above convergences give the convergence of the elastic energy
			$$\begin{aligned}
				 & \;\;\lim_{\e\to0} {1\over \e^5}\int_{\o\X\Yc}\Pi_\e(\wh{W}_\e(y, \nabla V_{n,\e}))dx'dy\\
				&=\lim_{\e\to0}\left({1\over\e^5}\int_{\o\X Y_\kappa^B}\Pi_\e(Q(y,(\nabla V_{n,\e}(\nabla V_{n,\e})^T-\GI_3)))dx'dy\right.\\
				&\hskip 50mm \left.+{1\over \e^3}\int_{\o\X Y_\kappa^M}\Pi_\e(Q(y,(\nabla V_{n,\e}(\nabla V_{n,\e})^T-\GI_3)))dx'dy\right)\\
				&=\int_{\o\X Y_\kappa^B}Q(y,E(\Uc_n)+e_y(\wh{\fu}_n))dx'dy+\int_{\o\X Y_\kappa^M}Q(y,e_y(\Gu^M_{n}))dx'dy,
			\end{aligned}$$
			and the right-hand side
			$$\lim_{\e\to0}{1\over \e^5}\int_{\o\X\Yc}\Pi_\e(f_\e\cdot(V_{n,\e}-\text{\bf id}))dx'dy=  |\Yc|\int_\o f\cdot\Uc_ndx' + \sum_{\alpha=1}^2\int_{\o}f_\alpha\left(\int_{Y_\kappa^M}\Gu_{n,\alpha}^M dy\right) dx'.$$
			Hence, with \eqref{Eq1118} and using the above convergences we obtain
			$$ \limsup_{\e\to0}{m_\e\over \e^5}\leq \lim_{\e\to0}{1\over \e^5}\GJ_\e(V_{n,\e})=\GJ(\Uc_{n},\wh{\fu}_{n},\Gu^M_{n}).$$
			Since this holds for every $n\in \N$, when $n$ tends infinity the strong convergences \eqref{Eq1113}--\eqref{Eq1114} yield
			$$\limsup_{\e\to0}{m_\e\over \e^5}\leq \lim_{n\to +\infty} \GJ(\Uc_{n},\wh{\fu}_{n},\Gu^M_{n})=\GJ(\Uc,\wh{\fu},\Gu^M),$$
			which completes the Step $2$.\\
			
			\noindent{\it {\bf Step 3.}}
			From Step $1$ and $2$ we have for every $(\Uc,\wh{\fu},(\Gu^M)) \in \D$
			$$ \GJ(\Uc',\wh{\fu}',(\Gu^M)')\leq \liminf_{\e\to0}{m_\e\over \e^5}\leq \limsup_{\e\to0}{m_\e\over \e^5}\leq \GJ(\Uc,\wh{\fu},\Gu^M).$$
			Thus, choosing $(\Uc,\wh{\fu},\Gu^M)=(\Uc',\wh{\fu}',(\Gu^M)')$ gives
			$$\GJ(\Uc',\wh{\fu}',(\Gu^M)')=\lim_{\e\to0}{m_\e\over \e^5},$$
			and, 
			$$\lim_{\e\to0}{m_\e\over \e^5}=\GJ(\Uc',\wh{\fu}',(\Gu^M)')=\min_{(\Uc,\wh{\fu},\Gu^M)\in \D} \GJ(\Uc,\wh{\fu},\Gu^M).$$
		\end{proof}

		Finally, we end this section with a convergence result for $\{u_\e\}_\e$, which is a direct consequence of the convergence \eqref{CFgu1} and the Theorem \ref{MainConR}.
		\begin{corollary}
			Let $\{v_\e\}_\e$ be a sequence of $\GV_\e$ such that
			$$ \lim_{\e\to0}{\GJ_\e(v_\e)\over \e^5}=\lim_{\e\to0}{m_\e\over\e^5}.$$
			Then there exists a subsequence (still denoted by $\e$) such that
$$
				\begin{aligned}
					{1\over \e^2}\Pi_\e(u_{\e,\alpha})&\rightharpoonup \Uc'_\alpha-y_3{\partial \Uc'_3\over\partial x_\alpha}+(\Gu^M_\alpha)'\quad \text{weakly in}\quad L^2(\O_R;H^1(\Yc)), \\
					{1\over \e}\Pi_\e(u_{\e,\alpha})&\rightharpoonup \Uc'_3\quad \text{strongly in}\quad L^2(\O_R;H^1(\Yc)),
				\end{aligned}
$$
			where $(\Uc',\wh{\fu}',(\Gu^M)')$ is a minimizer of $\GJ$ in $\D$.
		\end{corollary}				
		
		\subsection{The cell problems}\label{cellPro}
		To obtain the cell problems, we consider the variational formulation for $\wh{\fu}$ and $\Gu^M$ associated to the functional $\GJ_B$ and $\GJ_M$, respectively. We recall the limit elasticity energy as 	
		\begin{equation}\label{EQ757}
			\GJ(\Uc, \wh{\fu},\Gu^M)=\GJ_B(\Uc,\wh{\fu})+\GJ_M(\Gu^M)-|\Yc|\int_{\o} f\cdot\Uc dx'  -  \sum_{\alpha=1}^2\int_{\o}f_\alpha\left(\int_{Y_\kappa^M}\Gu_\alpha^M dy\right) dx',
		\end{equation}
		we have set 
		\begin{equation}\label{EQ1041}
			\begin{aligned}
				\GJ_B(\Uc,\wh{\fu})&=\int_{\o}\int_{Y_\kappa^B}a(y)(E(\Uc)+e_y(\wh{\fu})) : (E(\Uc)+e_y(\wh{\fu})) dydx',\\
				\GJ_M(\Gu^M)&=\int_{\o}\int_{Y_\kappa^M} a(y)(e_y( \Gu^M)) : e_y(\Gu^M)dydx',
			\end{aligned}
		\end{equation}
		and we get
		\begin{equation*}
			\begin{aligned}
				a(y)(E(\Uc)+e_y(\wh{\fu})) : (E(\Uc)+e_y(\wh{\fu}))&=a_{ijkl}(y)(E_{kl}(\Uc)+e_{kl}(\wh{\fu}))(E_{ij}(\Uc)+e_{ij}(\wh{\fu})),\\
				a(y)(e_y( \Gu^M)) : e_y(\Gu^M)&=a_{ijkl}(y)e_{kl}(\Gu^M)e_{ij}(\Gu^M),\\
				\sum_{\alpha=1}^2f_\alpha\int_{Y_\kappa^M}\Gu_\alpha^M dy&=\sum_{i=1}^3\int_{Y_\kappa^M}F_i\Gu^M_idy,\qquad F=(f_1,f_2,0).
			\end{aligned}
		\end{equation*}
		Let $(\Uc, \wh{\fu},\Gu^M)$ be a minimizer, then we have
		$$\GJ(\Uc, \wh{\fu},\Gu^M)\leq \GJ(\Uc, \wh{\fu}+t\wh{\fv},\Gu^M+t\Gv^M),\qquad \forall t\in \R,\quad \forall (\wh{\fv}, \Gv^M)\in L^2(\O_R;H^1_{per,0}(Y_\kappa^B))^3\X L^2(\O_R;\GH^1(Y_\kappa^M))^3.$$
		So, to get the variational formulation, we use the Euler-Lagrange equation (since $\GJ_B$ and $\GJ_M$ are quadratic forms in $e_y(\wh{\fu}))$ and $e_y(\Gu^M)$ respectively over a Hilbert-space) and we obtain:
		\begin{equation}\label{Eq1042}
			\begin{aligned}
				&\text{Find $\wh{\fu}\in L^2(\o;H^1_{per,0}(Y_\kappa^B))^3$ and $\Gu^M\in L^2(\o;\GH^1(Y_\kappa^M))^3$ such that}\\	
				&\begin{aligned}
					&\int_{Y_\kappa^B} a(y)\left(E(\Uc)+e_y(\wh{\fu})\right) : e_y(\wh{w}) dy=0\quad \text{for all}\quad \wh{w} \in H^1_{per,0}(Y_\kappa^B)^3,\\
					&\int_{Y_\kappa^M} a(y)\left(e_y(\Gu^M)\right) : e_y(\Gw) dy=\sum_{\alpha=1}^2 f_\alpha\left(\int_{Y_\kappa^M}\Gw_\alpha dy\right)\quad \text{for all}\quad \Gw \in \GH^1(Y_\kappa^M)^3\\
					& \hbox{a.e. in }\o.
				\end{aligned}
			\end{aligned}
		\end{equation}
		So, we have linear problems. Hence we have a unique solution for the above problems. Therefore,
		we get from the above equations 
		\begin{equation}\label{EQ761}
			\begin{aligned}
				&\text{Find $\wh{\fu}\in L^2(\O;H^1_{per,0}(Y_\kappa^B))^3$ such that}\\	
				&\begin{aligned}
					&\int_{Y_\kappa^B} a(y)\left(e_y(\wh{\fu})\right) : e_y(\wh{w}) dy=-\int_{Y_\kappa^B} a(y)\left(E(\Uc)\right) : e_y(\wh{w}) dy,\qquad \forall \wh{w} \in H^1_{per,0}(Y_\kappa^B)^3 \hbox{ a.e. in }\o.
				\end{aligned}
			\end{aligned}
		\end{equation}
		This shows that $\wh{\fu}$ can be expressed in terms of the elements of the tensor $E(\Uc)$ and some correctors.\\
		Let us denote
		\begin{equation}\label{EQMatrix}
			\begin{aligned}
				M^{11}&=\begin{pmatrix*}
					1 & 0&0\\
					0&0&0\\
					0&0&0
				\end{pmatrix*},
				\quad &&M^{22}=\begin{pmatrix*}
					0&0&0\\
					0&1&0\\
					0&0&0
				\end{pmatrix*},
				\quad M^{12}=&&&M^{21}=\begin{pmatrix*}
					0&1&0\\
					1&0&0\\
					0&0&0
				\end{pmatrix*},\\
				M^{31}=M^{13}&=\begin{pmatrix*}
					0 & 0&1\\
					0&0&0\\
					1&0&0
				\end{pmatrix*},
				\quad M^{23}=&&M^{32}=\begin{pmatrix*}
					0&0&0\\
					0&0&1\\
					0&1&0
				\end{pmatrix*},
				\quad &&&M^{33}=\begin{pmatrix*}
					0&0&0\\
					0&0&0\\
					0&0&1
				\end{pmatrix*}.
			\end{aligned}
		\end{equation}
		The cell problems in $Y_\kappa^B$ are
		\begin{equation}\label{Eq125}
			\begin{aligned}
				&\text{Find $\big(\chi^m_{11},\chi^m_{12}, \chi^m_{22},\chi^b_{11},\chi^b_{12},\chi^b_{22}\big)\in [H^1_{per,0}(Y_\kappa^B)^3]^6$ such that}\\
				&\hskip 2mm\left.\begin{aligned}
					&\int_{Y_\kappa^B}a(y)(M^{\alpha\beta}+e_y(\chi^m_{\alpha\beta})) : e_y(\wh{w}) dy =0,\\
					&\int_{Y_\kappa^B}a(y)(-y_3M^{\alpha\beta}+e_y(\chi^b_{\alpha\beta})) : e_y(\wh{w}) dy =0,
				\end{aligned}\right\}\;\;
				\\
				&\text{for all}\quad \wh{w}\in H^1_{per,0}(Y_\kappa^B)^3.
			\end{aligned}			
		\end{equation}
	The above equations imply
		\begin{equation}\label{Eq1044}
			\begin{aligned}
				&\wh{\fu}(x',y)=\Zc_{\alpha\beta}(x')\chi^m_{\alpha\beta}(y)+\partial_{\alpha\beta}\Uc_3(x')\chi^b_{\alpha\beta}(y),\quad &&\hbox{for a.e. } (x',y)\in \o\X Y^B_\kappa.
			\end{aligned}
		\end{equation}
		Similarly, we express $\Gu^M$ in terms of correctors $\chi^p=(\chi^p_1,\chi^p_2)\in [H^1(Y_\kappa^M)^3]^2$. Before that we give the problem for which the $\chi^p$ are  the solution
		\begin{equation}\label{Eq1046}
			\int_{Y_\kappa^M} a(y)\left( e_y(\chi^p_\alpha)+M^{ij}\right) : e_y(\Gw) dy = \int_{Y_\kappa^M}\Gw_\alpha dy \quad \forall \Gw\in \GH^1(Y_\kappa^M)^3.
		\end{equation}
		We use the variational formulation \eqref{Eq1042}$_2$ and above equation   which implies
		\begin{equation}\label{EQ947}
			\begin{aligned}
				\Gu^M(x',y)= \sum_{\alpha=1}^2 f_\alpha(x')\chi^p_\alpha(y)\qquad \hbox{for a.e. } (x',y)\in \o\X Y^M_\kappa. 
			\end{aligned}
		\end{equation}
		We set the homogenized coefficients in $Y_\kappa^B$ and $Y_\kappa^M$ as
		\begin{equation}\label{EQ1048}
			\begin{aligned}
				&a^{B,hom}_{\alpha\beta\alpha^\prime\beta^\prime} = \frac{1}{|Y_\kappa^B|} \int_{Y_\kappa^B} a_{ijkl}(y) \left[M^{\alpha\beta}_{ij} + 
				e_{y,ij}({\chi}_{\alpha\beta}^m)\right]M^{\alpha^\prime\beta^\prime}_{kl} dy,\\
				&b^{B,hom}_{\alpha\beta\alpha^\prime\beta^\prime} = \frac{1}{|Y_\kappa^B|} \int_{Y_\kappa^B} a_{ijkl}(y) \left[y_3M^{\alpha\beta}_{ij} + 
				e_{y,ij}({\chi}_{\alpha\beta}^b)\right]M^{\alpha^\prime\beta^\prime}_{kl} dy,\\
				&c^{B,hom}_{\alpha\beta\alpha^\prime\beta^\prime} = \frac{1}{|Y_\kappa^B|} \int_{Y_\kappa^B} a_{ijkl}(y) \left[y_3M^{\alpha\beta}_{ij} + 
				e_{y,ij}({\chi}_{\alpha\beta}^b)\right]y_3M^{\alpha^\prime\beta^\prime}_{kl} dy.
			\end{aligned}
		\end{equation}
		Hence, the homogenized energy is defined by
		\begin{equation}\label{EQ949}
			\GJ_{vK}^{hom}(\Uc)=|Y_\kappa^B|\GJ^{B, hom}(\Uc)-|\Yc|\int_{\o}f\cdot \Uc dx',
		\end{equation}
		where
		\begin{equation}\label{EQ950}
			\begin{aligned}
				\begin{aligned}
					\GJ^{B, hom}(\Uc)
					&= {1\over 2}\int_{\o}\Big(a^{B,hom}_{\alpha\beta\alpha^\prime\beta^\prime}\Zc_{\alpha\beta}\Zc_{\alpha^\prime\beta^\prime}+ b^{B,hom}_{\alpha\beta\alpha^\prime\beta^\prime}\Zc_{\alpha\beta}\partial_{\alpha^\prime\beta^\prime}\Uc_3+c^{B,hom}_{\alpha\beta\alpha^\prime\beta^\prime}\partial_{\alpha\beta}\Uc_3\partial_{\alpha^\prime\beta^\prime}\Uc_3\Big)dx'
				\end{aligned}
			\end{aligned}
		\end{equation}
		with
		$$ \Zc_{\alpha\beta}=e_{\alpha\beta}(\Uc)+{1\over 2}\partial_{\alpha}\Uc_3\partial_{\beta}\Uc_3.$$
		\begin{remark}
			The homogenized energy \eqref{EQ949} unveils an intriguing observation: the limit energy stemming from the matrix part \eqref{EQ1041}$_2$ remains absent. This intriguing phenomenon explains that the soft matrix exhibits a level of weakness that renders it incapable of exerting any discernible influence on the homogenized vK plate behavior. Consequently, the homogenized energy for this composite structure remains like that of a periodic perforated plate.
		\end{remark}
		\begin{remark}
			A common terminology in literature for the homogenized tensors coefficients $a^{B,hom}$, $b^{B,hom}$, $c^{B,hom}$ are membrane, coupling and bending stiffness tensor coefficients respectively. Moreover, the entries of $a^{B,hom}$ and $c^{B,hom}$ correspond to the membrane stiffness and flexural rigidity in classical plate and shell theory, respectively.
		\end{remark}
		The minimizer of this homogenized energy functional \eqref{EQ949} satisfies the variational problem: \begin{equation}\label{EQ951}
			\begin{aligned}
				&\text{Find } \Uc\in \D_0\text{ such that for all } \Wc\in\D_0:\\
				&\hskip 25mm \int_{\o}\Big(a^{B,hom}_{\alpha\beta\alpha^\prime\beta^\prime}\Zc_{\alpha\beta}(\Uc)\Zc_{\alpha^\prime\beta^\prime}(\Wc)+ {b^{B,hom}_{\alpha\beta\alpha^\prime\beta^\prime}\over 2}\left(\Zc_{\alpha\beta}(\Uc)\partial_{\alpha^\prime\beta^\prime}\Wc_3+\Zc_{\alpha\beta}(\Wc)\partial_{\alpha^\prime\beta^\prime}\Uc_3\right)\\
				&\hskip 75mm +c^{B,hom}_{\alpha\beta\alpha^\prime\beta^\prime}\partial_{\alpha\beta}\Uc_3\partial_{\alpha^\prime\beta^\prime}\Uc_3\Big)dx' =|\Yc|\int_{\o}f\cdot\Wc dx'.
			\end{aligned}
		\end{equation}
		We introduce the homogenized coefficients in the limit energy which  gives the convergence result for the homogenized energy, but before that we show that the associated quadratic form to this energy functional is coercive.
		
		\begin{lemma}\label{L716}
			The quadratic form associated to $\GJ_{vK}^{hom}$ is coercive.
		\end{lemma}
		\begin{proof}
			
			The associated quadratic form of the homogenized energy \eqref{EQ949} comes from the macroscopic problem \eqref{EQ951}, and is given by
			\begin{equation}\label{Eq1052}
				\begin{aligned}
					&\GQ_{vK}^{hom}(\xi,\eta)=
					a^{B,hom}_{\alpha\beta\alpha^\prime\beta^\prime}\xi_{\alpha\beta}\xi_{\alpha^\prime\beta^\prime}+ {b^{B,hom}_{\alpha\beta\alpha^\prime\beta^\prime}\over 2}\big(\xi_{\alpha\beta}\eta_{\alpha^\prime\beta^\prime}+\eta_{\alpha\beta}\xi_{\alpha^\prime\beta^\prime}\big)\\
					&\hskip 75mm 	+c^{B,hom}_{\alpha\beta\alpha^\prime\beta^\prime}\eta_{\alpha\beta}\eta_{\alpha^\prime\beta^\prime}
					,\qquad \forall \;\; (\xi,\eta)\in \S_3\X\S_3.	
				\end{aligned}
			\end{equation}
			From the definition of the homogenized coefficients given in \eqref{EQ1048}, for every $(\xi,\eta)\in\S_3\X\S_3$we have
			$$ \begin{aligned}
				&a^{B,hom}_{\alpha\beta\alpha^\prime\beta^\prime}\xi_{\alpha\beta}\xi_{\alpha^\prime\beta^\prime}+ {b^{B,hom}_{\alpha\beta\alpha^\prime\beta^\prime}\over 2}\big(\xi_{\alpha\beta}\eta_{\alpha^\prime\beta^\prime}+\eta_{\alpha\beta}\xi_{\alpha^\prime\beta^\prime})+c^{B,hom}_{\alpha\beta\alpha^\prime\beta^\prime}\eta_{\alpha\beta}\eta_{\alpha^\prime\beta^\prime}\\
				&\hskip 40mm={1\over |Y_\kappa^B|}\int_{Y_\kappa^B}a_{ijkl}\left[M_{ij}+e_{y,ij}(\Psi)\right]\left[M_{kl}+e_{y,kl}(\Psi)\right]dy
			\end{aligned}$$
			where
			$$ M=(\xi_{\alpha\beta}+y_3\eta_{\alpha\beta})M^{\alpha\beta}\qquad\text{and}\qquad\Psi=\xi_{\alpha\beta}\chi^m_{\alpha\beta}+\eta_{\alpha\beta}\chi^b_{\alpha\beta}.$$
			By the coercivity of $a_{ijkl}$ from \eqref{QPD}, we get
$$
				\int_{Y_\kappa^B}a_{ijkl}\left[M_{ij}+e_{y,ij}(\Psi)\right]\left[M_{kl}+e_{y,kl}(\Psi)\right]dy \geq c_0\int_{Y_\kappa^B}\big|M+e_{y}(\Psi)\big|^2_Fdy.
$$
			Then, from Lemma \ref{L107} together with the equivalence of norms in $[\R^{3\X3}]^2$, we obtain
			$$\begin{aligned}
				\big\|M+e_{y}(\Psi)\big\|^2_{L^2(Y_\kappa^B)} &\geq C\left(\|\xi\|^2_2+\|\eta\|^2_2+\|\Psi\|_{L^2(Y_\kappa^B)}^2\right),\\
				&\geq C\big\|(\xi,\eta)\big\|_2^2,\quad\forall\;(\xi,\eta)\in\S_3\X\S_3.
			\end{aligned} $$
			So, the above estimates together with \eqref{Eq1052}, yield
			$$ \GQ_{vK}^{hom}(\xi,\eta)\geq C\|(\xi,\eta)\|_2^2,\qquad\forall\;(\xi,\eta)\in\S_3\X\S_3.$$
			Hence, we proved that the associated quadratic form is coercive.
		\end{proof}
		We are ready to give the main result of this subsection.
		\begin{theorem}
			Under the assumption on the forces \eqref{AFB1} and \eqref{EQ74}, the problem
			\begin{equation}\label{EQ1052}
				\min_{\Uc\in\D_0} \GJ_{vK}^{hom}(\Uc)
			\end{equation}	
			admits at least a solution. Moreover, one has
			$$ m=	\lim_{\e\to0}{m_\e \over \e^5} = \min_{(\Uc,\wh{\fu},\Gu^M)\in \D} \GJ(\Uc,\wh{\fu},\Gu^M)= 	\min_{\Uc\in\D_0} \GJ_{vK}^{hom}(\Uc).$$	
		\end{theorem}
		\begin{remark} We remark the following  
			\begin{itemize}
				\item The existence of a minimizer of the limit homogenized energy \eqref{EQ1052} is given by the fact the associated quadratic form is coercive and lower semi-continuous.
				\item		Note that the homogenized energy $\GJ^{B,hom}(\Uc)$ is the vK limit energy.
			\end{itemize}
		\end{remark}	
		
		\section{Case 2: Asymptotic behavior under matrix pre-strain}\label{Sec08}
		In this section the effective elastic behavior of the composite plate under matrix pre-strain is derived, for that it is assumed that there exist pre-strain in the matrix part which is modeled by a multiplicative decomposition of the deformation gradient $\nabla v$ with the matrix field $A_\e: \O_\e\to\R^{3\X3}$.
		
		Above way of decomposition has first been introduced in the context of finite strain plasticity \cite{pretrain01} and \cite{prestrain02}, in which they decompose the deformation gradient into two parts one elastic and another inelastic which is written as $\nabla v=F_{el}A_\e$ for the strain, where $A_\e$ the pre-strain tensor (also called the inelastic part) and $F_{el}$ the elastic part, which is the only part that enters the elastic energy functional. Hence, we say that total elastic energy \eqref{EQJ101} describes a composite plate with heterogeneous matrix pre-strain $A_\e(\cdot)$.
		\subsection{The non-linear elasticity problem in Von-K\'arm\'an regime}\label{Ssec81}
		We consider the following energy functional
		\begin{equation}\label{EQJ101}
			\GJ^*_\e(v)=\int_{\O_\e}\wh{W}_\e(x,(\nabla v) A_\e^{-1})dx-\int_{\O_{\e}} f_{\e}(x)\cdot(v(x)-\text{\bf id}(x)),\quad \text{for all}\;\;v\in\GV_\e.
		\end{equation}
		
		The local elastic energy $\wh{W}$ and the external forces $f_\e$ are given by \eqref{Eq62} and \eqref{AFB1} with the pre-strain 				\begin{equation}\label{EQ102}
			A_\e(x)=\left\{
			\begin{aligned}
				&\GI_3,\quad &&x=(x',x_3)\in\O_\e^B,\\
				&\GI_3-\e B_\e(x'),\quad && x=(x',x_3)\in\O_\e^M,
			\end{aligned}\right.
		\end{equation}	
		with  $B_\e\in L^\infty(\o)^{3\X3}$.

		\begin{assumption} \label{Ass01}
			We suppose that for  the small pre-strain, there exists a $B\in  L^2(\o\X Y_\kappa^M)^{3\X3}$ independent of $y$ such that 
			\begin{equation}\label{EQ84}
				\e\|B_\e\|_{L^\infty(\o)}\leq C_P,\quad\text{and}\quad\Pi_\e(B_\e)\rightharpoonup B\;\;\text{weakly in}\;\;L^2(\o\X Y_\kappa^M)^{3\X3}.
			\end{equation} The constant does not depend on $\e$.\\
			Moreover, we assume that (see estimate \eqref{KornP2}$_1$)
			$$C_MC_P\leq{1\over 2}.$$
		\end{assumption}
		Since the pre-strain is small, then we can simplify the decomposition since for $\e<<1$, we have the inverse of $A_\e$ given by the Neumann series.  Thus, if $C_P<1$ then we have 
		$$\begin{aligned}
			A_\e^{-1}=\sum_{n=0}^\infty  (\e B_\e)^n=(\GI_3+\e B_\e)+O(\e^2),\\
			\text{det}(A_\e)= \text{det}(\GI_3-\e B_\e)=1+O(\e),
		\end{aligned}\quad\text{for a.e.}\;x\in \O_\e^M.$$ 
		Since we do not have $O(\e^2)$ for the $L^\infty$ norm but it is true for the $L^2$ norm. Hence, we replace $A^{-1}_\e $ by $\GI_3+\e B_\e$.
	
		Similarly, like in Subsection \ref{Ssec71}, we have as a consequence of \eqref{Eq62}--\eqref{SymCon} we have 
	$$
	\begin{aligned}
		&c_0 \big(\e^2\|dist(\nabla v A_\e^{-1}, SO(3))\|^{2}_{L^{2}(\O^M_\e)}+\|dist(\nabla v, SO(3))\|^{2}_{L^{2}(\O^B_\e)}\big)\\
		& \hskip 10mm \leq c_0\big(\e^2\|\nabla v A_\e^{-1}(\nabla v A_\e^{-1})^T-\GI_3\|^2_{L^2(\O^M_\e)}+\|\nabla v(\nabla v)^T-\GI_3\|^2_{L^2(\O^B_\e)}\big)\leq   \int_{\O_{\e}} \wh W_\e\big(\cdot,\nabla v A_\e^{-1}\big)\, dx.
	\end{aligned}
	$$
	First, observe that		
	$$\begin{aligned}
		dist(\nabla v,SO(3))=dist(\nabla v(\GI_3+\e B_\e)-\e\nabla v B_\e,SO(3))=\inf_{G\in SO(3)}|\nabla v(\GI_3+\e B_\e)-\e\nabla v B_\e-G|_F\\
		\leq \inf_{G\in SO(3)}|\nabla v(\GI_3+\e B_\e)-G|_F+|\e\nabla v B_\e|_F=dist(\nabla v(\GI_3+\e B_\e), SO(3))+|\e\nabla v B_\e|_F,\;\;
		\hbox{a.e. in } \O^M_\e.
	\end{aligned}$$
	With the above inequality with definition of $A_\e$ and estimate \eqref{KornP2} we get
	\begin{align*}
		&\|dist(\nabla v,SO(3))\|_{L^2(\O_\e^M)}\leq \|dist(\nabla v(\GI_3+\e B_\e),SO(3))\|_{L^2(\O_\e^M)}+\e\|\nabla v B_\e\|_{L^2(\O_\e^M)}\\
		&\leq \|dist(\nabla v(\GI_3+\e B_\e),SO(3))\|_{L^2(\O_\e^M)}+\e\| B_\e\|_{L^2(\O_\e^M)}+\e\|B_\e\|_{L^\infty(\o)}\|\nabla v-\GI_3\|_{L^2(\O^M_\e)}\\
		&\leq \|dist(\nabla v(\GI_3+\e B_\e),SO(3))\|_{L^2(\O_\e^M)}+C_P C_M\|dist(\nabla v,SO(3))\|_{L^2(\O_\e^M)}\\
		&+ {C_P C_M\over \e}\|dist(\nabla v,SO(3))\|_{L^2(\O_\e^B)}  + C \e^{3/2}.
	\end{align*}
	As consequence of the above estimate, we get
	\begin{equation*}
		\begin{aligned}
			&\|dist(\nabla v,SO(3))\|_{L^2(\O_\e^B)}+\e\|dist(\nabla v,SO(3))\|_{L^2(\O_\e^M)}\\
			&\hskip 25mm\leq (C_P C_M+1) \|dist(\nabla v,SO(3))\|_{L^2(\O_\e^B)}+\e C_MC_P\|dist(\nabla v,SO(3))\|_{L^2(\O_\e^M)}\\
			&\hskip 35mm + \e\|dist(\nabla v(\GI_3+\e B_\e),{SO(3)})\|_{L^2(\O_\e^M)}+C\e^{5/2}.
		\end{aligned}
	\end{equation*}
	For the Assumption \ref{Ass01}, we have $C_PC_M\leq {1\over 2}$ which yields
	\begin{equation*}
		\begin{aligned}
			&\|dist(\nabla v,SO(3))\|_{L^2(\O_\e^B)}+\e\|dist(\nabla v,SO(3))\|_{L^2(\O_\e^M)}\\
			&\hskip 25mm\leq \|dist(\nabla v,SO(3))\|_{L^2(\O_\e^B)}+ \e\|dist(\nabla v(\GI_3+\e B_\e),SO(3))\|_{L^2(\O_\e^M)}+C\e^{5/2}.
		\end{aligned}
	\end{equation*}
		Similarly, as Subsections \ref{Ssec71}, with the above estimates we have 
	\begin{equation}\label{EQ87}
		c_0 \big(\e^2\|dist(\nabla v,SO(3))\|^{2}_{L^{2}(\O^M_\e)}+\|dist(\nabla v, SO(3))\|^{2}_{L^{2}(\O^B_\e)}\big)-\int_{\O_\e}f_\e\cdot(v-\text{\bf id})dx \leq {\GJ^*}_{\e}(v)+C_1\e^5.
	\end{equation}
		Here, note that using the fact that $a(\cdot)\in L^\infty(\Yc)$ we have
		$$\begin{aligned}
			\GJ_\e^*(\text{\bf id})=\int_{\O_\e}\wh{W}(x,A_\e^{-1})dx=\int_{\O_\e^M}\wh{W}(x,\GI_3+\e B_\e)dx&\leq C\e^2\|(\GI_3+\e B_\e)^T(\GI_3+\e B_\e)-\GI_3\|^2_{L^2(\O_\e^M)}\\
			&= C\e^2\|\e B_\e^T+\e B_\e+\e^2 B_\e^T B_\e\|^2_{L^2(\O_\e^M)}.
		\end{aligned}$$
		Since, from the estimate \eqref{EQ84}, we have
		\begin{equation}\label{EQ88}
			\|\e B_\e^T+\e B_\e+\e^2 B_\e^T B_\e\|_{L^2(\O_\e^M)}\leq  2\e \|B_\e\|_{L^2(\O_\e^M)}+\e\big(\e\|B_\e\|_{L^\infty(\O_\e^M)}\big)\|B_\e\|_{L^2(\O_\e^M)}\leq C\e^{3/2}
		\end{equation} which imply
		$$\|(\GI_3+\e B_\e)^T(\GI_3+\e B_\e)-\GI_3\|_{L^2(\O^M_\e)}\leq C\e^{3/2}$$ and then
		$$\GJ_\e^*(\text{\bf id})\leq C\e^5.$$
		We choose only deformations $v\in\GV_\e$ such that $\GJ_\e^*(v)\leq \GJ_\e^*(\text{\bf id})$. Then, the non-linear elasticity problem reads as
		$$m_\e^*=\inf_{v\in\GV_\e}\GJ_\e^*(v)\leq \GJ_\e^*(\text{\bf id}).$$
		\begin{remark}
			Since, we have not assumed $A_\e \in SO(3)$, we have $0<\GJ_\e^*(\text{\bf id})\leq C\e^5$ which is in accordance with the mechanical property i.e. since the composite plate has pre-strain, it has some stored elastic energy.
		\end{remark}
		Proceeding as \eqref{Ssec72}, we have the following estimates:
		\begin{lemma}
			Let $v\in \GV_{\e}$ be a deformation such that $\GJ^*_{\e}(v)\leq\GJ^*(\text{\bf id})$. Assume the forces fulfill \eqref{AFB1} and \eqref{EQ74}, one has 
			\begin{equation}\label{EQ86}
				\begin{aligned}
					&\|dist(\nabla v, SO(3))\|_{L^{2}(\O^B_\e)}+\e\|dist(\nabla v, SO(3))\|_{L^{2}(\O^M_\e)}\leq C\e^{5/2},\\
					&\|\nabla v(\nabla v)^T-\GI_3\|_{L^2(\O^B_\e)}+\e \|\nabla v (\GI_3+\e B_\e)(\nabla v(\GI_3+\e B_\e))^T-\GI_3\|_{L^2(\O^M_\e)}\leq C\e^{5/2}.
				\end{aligned}
			\end{equation}	
			The constants do not depend on $\e$.
		\end{lemma}
		\begin{proof}
			Using \eqref{EQ87} gives rise to the estimate
			$$
			c_0\Big(\|dist(\nabla v, SO(3))\|^{2}_{L^{2}(\O^B_\e)}+\e^2\|dist(\nabla v, SO(3))\|^{2}_{L^{2}(\O^M_\e)}\Big)\leq\left|\int_{\O_\e} f_\e\cdot(v-\text{\bf id})dx\right|+C\e^5.
			$$
			We get the estimates \eqref{EQ86}$_{1}$ following the same arguments as in Lemma \ref{GOWL54} with the estimate \eqref{EQ87}.\\					
			Furthermore, the estimate \eqref{EFBM}$_1$ can be employed to arrive at 
			\begin{equation*}
				\begin{aligned}
					c_0\big(\big\|(\nabla v)^T \nabla v -\GI_3\|^2_{L^2(\O^B_\e)}+\e^2\|(\nabla v A_\e^{-1})^T\nabla v A_\e^{-1} -\GI_3 \big\|^2_{L^2(\O^M_\e)}\big) 
					\leq \int_{\O_\e} \wh{W}(x,\nabla v(x)A_\e^{-1})dx\\
					\leq \int_{\O_\e} f_\e(x)\cdot(v-\text{\bf id})(x)dx+C\e^5 \leq C\e^5,
				\end{aligned}
			\end{equation*}
			which in turn leads to
$$
		\|\nabla v(\nabla v)^T-\GI_3\|_{L^2(\O^B_\e)}+\e \|\nabla v (\GI_3+\e B_\e)(\nabla v(\GI_3+\e B_\e))^T-\GI_3\|_{L^2(\O^M_\e)}\leq C\e^{5/2}.
$$
			This completes our proof.
		\end{proof}
		
		As a consequence of the above Lemma, we have that there exists a strictly positive constants $c^*$ and $C^*$ independent of $\e$ such that
		$$-c^*\e^5\leq \GJ^*_\e(v)\leq \GJ^*(\text{\bf id})\leq C^*\e^5.$$
		Now with definition of $m_\e^*$, we have
		$$-c^*\leq {m_\e^*\over \e^5}\leq C^*.$$
		
		So, our aim is to give the limit of the re-scaled sequence $\ds \left\{{m^*_\e\over \e^5}\right\}_\e$, for that we first give the asymptotic behavior of the Green St. Venant's strain tensor.
		
		\subsection{Asymptotic behavior of the non-linear strain tensor with pre-strain}
		We proceed similarly like in Section \ref{Sec07}, so we have from \eqref{AFLP1} that $\{v_\e\}_\e$ is a sequence of deformations in $\GV_\e$ satisfying
		\begin{equation}\label{Ass02}
			\|dist(\nabla v_\e, SO(3))\|_{L^{2}(\O^B_\e)}+\e\|dist(\nabla v_\e, SO(3))\|_{L^{2}(\O^M_\e)}\leq C\e^{5/2}. 
		\end{equation}
		We give asymptotic behavior of the sequence $\{\nabla v_\e(\GI_3+\e B_\e)(\nabla v_\e (\GI_3+\e B_\e))^T-\GI_3\}_\e$.	
		
		So, we now give the limit of the non-linear strain tensor with matrix pre-strain, which is present in the total elastic energy.\\
		We recall the following classical result:
		\begin{lemma}\label{swc}
			Let $\{f_\e\}_\e$ and $\{g_\e\}_\e$ be sequence in $L^2(\o\X \Yc)$ such that
			$$ f_\e \to f\quad \text{strongly in}\;\;L^2(\o\X \Yc),\quad\text{and}\quad g_\e\rightharpoonup g\quad\text{weakly in}\;\;L^2(\o\X\Yc),$$
			then we have
			$$ f_\e g_\e \rightharpoonup fg \quad \text{weakly in}\quad L^1(\o\X\Yc).$$
		\end{lemma}	
		\begin{lemma}\label{fniq1}
			Suppose $A,B \in \GM_3$, we have
			$$ \left|ABA^T\right|_F\leq |B|_F\,|AA^T|_F.$$
		\end{lemma}
		\begin{proof} First, we recall the expressions of the Frobenius norm and the scalar product associated.
			For all $A$, $B \in \GM_3$, we have
			$$<A,B>_F=tr(A^TB)=tr(AB^T),\qquad |A|_F=\sqrt{<A,A>_F}=\sqrt{tr(A^TA)}.$$
			Using the Cauchy Schwartz inequality on the Frobenius inner product yields
			$$ |tr(AA^TB)|=|<AA^T,B>_F|\leq |B|_F|AA^T|_F,\quad\text{and}\quad |tr(A^TAB)|=|<A^TA,B>_F|\leq |B|_F|A^TA|_F,$$
			and we also have $|AB|_F\leq |A|_F|B|_F$.\\
			Using the above inequality, we obtain
			$$\begin{aligned}
				\left|ABA^T\right|^2_F=&|tr(\underbrace{AB}\,\underbrace{A^TA}\,\underbrace{B^TA^T})|=|tr(\underbrace{(AB)^TAB}\,\underbrace{A^TA})|\leq  |A^TA|_F\,|(AB)^TAB|_F\\
				=& |A^TA|_F\sqrt{tr((AB)^TAB(AB)^TAB)}=|A^TA|_F\;\sqrt{(tr(\underbrace{BB^T}\underbrace{A^TABB^TA^TA})}\\
				\leq&|A^TA|_F\sqrt{|\underbrace{A^TA}\underbrace{BB^T}\underbrace{A^TA}|_F|BB^T|_F}\leq |A^TA|_F^2|BB^T|_F\leq |B|^2_F|A^TA|_F^2.
			\end{aligned}$$	
			Hence, we got $|ABA^T|_F\leq  |B|_F\,|AA^T|_F$, since $|A^TA|_F=|AA^T|_F$.
		\end{proof} 
		We recall that for matrix function $A(\cdot,\cdot)\in [L^2(\o\X \Yc)]^{3\X3}$, we have
		\begin{equation}\label{FL2N}
			\|A\|_{L^2(\o\X\Yc)}=\sqrt{\int_{\o\X\Yc}|A(x',y)|^2_F dx'dy}.
		\end{equation} 
		
		\begin{lemma}\label{L102}
			For the same subsequence as in the Lemma \ref{CFuu1}, we have
			\begin{equation}\label{EQ107}
				\begin{aligned}
					{1\over 2\e^2} \Pi_\e\left(\nabla v_\e (\nabla v_\e)^T-\GI_3\right) &\rightharpoonup E(\Uc)+e_y(\wh{\fu})\;&& \text{weakly in}\quad L^2(\o\X Y_\kappa^B)^{3\X3},\\
					{1\over 2\e} \Pi_\e\left(\nabla v_\e (\GI_3+\e B_\e)(\GI_3+\e B_\e)^T(\nabla v_\e)^T-\GI_3\right) &\rightharpoonup   e_y(\Gu^M)+ Sym(B) \;&& \text{weakly in}\quad L^2(\o\X Y_\kappa^M)^{3\X3},
				\end{aligned}
			\end{equation}
			where the symmetric matrix $E(\Uc)$, $\wh{\fu}$ and $\Gu^M$ are all same as in Lemma \ref{CFGLST2}.
		\end{lemma}			
		\begin{proof}
			{\bf Step 1:} The convergence \eqref{EQ107}$_1$ is a direct consequence of the Lemma \ref{CFGLST2}. Here, we give the proof for the convergence \eqref{EQ107}$_2$.
			We have
			\begin{multline}\label{Cc2}
				\nabla v_\e (\GI_3+\e B_\e)(\GI_3+\e B_\e)^T(\nabla v_\e)^T-\GI_3=\big(\nabla v_\e (\nabla v_\e)^T-\GI_3\big)+\e\big(\nabla v_\e B_\e^T(\nabla v_\e)^T+\nabla v_\e B_\e (\nabla v_\e)^T\big)\\
				+\e^2\nabla v_\e B_\e B_\e^T (\nabla v_\e)^T.
			\end{multline} 
			Using the convergences \eqref{CFGLST1}$_2$, we get
			$$ \begin{aligned}
				{1\over 2\e} \Pi_\e\left(\nabla v_\e (\nabla v_\e)^T-\GI_3\right) &\rightharpoonup e_y(\Gu^M)\quad \text{weakly in}\;\;L^2(\o\X Y_\kappa^M)^{3\X3}.
			\end{aligned}$$
			{\bf Step 2:} We prove the convergence
			\begin{equation}\label{Cc1}
				{1\over 2}\left( \Pi_\e\left(\nabla v_\e B^T_\e (\nabla v_\e)^T\right)+\Pi_\e\left(\nabla v_\e B_\e (\nabla v_\e)^T\right)\right) \rightharpoonup Sym(B)\quad\text{weakly in}\;\;L^2(\o\X Y_\kappa^M)^{3\X3}.
			\end{equation}
			Using the strong convergence \eqref{C920}, the weak convergence of $B_\e$ from Assumption \ref{Ass01} and  Lemma \ref{swc}, we have 
			$$ \Pi_\e(\nabla v_\e B_\e^T)=\Pi_\e(\nabla v_\e)\Pi_\e(B_\e^T)\rightharpoonup B^T\quad\text{weakly in}\;\;L^1(\o\X Y_\kappa^M)^{3\X3}.$$
			
			Now, we show that $\{\Pi_\e(\nabla v_\e B^T_\e)\}_\e$ is uniformly bounded in $L^2(\o\X Y_\kappa^M)^{3\X3}$, then we will have the above convergence is weak in $L^2(\o\X Y_\kappa^M)^{3\X3}$. From the estimates \eqref{KornP2}$_1$-\eqref{Ass02} and property \eqref{EQ722} with the Assumption \ref{Ass01}, we obtain
			$$
			\begin{aligned}
				&\|\Pi_\e\big((\nabla v_\e-\GI_3) B^T_\e\big)\|_{L^2(\o\X Y_\kappa^M)}\leq C\|B_\e\|_{L^\infty(\o)}\|\Pi_\e(\nabla v_\e-\GI_3)\|_{L^2(\o\X Y_\kappa^M)}\leq C,\\
			\end{aligned}
			$$
			which imply
			\begin{equation}\label{Es810}
				\|\Pi_\e(\nabla v_\e B^T_\e)\|_{L^2(\o\X Y_\kappa^M)}\leq  \|\Pi_\e\big((\nabla v_\e-\GI_3) B^T_\e\big)\|_{L^2(\o\X Y_\kappa^M)}+ \|\Pi_\e(B^T_\e)\|_{L^2(\o\X Y_\kappa^M)}\leq C.
			\end{equation}
			
			So, we got
			$$  \Pi_\e(\nabla v_\e  B_\e^T) \rightharpoonup B^T \quad\text{weakly in}\;\;L^2(\o\X Y_\kappa^M)^{3\X3}.$$
			Similarly, using the above weak convergence and Lemma \ref{swc}, we get
			$$ \Pi_\e(\nabla v_\e B_\e^T(\nabla v_\e)^T)=\Pi_\e(\nabla v_\e B_\e^T)\Pi_\e((\nabla v_\e)^T)\rightharpoonup B^T\quad\text{weakly in}\;\;L^1(\o\X Y_\kappa^M)^{3\X 3}.$$
			From the estimates \eqref{EWA1}, \eqref{Ass02} and the Assumption \ref{Ass01}, we get
			$$\begin{aligned}
				\|(\nabla v_\e-\GI_3)(\nabla v_\e-\GI_3)^T\|_{L^2(\O_\e)}\leq \|\nabla v_\e(\nabla v_\e)^T-\GI_3\|_{L^2(\O_\e)}+2\|\nabla v_\e-\GI_3\|_{L^2(\O_\e)} \leq C\e^{3/2},
			\end{aligned} $$
			which implies
			\begin{equation}\label{Cc4}
				\|\Pi_\e\left((\nabla v_\e-\GI_3) (\nabla v_\e-\GI_3)^T\right)\|_{L^2(\o\X \Yc)}\leq C\e.
			\end{equation}
			Using the Lemma \ref{fniq1}, \eqref{EQ84}$_1$,  the equality \eqref{FL2N} and the fact that $\|\Pi_\e\big(\nabla v_\e-\GI_3\big)\|_{L^2(\o\X\Yc)}\leq C\e$ (from the estimates \eqref{KornP2}$_2$, \eqref{KornB1} and \eqref{EQ722}) we have
			$$ \begin{aligned}
				&\|\Pi_\e\left((\nabla v_\e-\GI_3) B_\e^T (\nabla v_\e-\GI_3)^T\right)\|^2_{L^2(\o\X \Yc)}\\
				&\hskip 35mm=\int_{\o\X\Yc}\left|\Pi_\e\left((\nabla v_\e(x',y)-\GI_3) B_\e^T(x') (\nabla v_\e(x',y)-\GI_3)^T\right)\right|^2_Fdx'dy\\
				& \hskip 35mm\leq \|B_\e\|^2_{L^\infty(\o)}\int_{\o\X\Yc}\left|\Pi_\e\left((\nabla v_\e(x',y)-\GI_3)((\nabla v_\e(x',y)-\GI_3))^T\right)\right|^2_Fdx'dy\\
				& \hskip 35mm\leq {C\over \e^2}\int_{\o\X\Yc}\left|\Pi_\e\left((\nabla v_\e(x',y)-\GI_3)((\nabla v_\e(x',y)-\GI_3))^T\right)\right|^2_Fdx'dy \leq C
			\end{aligned}$$
			which implies 
			$$ \|\Pi_\e\left((\nabla v_\e-\GI_3) B_\e^T (\nabla v_\e-\GI_3)^T\right)\|_{L^2(\o\X Y_\kappa^M)}\leq C.$$ 
			Similarly, like the estimate \eqref{Es810}, we have
			\begin{equation}\label{Es813}
				\|\Pi_\e\left((\nabla v_\e B_\e)^T\right)\|_{L^2(\o\X Y_\kappa^M)}\leq C.
			\end{equation}	 
			From the above estimates we have
			$$ \begin{aligned}
				&\|\Pi_\e\left(\nabla v_\e B^T_\e (\nabla v_\e)^T\right)\|_{L^2(\o\X Y_\kappa^M)}\\
				&\hskip 20mm\leq \|\Pi_\e\left((\nabla v_\e-\GI_3) B_\e^T (\nabla v_\e-\GI_3)^T\right)\|_{L^2(\o\X Y_\kappa^M)}+\|\Pi_\e(B^T_\e)\|_{L^2(\o\X Y_\kappa^M)}\\
				&\hskip 35mm+\|\Pi_\e(\nabla v_\e B^T_\e)\|_{L^2(\o\X Y_\kappa^M)}+\|\Pi_\e\left((\nabla v_\e B_\e)^T\right)\|_{L^2(\o\X Y_\kappa^M)}\leq C.
			\end{aligned}$$
			From all the above estimates and convergences, we get
			$$\Pi_\e\left(\nabla v_\e B^T_\e (\nabla v_\e)^T\right) \rightharpoonup B^T\quad\text{weakly in}\;\;L^2(\o\X Y_\kappa^M)^{3\X3}.$$
			Hence, we obtain the convergence \eqref{Cc1}.\\
			
			{\bf Step 3:} We show that 
			\begin{equation}\label{Cc5}
				{\e\over 2} \Pi_\e\left(\nabla v_\e B_\e B_\e^T (\nabla v_\e)^T\right) \rightharpoonup 0\quad\text{weakly in}\;\;L^2(\o\X Y_\kappa^M)^{3\X3}.
			\end{equation}
			Using  Assumption \ref{Ass01} and \eqref{Es813}, we have
			$$ \|\Pi_\e\left(\nabla v_\e B_\e B_\e^T(\nabla v_\e)^T\right)\|_{L^1(\o\X Y_\kappa^M)}=\|\Pi_\e\left(\nabla v_\e B_\e (\nabla v_\e B_\e)^T\right)\|_{L^1(\o\X Y_\kappa^M)}\leq \|\Pi_\e(\nabla v_\e B_\e)\|^2_{L^2(\o\X Y_\kappa^M)}\leq C.$$
			Hence
			\begin{equation}\label{Cc3}
				{\e\over 2} \Pi_\e\left(\nabla v_\e B_\e B_\e^T (\nabla v_\e)^T\right) \to 0\quad\text{strongly in}\;\;L^1(\o\X Y_\kappa^M)^{3\X3}.
			\end{equation}
			Similarly, using Lemma \ref{fniq1} and proceeding as Step 2, we get
			\begin{multline*}
				\|\Pi_\e\left((\nabla v_\e-\GI_3) B_\e B_\e^T(\nabla v_\e-\GI_3)^T\right)\|^2_{L^2(\o\X \Yc)}\\
				\leq \|B_\e\|^4_{L^\infty(\o)} \int_{\o\X\Yc}\left|\Pi_\e\left((\nabla v_\e(x',y)-\GI_3)(\nabla v_\e(x',y)-\GI_3)^T\right)\right|^2_Fdx'dy\leq {C\over \e^2},
			\end{multline*} 
			which imply
			\begin{equation}\label{Es817}
				\|\Pi_\e\left((\nabla v_\e-\GI_3) B_\e B_\e^T(\nabla v_\e-\GI_3)^T\right)\|_{L^2(\o\X Y_\kappa^M)}\leq {C\over \e}.
			\end{equation}
			So, we obtain
			\begin{equation}\label{Es816}
				\begin{aligned}
					&\|\Pi_\e\left(\nabla v_\e B_\e B_\e^T(\nabla v_\e)^T\right)\|_{L^2(\o\X Y^M_\kappa)}\\
					&\hskip 10mm\leq \|\Pi_\e\left((\nabla v_\e-\GI_3) B_\e B_\e^T(\nabla v_\e-\GI_3)^T\right)\|_{L^2(\o\X Y_\kappa^M)}+\|\Pi_\e(B_\e B^T_\e)\|_{L^2(\o\X Y_\kappa^M)}\\
					&\hskip 30mm+\|\Pi_\e\left(B_\e(\nabla v_\e B_\e)^T\right)\|_{L^2(\o\X Y_\kappa^M)}+\|\Pi_\e\left(\nabla v_\e B_\e B_\e^T\right)\|_{L^2(\o\X Y_\kappa^M)}.
			\end{aligned} \end{equation}
			From the estimates \eqref{EQ84}, \eqref{Es810} and \eqref{Es813}, we get
			$$\begin{aligned}
				&\|\Pi_\e\left(B_\e(\nabla v_\e B_\e)^T\right)\|_{L^2(\o\X Y_\kappa^M)}\leq \|B_\e\|_{L^\infty(\o)}\|\Pi_\e\left((\nabla v_\e B_\e)^T\right)\|_{L^2(\o\X Y_\kappa^M)}\leq {C\over \e},\\
				&\|\Pi_\e\left(\nabla v_\e B_\e B_\e^T\right)\|_{L^2(\o\X Y_\kappa^M)}\leq \|B_\e^T\|_{L^\infty(\o)}\|\Pi_\e\left(\nabla v_\e B_\e \right)\|_{L^2(\o\X Y_\kappa^M)}\leq {C\over \e},\\
				&\|\Pi_\e(B_\e B^T_\e)\|_{L^2(\o\X Y_\kappa^M)}\leq \|B_\e^T\|_{L^\infty(\o)}\|\Pi_\e(B_\e)\|_{L^2(\o\X Y_\kappa^M)}\leq {C\over \e}.
			\end{aligned}$$
			The above estimates together with \eqref{Es816} and \eqref{Es817} imply
			$$\e\|\Pi_\e\left(\nabla v_\e B_\e B_\e^T(\nabla v_\e)^T\right)\|_{L^2(\o\X Y^M_\kappa)}\leq C.$$
			So, we have the weak convergence\eqref{Cc5}.\\
			
			Hence, from the convergences \eqref{Cc1}, \eqref{Cc3} and \eqref{Cc5}, we obtain the convergence \eqref{EQ107}$_2$.
		\end{proof}

		We are now in position to give the asymptotic behavior of the re-scaled sequence $\ds\left\{{m^*_\e \over \e^5}\right\}_\e$.
		
		\subsection{Asymptotic behavior of the sequence $\ds\left\{{m^*_\e \over \e^5}\right\}_\e$}
		The limit from the convergence $\eqref{EQ107}$ allows us to investigate the limit of the non-linear elasticity problem. So, we introduce the limit re-scaled elastic energy for the non-linear elasticity problem as
$$
			\GJ^*(\Uc, \wh{\fu},\Gu^M)=\GJ_B(\Uc,\wh{\fu})+\GJ^*_M(\Gu^M)-|\Yc|\int_{\o} f\cdot\Uc dx'- \sum_{\alpha=1}^2\int_{\o}f_\alpha\Big(\int_\Yc\Gu_\alpha^M dy\Big)dx',						
$$
		where we have set
$$
			\begin{aligned}
				\GJ_B(\Uc,\wh{\fu})&=\int_{\o}\int_{Y_\kappa^B}Q(y,E(\Uc)+e_y(\wh{\fu})) dydx',\\
				\GJ^*_M(\Gu^M)&=\int_{\o}\int_{Y_\kappa^M} Q(y,e_y(\Gu^M)+ Sym(B))dydx'.
			\end{aligned}
$$
		Let us define the limit displacement space as
		$$\begin{aligned}
			&\O_R=\o\setminus\overline{\gamma},\qquad 
			\D_0=\Big\{\Uc=\left(\Uc_1,\Uc_2,\Uc_3\right)\in H^1(\O_R)^2\X H^2(\O_R) \;\;|\;\; \Uc=\partial_\alpha\Uc_3=0\quad\text{a.e. on}\;\; \gamma\Big \},\\
			&\GH^1(Y^M_\kappa)\doteq\Big\{\phi\in H^1(Y^M_\kappa)\;|\; \phi=0\;\;\hbox{a.e. on }\; \partial Y_\kappa\X(-\kappa,\kappa)\Big\}.
		\end{aligned}$$ 
		Before showing the convergence of the problem with $\Gamma$-convergence, we first prove that the limit-functional $\GJ^*$ attains a minimum on  $\D=\D_0\X L^2(\O_R;H^1_{per,0}(Y_\kappa^B))^3\X L^2(\O_R;\GH^1(Y_\kappa^M))^3$.

		\begin{lemma}
			The functional $\GJ^*$ admits a minimum on $\D$.
		\end{lemma}		
		\begin{proof}
			Only the triplets $(\Uc,\wh{\fu},\Gu^M)\in \D$ such that $\GJ^*(\Uc,\wh{\fu},\Gu^M)\leq \GJ^*(0,0,0)\leq C$\footnote{It should be noted that $\GJ^*(0,0,0)>0$.} are considered.
			
			Let us set
			$$ m=\inf_{(\Uc,\wh{\fu},\Gu^M)\in\D} \GJ(\Uc,\wh{\fu},\Gu^M).$$  
			We have $m\in[-\infty,C]$.\\
			
			The proof is in line of the proof of the Lemma \ref{ExMin1}. 
		\end{proof}			
		
		The following theorem is the main result of this section. It characterizes the limit of the re-scaled infimum of the total energy  $\ds {m^*_\e\over \e^5}={1\over \e^5}\inf_{v_\e\in\GV_\e}\GJ^*_\e(v_\e)$  as the minimum of the limit energy $\GJ^*$ over the space $\D$. 
		\begin{theorem}
			Under the Assumptions \ref{Ass01}, we have
			\begin{equation*}
				m^*=\lim_{\e\to0}{m^*_\e \over \e^5} = \min_{(\Uc,\wh{\fu},\Gu^M)\in \D} \GJ^*(\Uc,\wh{\fu},\Gu^M).
			\end{equation*}
		\end{theorem}	
		\begin{proof}
			The proof is same like done for the Theorem \ref{MainConR}.
		\end{proof}
		
		\subsection{The cell problems}
		To obtain the cell problems, we consider the variational formulation for $\wh{\fu}$ and $\Gu^M$ associated to the functional $\GJ_B$ and $\GJ^*_M$, respectively. We recall the limit elasticity energy as 	
		$$
		\GJ^*(\Uc, \wh{\fu},\Gu^M)=\GJ_B(\Uc,\wh{\fu})+\GJ^*_M(\Gu^M)-|\Yc|\int_{\O} f\cdot\Uc dx'- \sum_{\alpha=1}^2\int_{\o}f_\alpha\Big(\int_\Yc\Gu_\alpha^M dy\Big)dx',	
		$$
		we have set 
$$
			\begin{aligned}
				\GJ_B(\Uc,\wh{\fu})&=\int_{\o}\int_{Y_\kappa^B}a(y)(E(\Uc)+e_y(\wh{\fu})) : (E(\Uc)+e_y(\wh{\fu})) dydx',\\
				\GJ^*_M(\Gu^M)&=\int_{\o}\int_{Y_\kappa^M} a(y)\left(e_y(\Gu^M)+ Sym(B)\right) : \left(e_y(\Gu^M)+ Sym(B)\right)dydx',
			\end{aligned}
$$
		Similarly, proceeding as Subsection \ref{cellPro}, we assume $(\Uc, \wh{\fu},\Gu^M)$ be a minimizer, then we have
		$$\GJ^*(\Uc, \wh{\fu},\Gu^M)\leq \GJ^*(\Uc, \wh{\fu}+t\wh{\fv},\Gu^M+t\Gv^M),\quad \forall t\in \R,\quad \forall (\wh{\fv}, \Gv^M)\in L^2(\O_R;H^1_{per,0}(Y_\kappa^B))^3\X L^2(\O_R;\GH^1(Y_\kappa^M))^3.$$
		Using the fact that $\GJ_B$ and $\GJ^*_M$ are quadratic forms in $e_y(\wh{\fu})$ and $e_y(\Gu^M)$ respectively over a Hilbert-space and we obtain:
$$
			\begin{aligned}
				&\text{Find $\wh{\fu}\in L^2(\o;H^1_{per,0}(Y_\kappa^B))^3$ and $\Gu^M\in L^2(\o;\GH^1(Y_\kappa^M))^3$ such that for all $\wh{w} \in H^1_{per,0}(Y_\kappa^B)^3$ and}\\
				&\hbox{$\Gw \in \GH^1(Y_\kappa^M)^3$, we have}\\	
				&\hspace{10mm}\left.
				\begin{aligned}
					&\int_{Y_\kappa^B} a(y)\left(E(\Uc)+e_y(\wh{\fu})\right) : e_y(\wh{w}) dy=0,\\
					&\int_{Y_\kappa^M} a(y)\left(e_y(\Gu^M)+Sym(B)\right) : e_y(\Gw) dy=\sum_{\alpha=1}^2 f_\alpha\left(\int_{Y_\kappa^M}\Gw_\alpha dy\right),
				\end{aligned}\right\}\quad\hbox{a.e. in $\o$}.
			\end{aligned}
$$ So, we have linear problems. Hence we have a unique solution for the above problems. Therefore,
		we get from the above equations \eqref{EQ761}, and the cell problems in $Y_\kappa^B$ is the same as \eqref{Eq125} which imply
$$
			\begin{aligned}
				&\wh{\fu}(x',y)=\Zc_{\alpha\beta}(x')\chi^m_{\alpha\beta}(y)+\partial_{\alpha\beta}\Uc_3(x')\chi^b_{\alpha\beta}(y),\quad &&\hbox{for a.e. } (x',y)\in \o\X Y^B_\kappa.
			\end{aligned}
$$
		Similarly, we express $\Gu^M$ in terms of correctors $\chi^p=(\chi^p_1,\chi^p_2)\in [\GH^1(Y_\kappa^M)^3]^2$. Before that we give the problem for which the $\chi^p$ are  the solution
$$
			\int_{Y_\kappa^M} a(y)\left( e_y(\chi^p_\alpha)+M^{ij}+Sym(B)M^{ij}\right) : e_y(\Gw) dy = \int_{Y_\kappa^M}\Gw_\alpha dy \quad \forall \Gw\in \GH^1(Y_\kappa^M)^3.
$$
		We use the variational formulation \eqref{Eq1042}$_2$ and above equation   which implies
$$
			\begin{aligned}
				\Gu^M(x',y)= \sum_{\alpha=1}^2 f_\alpha(x')\chi^p_\alpha(y)\qquad \hbox{for a.e. } (x',y)\in \o\X Y^M_\kappa. 
			\end{aligned}
$$
		We set the homogenized coefficients in  $Y_\kappa^M$ as
$$
			\begin{aligned}
				&a^{M,hom}_{i'j'k'l'} = \frac{1}{|Y_\kappa^M|} \int_{Y_\kappa^M} a_{ijkl}(y) \left[M^{i'j'}_{ij} + 
				e_{y,ij}({\chi}_{\alpha}^p)\right]M^{k'l'}_{kl} dy,\\
				&d^{M,hom}_{i'j'k'l'} = \frac{1}{|Y_\kappa^M|} \int_{Y_\kappa^M} a_{ijkl}(y) \left[M^{i'j'}_{ij}\right]M^{k'l'}_{kl} dy,					
			\end{aligned}								
$$
		where the matrices $M_{ij}$ are given by \eqref{EQMatrix}.
		Hence, the homogenized energy is defined by
		\begin{equation}\label{HomPre}
			\GJ_{vK}^{hom^*}(\Uc)=|Y_\kappa^B|\GJ^{B, hom}(\Uc)-|\Yc|\int_{\o}f\cdot \Uc dx'+|Y_\kappa^M|\GJ^{M,hom},
		\end{equation}
		where
$$
			\begin{aligned}
				\begin{aligned}
					&\GJ^{B, hom}(\Uc)
					&&= {1\over 2}\int_{\o}\Big(a^{B,hom}_{\alpha\beta\alpha^\prime\beta^\prime}\Zc_{\alpha\beta}\Zc_{\alpha^\prime\beta^\prime}+ b^{B,hom}_{\alpha\beta\alpha^\prime\beta^\prime}\Zc_{\alpha\beta}\partial_{\alpha^\prime\beta^\prime}\Uc_3+c^{B,hom}_{\alpha\beta\alpha^\prime\beta^\prime}\partial_{\alpha\beta}\Uc_3\partial_{\alpha^\prime\beta^\prime}\Uc_3\Big)dx',\\
					&\GJ^{M,hom}&&={1\over 2}\int_{\o}\Big(a^{M,hom}_{i'j'k'l'}\GF Sym(B) +d^{M,hom}_{i'j'k'l'}Sym(B)Sym(B)\Big)dx',
				\end{aligned}
			\end{aligned}
$$
		with
		$$ \Zc_{\alpha\beta}=e_{\alpha\beta}(\Uc)+{1\over 2}\partial_{\alpha}\Uc_3\partial_{\beta}\Uc_3,\quad\text{and}\quad\GF=\sum_{\alpha=1}^2 f_\alpha.$$
		The minimizer of this homogenized energy functional \eqref{HomPre} satisfies the variational problem:
		\begin{equation*}
			\begin{aligned}
				&\text{Find } \Uc\in \D_0\text{ such that for all } \Wc\in\D_0:\\
				&\hskip 25mm \int_{\o}\Big(a^{B,hom}_{\alpha\beta\alpha^\prime\beta^\prime}\Zc_{\alpha\beta}(\Uc)\Zc_{\alpha^\prime\beta^\prime}(\Wc)+ {b^{B,hom}_{\alpha\beta\alpha^\prime\beta^\prime}\over 2}\left(\Zc_{\alpha\beta}(\Uc)\partial_{\alpha^\prime\beta^\prime}\Wc_3+\Zc_{\alpha\beta}(\Wc)\partial_{\alpha^\prime\beta^\prime}\Uc_3\right)\\
				&\hskip 75mm +c^{B,hom}_{\alpha\beta\alpha^\prime\beta^\prime}\partial_{\alpha\beta}\Uc_3\partial_{\alpha^\prime\beta^\prime}\Uc_3\Big)dx' =|\Yc|\int_{\o}f\cdot\Wc dx'.
			\end{aligned}
		\end{equation*}
		We introduce the homogenized coefficients in the limit energy which  gives the convergence result for the homogenized energy. From the Lemma \ref{L716}, the quadratic form associated to $\GJ_{vK}^{hom}$ is coercive.
		
		We are ready to give the main result of this subsection.
		\begin{theorem}
			Under the assumption on the forces \eqref{AFB1} and \eqref{EQ74}, the problem
$$
				\min_{\Uc\in\D_0} \GJ_{vK}^{hom^*}(\Uc)
$$
			admits at least a solution. Moreover, one has
			$$ m^*=	\lim_{\e\to0}{m^*_\e \over \e^5} = \min_{(\Uc,\wh{\fu},\Gu^M)\in \D} \GJ^*(\Uc,\wh{\fu},\Gu^M)= 	\min_{\Uc\in\D_0} \GJ_{vK}^{hom^*}(\Uc).$$	
		\end{theorem}
		\begin{remark} We remark the following  
			\begin{itemize}
				\item The existence of a minimizer of the limit homogenized energy \eqref{HomPre} is given by the fact that the associated quadratic form is coercive and lower semi-continuous.
				\item Note that the homogenized energy \eqref{HomPre} has an addition term due to the matrix pre-strain, which is absent on the homogenized energy \eqref{EQ949}, showing that for the soft matrix to have any effect in the homogenized model pre-strain is necessary.
				\item We have
				$$\GJ^*(0,0,0)= \int_{\o}\int_{Y_\kappa^M} a(y)\;Sym(B) :\;Sym(B)dydx'= {1\over 2}\int_{\o}d^{M,hom}_{i'j'k'l'}Sym(B)Sym(B)dx'.$$
			\end{itemize}
		\end{remark}	
		We end this subsection by showing that our homogenized vK plate, obtained through the incorporation of small matrix pre-strain (Case 2), exhibits orthotropic behavior\footnote{We know that this is true for Von-K\'arm\'an and linear regime, since the homogenized coefficients depend only on the cell problems.}.
		
		We know that for isotropic homogenized materials, we have from \cite{elasticityO} that
		$$ a_{ijkl}(y)=\lambda\delta_{ik}\delta_{jl}+\mu\delta_{ij}\delta_{kl}+\mu\delta_{il}\delta_{kj},\quad \text{for all}\;\;y\in\Yc.$$
		where $\lambda>0$, $\mu>0$ are the material Lam\'e constants and $\delta_{ij}$ is the Kronecker symbol.
%		
%		that the relation between the linearized strain tensor and stress tensor, which is given by
%$$
%			\sigma(u)=\lambda Tr(e(u))\GI_3+2\mu e(u),
%$$
%		where $\GI_3$ is the unit $3\X3$ matrix and

		\begin{lemma}
			Let us assume that the plate is made from isotropic and homogeneous material, then we have
$$
				b^{B,hom}_{\alpha\beta\alpha^\prime\beta^\prime} =0,\qquad \forall\;(\alpha\beta\alpha^\prime\beta^\prime)\in \{1,2\}^4,
$$
			and also have
$$
				\begin{aligned}
					a^{B, hom}_{1111}=a^{B, hom}_{2222}\quad\text{and}\quad a_{\alpha\alpha12}^{B, hom}=0,\quad \alpha\in\{1,2\},\\
					c_{1111}^{B, hom}=c_{2222}^{B, hom}\quad\text{and}\quad c^{B, hom}_{\alpha\alpha12}=0,\quad \alpha\in\{1,2\},\\
					a^{M,hom}_{1111}=a^{M,hom}_{2222}\quad \text{and} \quad a^{M,hom}_{\alpha\alpha12}=0,\quad \alpha\in\{1,2\}.
				\end{aligned}
$$
		\end{lemma}	
		The proof is in line with Lemma 6.9 of \cite{GOW2}.
\section{Extension to general periodic composite plate}\label{Sec09} 
		In this section, the extension of our results to a more general periodic composite plate is undertaken. The key concept to emphasize is that our results can be extended to any periodic perforated plate with holes filled by a soft matrix. This extension yields similar results as presented in Lemma \ref{ExDef1}, with the holes being bounded domains featuring Lipschitz boundaries, thereby allowing for the attainment of the estimates \eqref{BGShell3.3}.

In the following discussion, we explore cases where the general holes take the form of \eqref{GenHole} (refer to Figure \ref{Fig02}).
		\subsection{General periodic composite plate} 
		Our general domain is a periodic composite plate in which the part filled with soft matrix is connected. We set the following: 
		\begin{itemize}
			\item $Y^B_\kappa$ be an  open connected  subset of $\Yc$ with Lipschitz boundary (see for example Figure \ref{Fig02}a) such that interior $\big(\overline{Y^B_\kappa}\cup(\overline{Y^B_\kappa}+\Ge_\alpha)\big)$  is connected for $\alpha=1,2$;
			\item  $Y^M_\kappa=\Yc\setminus\overline{Y}^B_\kappa$ is such that  interior$\big(\overline{Y^M_\kappa}\cup(\overline{Y^M_\kappa}+\Ge_\alpha)\big)$ is connected for $\alpha=1,2$;
			\item  $Y^B_\kappa$, $Y^M_\kappa$ are  the reference cells of the stiff and matrix part respectively;
			\item $\Xi_\e=\{\xi\in\Z^2\X\{0\}\;|\;\e(\xi+\Yc)\cap\O_\e\neq \emptyset\}$;
%			\item $\Xi_{\e,\alpha}=\{\xi\in\Xi_\e\;|\;\Xi+\Ge_\alpha\in\o_\e\}$, $\alpha=1,2$;
			\item $\wt{\O}_\e^\xi=\e(\xi+Y^B_\kappa)$, and $\O_\e^\xi=\e(\xi+Y^M_\kappa)$ for all $\xi\in\Xi_\e$;
			\item $\O^B_\e= \bigcup_{\xi\in\Xi_\e}\wt{\O}_\e^\xi $, see for example Figure \ref{Fig02}a.
%			\item 
		\end{itemize}
		We assume that the domain $\O_\e^M$ is filled with soft matrix and is given by  $\O_\e^M=\O_\e\setminus\overline{\O_\e^B}$ . Then, we have
		\begin{equation}\label{GenHole}
			\O_\e^M=\hbox{interior} \bigcup_{\xi\in\Xi_\e}\overline{\O_\e^\xi}.
		\end{equation} 
		Here, we note the following
		\begin{itemize}
			\item 		Here, we have $\bigcup_{\alpha=1}^2\Xi_{\e,\alpha}\subset\Xi_\e$ and that the domain $\O_\e^B$ is connected and made of stiff material.
			\item The domain filled with soft matrix $\O_\e^M$ is also connected.
			\item   This structure is (like the one in Figure \ref{Fig02}a) is $3$-dimensional and only periodic in two directions, in the third direction it is thin that is, its thickness is of same order $r=\kappa\e$ as the period of the other two directions\footnote{This particular structure has been extensively investigated in the literature, specifically in the works of \cite{GOW}, \cite{GFOW}, and \cite{GOW2}. In these studies, they consider the relationship $r=\kappa\varepsilon$ with $\kappa\in(0,1/3]$ as stated in Remark A.1 of \cite{GOW}.}.
		\end{itemize}
		  \begin{figure}[htb]
		\centering
		\subfigure[]{
			\begin{tikzpicture}[remember picture,
				x={(1cm,0cm)},
				y={(0cm,1cm)},
				z={({0.5*cos(45)},{0.5*sin(45)})},
				fill opacity = 1, scale=0.58]
				
				%		\clip (-2,-1) rectangle (10,10);

				% back beam
				\begin{scope}[thick]

					%outlines + fill
					\draw[fill = gray] (0,0,5)--(0,0,6)--(0,1,6)--(0,1,5)--cycle;
					\draw[name path = frontlow]  (6,1,5)--(5,1,5) to[out=180,in=0] (1,0,5) -- (0,0,5) ;
					\draw[name path = fronthigh]  (6,2,5)--(5,2,5) to[out=180,in=0] (1,1,5) -- (0,1,5);
					\draw[name path = backlow]   (6,1,6)--(5,1,6) to[out=180,in=0] (1,0,6) -- (0,0,6) ;
					\draw[name path = backhigh] (6,2,6)--(5,2,6) to[out=180,in=0] (1,1,6) -- (0,1,6) ;
					
					\tikzfillbetween[of=frontlow and fronthigh]{ left color=gray!20!white, right color = gray, shading angle=60,};
					\tikzfillbetween[of=fronthigh and backhigh]{ left color=gray, right color = gray!20!white, shading angle=180,};
					
					%redraw lines
					%	\draw[name path = frontlow] (6,1,5)--(5,1,5) to[out=180,in=0] (1,0,5) -- (0,0,5);
					%	\draw[name path = fronthigh] (6,2,5)--(5,2,5) to[out=180,in=0] (1,1,5) -- (0,1,5);
					%	\draw[name path = backhigh] (6,2,6)--(5,2,6) to[out=180,in=0] (1,1,6) -- (0,1,6);
					\draw[name path = frontlow] (6,1,5)--(5,1,5) to[out=180,in=0] (1,0,5) -- (0,0,5) ;
					\draw[name path = fronthigh] (6,2,5)--(5,2,5) to[out=180,in=0] (1,1,5) -- (0,1,5 );
					\draw[name path = backhigh]  (6,2,6)--(5,2,6) to[out=180,in=0] (1,1,6) -- (0,1,6) ;
					
					%
					%	\draw[fill = gray] (6,1,5)--(6,1,6)--(6,2,6)--(6,2,5)--cycle;
					\draw[fill = gray] (6,0,5)--(6,0,6)--(6,1,6)--(6,1,5)--cycle;
				\end{scope}

				% left beam
				\begin{scope}[thick,%dashed
					]
					
					%outlines + fill
					\draw[fill=gray] (1,2,6)--(1,1,6)--(0,1,6)--(0,2,6)--cycle;
					\draw[name path = righthigh] (1,1,0)--(1,1,1) to[out=55,in=-135] (1,2,5)--(1,2,6) ;
					\draw[name path = rightlow] (1,0,0)--(1,0,1) to[out=45,in=-125] (1,1,5)--(1,1,6) ;
					\draw[name path = lefthigh] (0,1,0)--(0,1,1) to[out=55,in=-135] (0,2,5)--(0,2,6) ;
					\draw[name path = leftlow] (0,0,0)--(0,0,1) to[out=45,in=-125] (0,1,5)--(0,1,6) ;
					\draw[fill=gray] (0,0,0)--(0,1,0)--(1,1,0)--(1,0,0)--cycle;
					
					\tikzfillbetween[of=rightlow and righthigh]{ left color=gray!20!white, right color = gray, shading angle=60,};
					\tikzfillbetween[of=lefthigh and righthigh]{ left color=gray, right color = gray!20!white, shading angle=180,};
					
					%redraw lines
					\draw[name path = righthigh] (1,1,0)--(1,1,1) to[out=55,in=-135] (1,2,5)--(1,2,6) ;
					\draw[name path = rightlow] (1,0,0)--(1,0,1) to[out=45,in=-125] (1,1,5)--(1,1,6) ;
					\draw[name path = lefthigh] (0,1,0)--(0,1,1) to[out=55,in=-135] (0,2,5)--(0,2,6) ;
					\draw[] (0,2,6)--(1,2,6)--(1,1,6);
				\end{scope}

				%front beam
				\begin{scope}[thick,%dashed
					]
					%outlines + fill
					\draw[name path = frontlow]  (0,1,0)--(1,1,0) to[out=0,in=180] (5,0,0) -- (6,0,0) ;
					\draw[name path = fronthigh] (0,2,0)--(1,2,0) to[out=0,in=180] (5,1,0) -- (6,1,0) ;
					\draw[name path = backlow]    (0,1,1)--(1,1,1) to[out=0,in=180] (5,0,1) -- (6,0,1) ;
					\draw[name path = backhigh]   (0,2,1)--(1,2,1) to[out=0,in=180] (5,1,1) -- (6,1,1);
					
					%	\draw[fill=gray ] (0,1,0)--(0,2,0)--(0,2,1)--(0,1,1)--cycle;
					\draw[fill=gray ] (0,1,0)--(0,2,0)--(0,2,1)--(0,1,1)--cycle;
					
					\tikzfillbetween[of=frontlow and fronthigh]{ left color=gray!20!white, right color = gray, shading angle=60,};
					\tikzfillbetween[of=fronthigh and backhigh]{ left color=gray, right color = gray!20!white, shading angle=180,};

					%redraw lines			
					%	\draw[name path = frontlow] (0,1,0)--(1,1,0) to[out=0,in=180] (5,0,0) -- (6,0,0);
					%	\draw[name path = fronthigh] (0,2,0)--(1,2,0) to[out=0,in=180] (5,1,0) -- (6,1,0);
					%	\draw[name path = backhigh] (0,2,1)--(1,2,1) to[out=0,in=180] (5,1,1) -- (6,1,1);
					\draw[name path = frontlow]   (0,1,0)--(1,1,0) to[out=0,in=180] (5,0,0) -- (6,0,0) ;
					\draw[name path = fronthigh]  (0,2,0)--(1,2,0) to[out=0,in=180] (5,1,0) -- (6,1,0) ;
					\draw[name path = backhigh]  (0,2,1)--(1,2,1) to[out=0,in=180] (5,1,1) -- (6,1,1);
					%	\draw[] (0,1,0)--(0,2,0)--(0,2,1);
					
					%
					%	\draw[fill=gray] (6,0,0)--(6,1,0)--(6,1,1)--(6,0,1)--cycle;
					\draw[fill=gray] (6,0,0)--(6,1,0)--(6,1,1)--(6,0,1)--cycle;
				\end{scope}

				%right beam
				\begin{scope}[thick]
					
					\draw[] (6,0,6)--(6,1,6)-- (5,1,6)-- (5,0,6)--cycle;
					
					%outlines + fill
					\draw[name path = rightlow] (6,1,0)--(6,1,1) to[out=45,in=-145] (6,0,5) -- (6,0,6);
					\draw[name path = righthigh] (6,2,0)--(6,2,1) to[out=35,in=-145] (6,1,5) -- (6,1,6);
					\draw[name path = leftlow] (5,1,0)--(5,1,1) to[out=45,in=-135] (5,0,5) -- (5,0,6);
					\draw[name path = lefthigh] (5,2,0)--(5,2,1) to[out=35,in=-145] (5,1,5) -- (5,1,6);

					\tikzfillbetween[of=rightlow and righthigh]{ left color=gray!20!white, right color = gray, middle color=black!40!gray,};
					\tikzfillbetween[of=righthigh and lefthigh]{ top color=gray, bottom color = gray!20!white, middle color=black, shading angle=0};
					
					%redraw lines
					\draw[name path = rightlow] (6,1,0)--(6,1,1) to[out=45,in=-145] (6,0,5) -- (6,0,6);
					\draw[name path = righthigh] (6,2,0)--(6,2,1) to[out=35,in=-145] (6,1,5) -- (6,1,6);
					%	\draw[name path = leftlow] (5,1,0)--(5,1,1) to[out=45,in=-135] (5,0,5) -- (5,0,6);
					\draw[name path = lefthigh] (5,2,0)--(5,2,1) to[out=35,in=-145] (5,1,5) -- (5,1,6);

					\draw[fill=gray] (5,1,0)--(6,1,0)--(6,2,0)--(5,2,0)--cycle;
				\end{scope}

				%%clipping ding
				\begin{scope}[thick]
					%			\clip[draw=none](5,1,5)--(6,1,5)--(6,0,5)--(6,0,6)--(6,1,6)--(5,1,6)--cycle; %% for opacity <1
					\clip[draw=none](4,0,5)--(7,0,5)--(7,3,5)--(4,3,5)--cycle; %% for opacity =1
					
					\draw[fill = gray] (0,0,5)--(0,0,6)--(0,1,6)--(0,1,5)--cycle;
					\draw[name path = frontlow]  (6,1,5)--(5,1,5) to[out=180,in=0] (1,0,5) -- (0,0,5) ;
					\draw[name path = fronthigh]  (6,2,5)--(5,2,5) to[out=180,in=0] (1,1,5) -- (0,1,5);
					\draw[name path = backlow]   (6,1,6)--(5,1,6) to[out=180,in=0] (1,0,6) -- (0,0,6) ;
					\draw[name path = backhigh] (6,2,6)--(5,2,6) to[out=180,in=0] (1,1,6) -- (0,1,6) ;
					
					\tikzfillbetween[of=frontlow and fronthigh]{ left color=gray!20!white, right color = gray, shading angle=60,};
					\tikzfillbetween[of=fronthigh and backhigh]{ left color=gray, right color = gray!20!white, shading angle=180,};
					
					%redraw lines
					%	\draw[name path = frontlow] (6,1,5)--(5,1,5) to[out=180,in=0] (1,0,5) -- (0,0,5);
					%	\draw[name path = fronthigh] (6,2,5)--(5,2,5) to[out=180,in=0] (1,1,5) -- (0,1,5);
					%	\draw[name path = backhigh] (6,2,6)--(5,2,6) to[out=180,in=0] (1,1,6) -- (0,1,6);
					\draw[name path = frontlow] (6,1,5)--(5,1,5) to[out=180,in=0] (1,0,5) -- (0,0,5) ;
					\draw[name path = fronthigh] (6,2,5)--(5,2,5) to[out=180,in=0] (1,1,5) -- (0,1,5 );
					\draw[name path = backhigh]  (6,2,6)--(5,2,6) to[out=180,in=0] (1,1,6) -- (0,1,6) ;

					\draw[fill = gray] (6,1,5)--(6,1,6)--(6,2,6)--(6,2,5)--cycle;
					%	\draw[fill = gray] (6,0,5)--(6,0,6)--(6,1,6)--(6,1,5)--cycle;
					
				\end{scope}
				
				\begin{scope}
					\draw[|-|,thick] (0.5,-0.6,0) -- (5.5,-0.6,0) node [midway, below, sloped] (TextNode1) {\Large $\e/2$};
					\draw[dashed,thick] (0.5,-0.6,0) -- (0.5,0.5,0);
					\draw[dashed,thick] (5.5,-0.6,0) -- (5.5,1.5,0);
					
					\draw[|-|,thick] (0,-0.16,0) -- (1,-0.16,0) node [pos=.25,yshift=-1pt,xshift=-2pt, below, sloped] (TextNode2) {\Large $2r$};
					\draw[|-|,thick] (-0.16,0,0) -- (-0.16,1,0) node [midway, above, sloped] (TextNode3) {\Large $2r$};
					
					\draw[|-|,thick] (6,-0.5,0.5) -- (6,-0.5,5.5) node [midway, below, sloped] (TextNode4) {\Large $\e/2$};
					\draw[dashed,thick] (6,-0.5,0.5)-- (6,0.5,0.5);
					\draw[dashed,thick] (6,-0.5,5.5)-- (6,1.5,5.5);
				\end{scope}
		\end{tikzpicture}}
		\hspace{1cm}
		\subfigure[]{
			\includegraphics[scale=0.4]{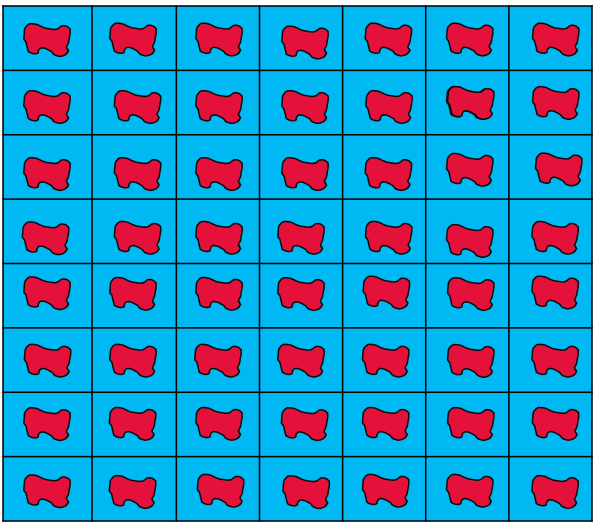}}
		\caption{(a) The domain is a quarter of the periodicity cell of the full structure. (b) $2$D periodic composite plate: Blue part is with stiff material and red part is with soft matrix.}\label{Fig02}
	\end{figure}
		Here, we have a similar extension result like in Lemma \ref{ExDef1}, which is a consequence of the extension result given in see [\cite{GOW2}, Lemma 2.1].
		\begin{proposition}
			For every deformation $v\in H^1(\O_\e^B)^3$, there exist a deformation $\Gv\in H^1(\O_\e)^3$ satisfying
			\begin{equation}\label{ExDef3}
				\begin{aligned}
					\Gv_{|\O_\e^B}=v,\quad\text{and}\quad
					\|dist(\nabla\Gv,SO(3))\|_{L^2(\O_\e)}\leq C\|dist(\nabla v,SO(3))\|_{L^2(\O_\e^B)}.
				\end{aligned}
			\end{equation}
			The constant does not depend on $\e$.
		\end{proposition}	
		Note that $\wt{\O}_\e^\xi$ and $\O_\e^\xi$ are the periodicity cell of the stiff part and the matrix part respectively. Here, we observe that $Y_\kappa^M$ (resp. $\O_\e^\xi$) is a bounded domain, but we do not assume it has a Lipschitz boundary (see Figure \ref{Fig02}a). For, this reason we proceed differently. The deformation $v$ is defined in $\O_\e$ and satisfies
		$$\|dist(\nabla v, SO(3))\|_{L^{2}(\O^B_\e)}+\e\|dist(\nabla v, SO(3))\|_{L^{2}(\O^M_\e)}\leq C\e^{5/2}\|f\|_{L^2(\o)}.$$ 
		The restriction of $v$ to $\O^M_\e$ is naturally extended by itself and satisfies
		$$\|dist(\nabla v, SO(3))\|_{L^{2}(\O_\e)}\leq C\e^{3/2}\|f\|_{L^2(\o)}.$$ Then, we can decompose $v$ since it is a deformation of the plate $\O_\e$. After passing to the limit, we can use the estimate $\|dist(\nabla v, SO(3))\|_{L^{2}(\O^B_\e)}\leq C\e^{5/2}\|f\|_{L^2(\o)}$ to show that the elementary limit of the extended deformation vanishes in $\o\X Y_\kappa^B$. Hence all our results are valid for the general composite plate with soft matrix.
	\begin{remark}
			Our results are also valid for composite plate $\O_\e$ when the part with soft matrix $\O_\e^M$ is not connected, we define such structure as
			\begin{itemize}
				\item  Let $\GC^*\subset Y$ be a bounded  domain with Lipschitz boundary such that the interior$\left(\overline{\GC}\cup(\overline{\GC}+\Ge_\alpha)\right)$  for $\alpha=1,2$ are connected  with Lipschitz boundary, where $\GC=Y\setminus\overline{\GC^*}$, for example see Figure \ref{Fig02}b;
				\item $Y_\kappa^B=\GC\X(-\kappa,\kappa)$, and $Y_\kappa^M=\Yc\setminus\overline{Y_\kappa^B}$.
			\end{itemize}
			Using $Y_\kappa^B$ and $Y_\kappa^M$ as the reference periodicity cell of the stiff part and matrix part respectively, we define $\O_\e^B$ and $\O_\e^M$. Further analysis is same as the case with connected $\O_\e^M$.  Hence all our results are valid for this type of composite plate with soft matrix. \footnote{The periodic perforated plate $\O_\e^B$ is studied in \cite{larysa1} with respect to a linear elasticity problem starting with displacements.}.
		\end{remark}
		\section{Conclusion}\label{sec10}
	     This paper focuses on the asymptotic analysis of a periodic composite plate with matrix pre-strain within the framework of non-linear elasticity. Two cases are considered: one without pre-strain and one with pre-strain. In both cases, the total elastic energy follows the von-K\'arm\'an (vK) regime. The re-scaling unfolding operator is utilized to derive the asymptotic behavior of the Green St. Venant's strain tensor. The existence of a minimizer for the limit energy is shown through $\Gamma$-convergence. The discussion of both cases aims to demonstrate that the soft matrix only has an impact on the homogenized plate when pre-strain is present. This behavior is attributed to the relative weakness of the matrix compared to the stiff part, as well as the stability of the structure (due to similar orders of thickness $r$ and periodicity $\varepsilon$). Furthermore, it is shown that the cell problems are linear elastic and the homogenized von-K\'arm\'an plate exhibits orthotropic behavior for isotropic homogenized material. The findings of this study have practical applications in the fields of 3D printing, textiles, and aerospace industries.
%	        	This paper focuses on the simultaneous homogenization and dimension reduction of periodic composite plates within the framework of non-linear elasticity. The composite plate in its reference (undeformed) configuration consists of a periodic perforated plate made of stiff material with holes filled by soft matrix material. The structure is clamped on a cylindrical part. Two cases of asymptotic analysis are considered: one without pre-strain and the other with matrix pre-strain. In both cases, the total elastic energy is in the von-K\'arm\'an (vK) regime ($\varepsilon^5$).
%	     To analyze the asymptotic behavior, a new splitting of the displacements is introduced. The displacements are decomposed using the Kirchhoff-Love (KL) plate displacement decomposition. The use of re-scaling unfolding operator allows for the derivation of the asymptotic behavior of the Green-St Venant's strain tensor in terms of displacements. The limit homogenized energy is shown to be of vK type with linear elastic cell problems, which is established using the $\Gamma$-convergence.
%	     Additionally, it is shown that for isotropic homogenized material our limit vK plate is orthotropic. The derived results have practical applications in the design and analysis of composite structures.\\
%	     
%	     
\section*{Acknowledgment}
		    The research was funded by DFG, German Research Foundation, project number OR 190/10-1 and AIF project OptiDrape, both in the collaboration with the textile institute ITA in Aachen, Germany.

%%%%%%%%%%% The bibliography starts:
% if your bibliography is in bibtex format, use those commands:
\bibliographystyle{ios1}           % Style BST file.
\bibliography{bibliography}        % Bibliography file (usually '*.bib')

\end{document}